\documentclass[twoside]{article}
\usepackage[utf8]{inputenc}
\usepackage[T2A]{fontenc}
\usepackage{array}
\usepackage{amssymb}
\usepackage{amsmath}
\usepackage{amsthm}
\usepackage{amsfonts}
\usepackage{latexsym}
\usepackage{hyperref}
\usepackage{indentfirst}
\usepackage{appendix}
\usepackage{bm}
\usepackage{enumerate}
\usepackage{graphicx}
\usepackage{epsf}
\usepackage{wrapfig}
\usepackage[mathscr]{euscript}
\usepackage{indentfirst}
\usepackage[english,russian]{babel}
\usepackage{textcomp}
\usepackage{tikz}
\newtheorem{theorem}{Theorem}
\newtheorem{lemma}{Lemma}
\newtheorem{definition}{Definition}
\newtheorem{proposition}{Proposition}

\newtheorem{corollary}{Corollary}
\newtheorem{remark}{Remark}

\newcommand{\boxto}{\ensuremath{%
		\mathrel{\Box\kern-1.5pt\raise1pt\hbox{$\mathord{\rightarrow}$}}}}
	
\newcommand{\diamondto}{\ensuremath{%
		\mathrel{\Diamond\kern-1.5pt\raise1pt\hbox{$\mathord{\rightarrow}$}}}}
\newcommand{\boxTo}{\ensuremath{%
		\mathrel{\Box\kern-1.5pt\raise1pt\hbox{$\mathord{\Rightarrow}$}}}}
\newcommand{\diamondTo}{\ensuremath{%
		\mathrel{\Diamond\kern-1.5pt\raise1pt\hbox{$\mathord{\Rightarrow}$}}}}

\def\@clipped@vdash{%
	\raise .6ex\hbox{\clipbox{0pt .6ex 0pt .6ex}{$\vdash$}}%
}

\begin{document}
{\selectlanguage{english}
\binoppenalty = 10000 %
\relpenalty   = 10000 %

\pagestyle{headings} \makeatletter
\renewcommand{\@evenhead}{\raisebox{0pt}[\headheight][0pt]{\vbox{\hbox to\textwidth{\thepage\hfill \strut {\small Grigory K. Olkhovikov}}\hrule}}}
\renewcommand{\@oddhead}{\raisebox{0pt}[\headheight][0pt]{\vbox{\hbox to\textwidth{{Nelsonian conditionals and their first-order embeddings}\hfill \strut\thepage}\hrule}}}
\makeatother

\title{Conditional reasoning and the shadows it casts onto the first-order logic: the Nelsonian case}
\author{Grigory K. Olkhovikov\\ Department of Philosophy I\\ Ruhr University Bochum\\
email: grigory.olkhovikov@\{rub.de, gmail.com\}}
\date{}
\maketitle
\begin{quote}
{\bf Abstract.} We define a natural notion of standard translation for the formulas of conditional logic which is analogous to the standard translation of modal formulas into the first-order logic. We briefly show that this translation works (modulo a lightweight first-order encoding of the conditional models) for the minimal classical conditional logic $\mathsf{CK}$ introduced by Brian Chellas in \cite{chellas}; however, the main result of the paper is that a classically equivalent reformulation of these notions (i.e. of standard translation plus theory of conditional models) also faithfully embeds the basic Nelsonian conditional logic $\mathsf{N4CK}$, introduced in \cite{nelsonian} into $\mathsf{QN4}$, the paraconsistent variant of Nelson's first-order logic of strong negation. Thus $\mathsf{N4CK}$ is the logic induced by the Nelsonian reading of the classical Chellas semantics of conditionals and can, therefore, be considered a faithful analogue of $\mathsf{CK}$ on the non-classical basis provided by the propositional fragment of $\mathsf{QN4}$. Moreover, the methods used to prove our main result can be easily adapted to the case of modal logic, which allows to improve an older result \cite[Proposition 7]{odintsovwansing} by S. Odintsov and H. Wansing about the standard translation embedding of the Nelsonian modal logic $\mathsf{FSK}^d$ into $\mathsf{QN4}$. 
\end{quote}
\begin{quote}{\bf Keywords.} conditional logic, strong negation, paraconsistent logic, modal logic, first-order logic, constructive logic 
\end{quote}

\section{Introduction}\label{S:intro}
The present paper is a study of the relation between the relatively new system of conditional logic $\mathsf{N4CK}$, introduced recently in \cite{nelsonian} by the author and the paraconsistent version $\mathsf{QN4}$ of Nelson's logic of strong negation.\footnote{The only difference between $\mathsf{QN4}$ and the original version $\mathsf{QN3}$ of Nelson's logic of strong negation (see \cite{nelson}), is that in $\mathsf{QN4}$ the extensions and the anti-extensions of predicates are no longer required to be disjoint.} The main result of the paper says that a natural notion of the standard first-order translation of conditional formulas provides a faithful embedding of $\mathsf{N4CK}$ into $\mathsf{QN4}$ modulo the assumption of a certain first-order theory encoding the notion of a conditional model. 

Since this embedding is, in a sense, the same embedding that obtains in the case of $\mathsf{CK}$, the minimal classical conditional logic, relative to the classical first-order logic $\mathsf{QCL}$, one can view this result as showing that $\mathsf{N4CK}$ as conditional logic is the same logic as $\mathsf{CK}$ only \textit{read non-classically}, that is to say, read in terms of the Nelson's logic of strong negation rather than classical logic. In this capacity, $\mathsf{N4CK}$ can be viewed as a natural candidate for the role of the minimal normal conditional logic extending $\mathsf{N4}$, the propositional fragment of $\mathsf{QN4}$.

One can better appreciate the true meaning of this result if one views it as the final piece in the mosaic of results relating the classical, the intuitionistic and the Nelsonian modal and conditional logics to their corresponding fragments of first-order reasoning by way of standard translation embeddings --- but also to one another. In order to supply this richer context, we have to do quite a bit of preliminary work before we get to the main proof, if this paper is to be reasonably self-contained. Our strive towards this goal explains most of our choices related both to the structure of the paper and to its length. We did our best to compensate for the latter shortcoming by making our explanations as lucid and easy to follow as possible. The more technical and tedious parts of our reasoning are systematically shifted to numerous appendices to be found at the end of the paper.

The rest of the paper is organized as follows. Section \ref{S:Prel} introduces the notational preliminaries, after which Section \ref{S:logics} defines a notion of logic which is wide enough to cover every system to be mentioned below. We then proceed to introduce three first-order logics in Section \ref{S:fo}, namely $\mathsf{QCL}$, $\mathsf{QN4}$, and $\mathsf{QIL}^+$, the positive fragment of the first-order intuitionistic logic  $\mathsf{QIL}$. The main work in this section is directed towards familiarizing the reader with some elementary results about $\mathsf{QN4}$ which is the least known of the three logics. To facilitate our main proof, we also need to define a novel sheaf semantics for $\mathsf{QN4}$ and we show its adequacy for the logic.

The propositional fragments of the logics dealt with in Section \ref{S:fo}  will not be introduced separately but will in each case be referred to by omitting the initial $\mathsf{Q}$ in the name of the corresponding logic, so that we will be mentioning $\mathsf{CL}$, $\mathsf{N4}$, $\mathsf{IL}^+$, and $\mathsf{IL}$ below without further explanation. All of the propositional fragments of the first-order logics are assumed to be given over the set $Prop$ to be introduced in Section \ref{S:idea}.

The main task of the latter section, however, is to provide the context for the main result of the paper by recalling several classical definitions and results about the standard translation of modal logics into first-order logics. This is followed by Section \ref{S:conditional} where we introduce the classical conditional logic $\mathsf{CK}$ and show how its properties can be viewed as a natural continuation of the properties of the modal logics laid out in the previous section. After that, Section \ref{S:N4CK} introduces the main system of the paper, $\mathsf{N4CK}$, and proves our main result. The proof procedure explained in Section \ref{S:N4CK} has the additional merit that it allows to improve on some known results about the standard translation embeddings of $\mathsf{N4}$-based modal logics. Finally, Section \ref{S:conclusion} sums up the broader meaning of the results obtained as well as offers concluding remarks; we also chart several avenues for further research.

\section{Preliminaries}\label{S:Prel}
We use this section to fix some notations to be used throughout the paper.

We will use IH as the abbreviation for Induction Hypothesis in the inductive proofs, and we will write $\alpha:=\beta$ to mean that we define $\alpha$ as $\beta$. We will use the usual notations for sets and functions. The set of all subsets of $X$ will be denoted by $\mathcal{P}(X)$. The natural numbers are undrestood as the finite von Neumann ordinals, and $\omega$ as the smallest infinite ordinal. Given an $n \in \omega$ and a tuple $\alpha = (x_1,\ldots,x_n)$ of any sort of objects, we will refer to $\alpha$ by $\bar{x}_n$ and will denote by $init(\alpha)$ and $end(\alpha)$ the initial and final element of $\alpha$, that is to say, $x_1$ and $x_n$, respectively. More generally, given any $i < \omega$ such that $1 \leq i \leq n$, we set that $\pi^i(\alpha):= x_i$, i.e. that $\pi^i(\alpha)$ denotes the $i$-th projection of $\alpha$. Given another tuple $\beta = \bar{y}_m$, we will denote by $(\alpha)^\frown(\beta)$ the concatenation of the two tuples, i.e. the tuple $(x_1,\ldots,x_n,y_1,\ldots,y_m)$. The empty tuple will be denoted by $\Lambda$.

We will extensively use ordered couples of sets which we will call \textit{bi-sets}. Relations are understood as sets of ordered tuples. Given binary relations $R \subseteq X \times Y$ and $S\subseteq Y\times Z$, we denote their composition by $R\circ S:= \{(a,c)\mid\text{ for some }b \in Y,\,(a, b)\in R,\,(b,c)\in S\}$.

We view functions as relations with special properties; we write $f:X \to Y$ to denote a function $f \subseteq X\times Y$ such that its left projection is all of $X$. If  $f: X \to Y$ and $Z\subseteq X$ then we will denote the image of $Z$ under $f$ by $f(Z)$. In view of our previous convention for relations, for any given two functions $f:X\to Y$ and $g:Y\to Z$, we will denote the function $x\mapsto g(f(x))$ by $f\circ g$, even though, in the existing literature, this function is often denoted by $g\circ f$ instead.

Given a set $X$, we will denote by $id[X]$ the identity function on $X$, i.e. the function $f:X\to X$ such that $f(x) = x$ for every $x \in X$. In relation to compositions, functions of the form $id[X]$ have a special importance as a limiting case. More precisely, given a set $X$ and a family $F$ of functions from $X$ to $X$, we will assume that the composition of the empty tuple of functions from $F$ is just $id[X]$.  

Furthermore, if $f:X\to Y$ is any function, $x \in X$ and $y \in Y$, we will denote by $f[x/y]$ the unique function $g:X\to Y$ such that, for a given $z \in X$ we have:
$$
g(z):=\begin{cases}
	y,\text{ if }z = x;\\
	f(z),\text{ otherwise. }
\end{cases}
$$

\section{Logics}\label{S:logics}
We are going to give a notion of logic that is wide enough to cover every formal system to be considered below, without aspiring to give any sort of ultimate generalization of this notion. 

Logics are based on languages, and in this paper we confine ourselves to considering languages of two types: the propositional language with an added conditional or modal operator(s) and the first-order relatinal languages with equality. Speaking generally, a \textit{language} $\mathcal{L}$ is simply a certain set.  The elements of languages are called their \textit{formulas}. The languages are often generated from certain sets of atoms by repeated application of connectives and quantifiers. Every language considered in this paper will include at least the binary connectives $\to$, $\wedge$, and $\vee$. In this paper, we will treat the sets of formulas generated by different sets of atoms over the same set of logical symbols as different languages rather than different versions of the same language.

Although one and the same logic can be formulated over different versions of the same language (or, in our terminology, over different languages), in this paper we will abstract away from such subtleties, and will simply treat a logic as a consequence relation, or, in other words, as a set of \textit{consecutions} over a particular language  $\mathcal{L}$. More precisely, a logic is a set $\mathsf{L} \subseteq \mathcal{P}(\mathcal{L})\times\mathcal{P}(\mathcal{L})$, for some language $\mathcal{L}$, where $(\Gamma, \Delta)\in \mathsf{L}$ iff $\Delta$ $\mathsf{L}$-follows from $\Gamma$ (we will also denote this by $\Gamma\models_\mathsf{L}\Delta$). We will say that $(\Gamma, \Delta)$ is $\mathsf{L}$-satisfiable iff $(\Gamma, \Delta)\notin \mathsf{L}$. Given a $\phi \in \mathcal{L}$, $\phi$ is $\mathsf{L}$-valid or a theorem of $\mathsf{L}$ (we will also write $\phi \in \mathsf{L}$) iff $(\emptyset, \{\phi\}) \in \mathsf{L}$.

Every logic considered in this paper will be introduced either by its intended semantics, or by a complete Hilbert-style axiomatization; for most logics in this paper, we will mention both.

The semantics of logics is going to be laid out according to the following general scheme. Recall that, given a classically flavored logic $\mathsf{L}$, we typically define $\mathsf{L}$ by setting that $(\Gamma, \Delta)\in \mathsf{L}$ iff the truth of every $\phi\in \Gamma$ implies the truth of some $\psi \in \Delta$. In the more general setting of our paper, if $\mathcal{L}$ is a language, then a \textit{semantics} for $\mathcal{L}$ is a pair $\sigma = (EP_\sigma, \models_\sigma)$, where $EP_\sigma$ is a (definable) class called the class of $\sigma$-\textit{evaluation points}. The other element of the semantics, is a (class-)relation $\models_\sigma\subseteq EP_\sigma\times\mathcal{L}$ called the ($\sigma$-)satisfaction relation.

In case $(pt, \phi) \in \models_\sigma$, we write $pt\models_\sigma\phi$, and say that $pt$ $\sigma$\textit{-satisfies} $\phi$. More generally, given any $\Gamma, \Delta \subseteq \mathcal{L}$, and a $pt \in EP_\sigma$, we say that $pt$ $\sigma$\textit{-satisfies} $(\Gamma, \Delta)$  and write $pt\models_\sigma(\Gamma, \Delta)$ iff:
$$
(\forall \phi \in \Gamma)(pt\models_\sigma\phi)\text{ and }(\forall \psi \in \Delta)(pt\not\models_\sigma\psi).
$$
More conventionally, we say that $pt$ satisfies $\Gamma$, and write $pt\models_\sigma\Gamma$, iff $pt\models_\sigma(\Gamma, \emptyset)$. This latter format is in fact sufficient to set up a semantics as long as our logics can express Boolean negation. However, this is not the case for many logics to be considered in this paper. For such logics the ``double-entry'' format for satisfaction relation proves to be more convenient and flexible.

Next, given any $\Gamma, \Delta \cup \{\phi\} \subseteq \mathcal{L}$, we say that $(\Gamma, \Delta)$ (resp. $\Gamma$, $\phi$) is $\sigma$-satisfiable iff, for some $pt \in EP_\sigma$, $pt$ satisfies $(\Gamma, \Delta)$ (resp. $\Gamma$, $\phi$). Finally, $\Delta$ $\sigma$-follows from $\Gamma$ (written $\Gamma \models_\sigma \Delta$) iff $(\Gamma, \Delta)$ is $\sigma$-unsatisfiable, in other words, if every $\sigma$-evaluation point satisfying every formula from $\Gamma$, also satisfies at least one formula from $\Delta$.

We now say that, for a given language $\mathcal{L}$, a semantics $\sigma$ over $\mathcal{L}$ induces the logic $\mathsf{L}$ over $\mathcal{L}$ and write $\mathsf{L} = \mathbb{L}(\sigma)$ iff for all $\Gamma, \Delta \subseteq \mathcal{L}$, we have $(\Gamma, \Delta) \in \mathsf{L}$ iff $(\Gamma, \Delta)$ is $\sigma$-unsatisfiable; in other words, $\mathsf{L} = \mathbb{L}(\sigma)$ means that, for all $\Gamma, \Delta \subseteq \mathcal{L}$, $\Delta$ $\mathsf{L}$-follows from $\Gamma$ iff $\Delta$ $\sigma$-follows from $\Gamma$. In case $\sigma$ is also used to introduce  $\mathsf{L}$ by definition, we write $\mathsf{L} := \mathbb{L}(\sigma)$.

As for the Hilbert-style systems, all of them will be given by a finite number of axiomatic schemas $\alpha_1,\ldots,\alpha_n$ augmented with a finite number of inference rules $\rho_1,\ldots,\rho_m$, so the most general format sufficient for the present paper is $\Sigma(\bar{\alpha}_n; \bar{\rho}_m)$. Every Hilbert-style systems considered below, happens to extend a certain minimal system which we will denote by $\mathfrak{IL}^+$.\footnote{In fact, $\mathfrak{IL}^+$ is the standard axiomatization of  $\mathsf{IL}^+$.} We have $\mathfrak{IL}^+:= \Sigma(\alpha_1-\alpha_8;\eqref{E:mp})$, where:
\begin{align*}
	&\phi \to(\psi\to\phi)\,(\alpha_1),\quad(\phi\to(\psi\to\chi))\to((\phi\to\psi)\to(\phi\to\chi))\,(\alpha_2),\\
	&\quad(\phi\wedge\psi)\to\phi\,(\alpha_3),\quad(\phi\wedge\psi)\to\psi\,(\alpha_4),\quad \phi\to(\psi\to (\phi\wedge\psi))\,(\alpha_5),\\
	&\quad\phi\to(\phi\vee \psi)\,(\alpha_6),\quad\psi\to(\phi\vee \psi)\,(\alpha_7),\quad (\phi\to\chi)\to((\psi \to \chi)\to ((\phi\vee\psi)\to \chi))\,(\alpha_8)
\end{align*}
and:
\begin{align}
	\text{From }\phi, \phi \to \psi&\text{ infer }\psi\label{E:mp}\tag{MP}
\end{align}
It is therefore important for our purposes to be able to refer to Hilbert-style systems as extensions of other systems. If $\mathfrak{Ax} = \Sigma(\bar{\alpha}_n; \bar{\rho}_m)$, and $\beta_1,\ldots,\beta_k$ are some new axiomatic schemes and $\sigma_1,\ldots,\sigma_r$ are some new rules, then we will write $\mathfrak{Ax}+(\bar{\beta}_k;\bar{\sigma}_r)$ to denote the system $\Sigma((\bar{\alpha}_n)^\frown(\bar{\beta}_k);(\bar{\rho}_m)^\frown(\bar{\sigma}_r))$.

Axiomatic systems can be viewed as operators generating logics when applied to languages. More precisely, if $\mathfrak{Ax} = \Sigma(\bar{\alpha}_n; \bar{\rho}_m)$ and $\mathcal{L}$ is a language, then $\mathsf{L} = \mathfrak{Ax}(\mathcal{L})$ can be described as follows. We say that a $\phi \in \mathcal{L}$ is \textit{provable} in $\mathfrak{Ax}(\mathcal{L})$ iff there exists a finite sequence $\bar{\psi}_k\in\mathcal{L}^k$ such that every formula in this sequence is either a substitution instance of one of $\bar{\alpha}_n$ or results from an application of one of $\bar{\rho}_m$ to some earlier formulas in the sequence and $\psi_k = \phi$; we will say that $(\Gamma, \Delta) \in \mathfrak{Ax}(\mathcal{L})$ iff $(\Gamma, \Delta) \in \mathcal{P}(\mathcal{L})\times\mathcal{P}(\mathcal{L})$ and there exists a sequence $\bar{\chi}_r\in \mathcal{L}^r$ such that every formula in it is either in $\Gamma$, or is provable in $\mathfrak{Ax}(\mathcal{L})$ or results from an application of \eqref{E:mp} to a pair of earlier formulas in the sequence, and, for some $\theta_1,\ldots,\theta_s\in \Delta$ we have $\chi_r = \theta_1\vee\ldots\vee\theta_s$. This definition makes sense in the context of our paper, since every language that we are going to consider contains $\vee$, and every axiomatic system that we are going to consider contains \eqref{E:mp}
 We will also express the fact that  $(\Gamma, \Delta) \in \mathfrak{Ax}(\mathcal{L})$ by writing $\Gamma\vdash_{\mathfrak{Ax}(\mathcal{L})}\Delta$. If also $\mathsf{L} = \mathfrak{Ax}(\mathcal{L})$, then we can write $\Gamma\vdash_\mathsf{L}\Delta$ instead of $\Gamma\vdash_{\mathfrak{Ax}(\mathcal{L})}\Delta$. In the latter case we will also have, for every $\phi \in \mathcal{L}$, that $\phi \in \mathsf{L}$ iff $\vdash_{\mathfrak{Ax}(\mathcal{L})}\phi$ iff $\phi$ is provable in $\mathfrak{Ax}(\mathcal{L})$.

\section{The first-order languages and their logics}\label{S:fo}
We start by defining a handful of first-order languages (relational with equality) according to the scheme laid out in the previous section. First, we let $\Pi$ denote the set $\{p^1_n\mid n\in \omega\}\cup \{S^1, O^1, E^2, R^3\}$. In case $\Omega \subseteq \Pi$, we set $\Omega^\pm:=  \{(P_+)^n, (P_-)^n\mid P^n\in \Omega\}$. Next, we define $Sign:= \Pi \cup \Pi^\pm \cup \{\epsilon^2\}$. The elements of $Sign$ will serve as predicate letters; in order to define the first-order atoms, we also need to supply the set $Ind:= \{v_n\mid n \in \omega\}$ of individual variables. 

If now $\Omega \subseteq Sign$, then the set $At(\Omega):= \{x \equiv y, Q(\bar{x}_n)\mid Q^n\in \Omega,\,x,y,\bar{x}_n \in Ind\}$ is called the set of $\Omega$-atoms. The set $Lit(\Omega):= At(\Omega)\cup \{\sim \phi\mid\phi\in At(\Omega)\}$ is called the set of $\Omega$-literals. The \textit{first-order language} $\mathcal{FO}(\Omega)$ is then generated on the basis of $At(\Omega)$ by the following BNF (where $x \in Ind$):
$$
\phi::= At(\Omega)\mid \phi\wedge\phi\mid\phi\vee\phi\mid\phi\to\phi\mid\sim\phi\mid\forall x\phi\mid\exists x\phi.
$$
The \textit{positive first-order language} $\mathcal{FO}^+(\Omega)$ is the ($\sim$)-free subset of $\mathcal{FO}(\Omega)$.

As is usual, the \textit{equivalence} $\phi\leftrightarrow\psi$ is the abbreviation for $(\phi \to \psi)\wedge(\psi\to \phi)$.

Given an $\Omega \subseteq Sign$ and a formula $\phi \in \mathcal{FO}(\Omega)$, we denote by $Sub(\phi)$ the set of subformulas of $\phi$ assuming its standard definition by induction on the construction of $\phi$. Furthermore, we can inductively define for $\phi$ its set of \textit{free} variables in a standard way (see, e.g. \cite[p. 64]{vandalen}). This set, denoted by $FV(\phi)$, is always finite. Given an $n \in \omega$, and a $\bar{x}_n \in Ind^n$, we will denote by $\mathcal{FO}(\Omega)^{\bar{x}_n}$ the set $\{\phi \in  \mathcal{FO}(\Omega)\mid FV(\phi) \subseteq \{\bar{x}_n\}\}$. If $\phi \in \mathcal{FO}(\Omega)^\emptyset$, then $\phi$ is called a $\Omega$-\textit{sentence}. Finally, given some $x,y \in Ind$, we assume a standard definition for the property of $y$ being substitutable for $x$ in $\phi$; in case this property holds, we define the result $\phi[x/y]$ of this substitution simply as the result of replacing all free occurrences of $x$ in $\phi$ with the occurrences of $y$.


We are going to define three first-order logics, the classical logic $\mathsf{QCL}$, the positive intuitionistic logic $\mathsf{QIL}^+$, and, finally, the logic $\mathsf{QN4}$, which represents the paraconsistent variant of Nelson's first-order logic of strong negation. Each of these logics will be defined over a different set of the first-order language variants. Among the three logics, the classical logic has the simplest semantics. We describe it as follows:
\begin{definition}\label{D:classical-fo-model}
	Given an $\Omega \subseteq Sign$, a classical first-order model over $\Omega$ (also called classical first-order $\Omega$-model) is a tuple $\mathcal{M} = (U^\mathcal{M}, \{P^\mathcal{M}\mid P^n \in \Omega,\,n \in \omega\})$, where $U \neq \emptyset$ is called the domain of $\mathcal{M}$ and, for every $P^n \in \Omega$, $P^\mathcal{M} \subseteq U^n$. The class of all classical first-order $\Omega$-models will be denoted by $\mathbb{C}(\Omega)$.  
\end{definition}
We will also need the following notion relating classical models:
\begin{definition}\label{D:homomorphism}
	Let $\Omega \subseteq Sign$, and let $\mathcal{M}, \mathcal{N} \in \mathbb{C}(\Omega)$. Then a function $f:U^\mathcal{M}\to U^\mathcal{N}$ is called a \textit{homomorphism from} $\mathcal{M}$ \textit{to} $\mathcal{N}$ (written $f:\mathcal{M}\to \mathcal{N}$) iff $\bar{a}_n \in P^\mathcal{M}$ implies $f(\bar{a}_n) \in P^\mathcal{N}$
	for every $n \in \omega$, every $\bar{a}_n \in (U^\mathcal{M})^n$, and every $P^n\in \Omega$. The set of all homomorphisms from $\mathcal{M}$ to $\mathcal{N}$ will be denoted by $Hom(\mathcal{M},\mathcal{N})$. 
%
\end{definition}
The class $EP_c(\Omega)$ of \textit{classical} $\Omega$\textit{-evaluation points} is then the class 
$$
\{(\mathcal{M}, f)\mid \mathcal{M}\in \mathbb{C}(\Omega),\,f:Ind\to U^\mathcal{M}\}.
$$

We now define $\mathsf{QCL}(\Omega):=\mathbb{L}(EP_c(\Omega), \models_c)$, where $\models_c$ is the classical first-order satisfaction relation. We will often write $\mathcal{M}\models_{c}\phi[f]$ instead of $(\mathcal{M}, f)\models_{c}\phi$ (also for the non-classical first-order logics).
The relation itself is defined by the following induction on the construction of $\phi$:
\begin{align*}
	\mathcal{M}\models_{c}P(\bar{x}_n)[f] &\text{ iff } f(\bar{x}_n)\in P^\mathcal{M} &&P^n\in \Omega\\
	\mathcal{M}\models_{c}(x\equiv y)[f] &\text{ iff } f(x) = f(y)\\
	\mathcal{M}\models_{c}(\psi\wedge\chi)[f] &\text{ iff } \mathcal{M}\models_{c}\psi[f]\text{ and }\mathcal{M}\models_{c}\chi[f]\\
	\mathcal{M}\models_{c}(\psi\vee\chi)[f] &\text{ iff } \mathcal{M}\models_{c}\psi[f]\text{ or }\mathcal{M}\models_{c}\chi[f]\\
	\mathcal{M}\models_{c}(\psi\to\chi)[f] &\text{ iff } \mathcal{M}\not\models_{c}\psi[f]\text{ or }\mathcal{M}\models_{c}\chi[f]\\
	\mathcal{M}\models_{c}\sim\psi[f] &\text{ iff } \mathcal{M}\not\models_{c}\psi[f]\\
\mathcal{M}\models_{c}(\exists x\psi)[f] &\text{ iff } (\exists a\in U^\mathcal{M})(\mathcal{M}\models_{c}\psi[f[x/a]])\\
	\mathcal{M}\models_{c}(\forall x\psi)[f] &\text{ iff } (\forall a\in U^\mathcal{M})(\mathcal{M}\models_{c}\psi[f[x/a]])
\end{align*}
Our next logic is the positivie intuitionistic logic $\mathsf{QIL}^+$. Its semantics (defined here for every $\Omega\subseteq Sign$ over the language $\mathcal{FO}^+(\Omega)$) is somewhat more involved, and exists in several variants. By far the most popular one is the so-called \textit{Kripke semantics}, see, e.g. \cite[Ch. 3]{gss}. However, the proof of our main result proceeds more conveniently on the basis of a somewhat involved semantics of Kripke sheaves which we define next.
\begin{definition}\label{D:intuitionistic-sheaf}
	Given an $\Omega\subseteq Sign$, an intuitionistic Kripke $\Omega$-sheaf is any structure of the form $\mathcal{S} = (W, \leq, \mathtt{M}, \mathtt{H})$, such that:
	\begin{enumerate}
		\item $W \neq \emptyset$ is the set of worlds, or nodes.
		
		\item $\leq$ is a reflexive and transitive relation (also called a preorder) on $W$.
		
		\item $\mathtt{M}:W\to \mathbb{C}(\Omega)$.
		
		\item $\mathtt{H}:\{(\mathbf{w},\mathbf{v})\in W^2\mid \mathbf{w}\leq \mathbf{v}\}\to Hom(\mathtt{M}(\mathbf{w}), \mathtt{M}(\mathbf{v}))$ such that the following holds:
		\begin{enumerate}
			\item $\mathtt{H}(\mathbf{w},\mathbf{w}) = id[U^{\mathtt{M}(\mathbf{w})}]$ for every $\mathbf{w} \in W$.
			
			\item $\mathtt{H}(\mathbf{w},\mathbf{v})\circ \mathtt{H}(\mathbf{v},\mathbf{u}) = \mathtt{H}(\mathbf{w},\mathbf{u})$ for all $\mathbf{w},\mathbf{v},\mathbf{u} \in W$ such that $\mathbf{w} \leq \mathbf{v} \leq \mathbf{u}$.
		\end{enumerate} 
	\end{enumerate}
	We will often write $\mathtt{M}_\mathbf{w}$, $\mathtt{H}_{\mathbf{w}\mathbf{v}}$ in place of $\mathtt{M}(\mathbf{w})$, $\mathtt{H}(\mathbf{w},\mathbf{v})$, respectively.
	
	The class of all intuitionistic Kripke $\Omega$-sheaves will be denoted by $\mathbb{I}(\Omega)$.
\end{definition}
\begin{remark}\label{Rm:model-notation}
In this paper we will introduce multiple notions of model-like structure; in every case, we will assume that the default representation of the structure is given in the definition and that all decorations applied to the default notation for a structure of a given sort are also inherited by the elements of its default structure, unless explicitly stated otherwise. For example, in the case of intuitionistic sheaves this means that every $\mathcal{S}\in \mathbb{I}(\Omega)$ is given as  $(W, \leq, \mathtt{M}, \mathtt{H})$ and that $\mathcal{S}_n, \mathcal{S}' \in \mathbb{I}(\Omega)$ are always given as  $(W_n, \leq_n, \mathtt{M}_n, \mathtt{H}_n)$ and  $(W', \leq', \mathtt{M}', \mathtt{H}')$, respectively, unless explicitly stated otherwise. 
\end{remark}
For any $\Omega\subseteq Sign$, the class $EP_{i}(\Omega)$ of \textit{intuitionistic evaluation points} is the class
$$
\{(\mathcal{S}, \mathbf{w}, f)\mid \mathcal{S}\in \mathbb{I}(\Omega),\,\mathbf{w} \in W,\,f:Ind \to U^{\mathtt{M}_\mathbf{w}}\}.
$$
The satisfaction relation $\models_{i}$ is then defined by the following induction:
\begin{align*}
	\mathcal{S},\mathbf{w}\models_{i}P(\bar{x}_n)[f] &\text{ iff } \mathtt{M}_\mathbf{w}\models_{c}P(\bar{x}_n)[f] \text{ iff } f(\bar{x}_n)\in P^{\mathtt{M}_\mathbf{w}} &&P^n\in \Omega\\
	\mathcal{S},\mathbf{w}\models_ix \equiv y  &\text{ iff } f(x) = f(y)\\
	\mathcal{S},\mathbf{w}\models_i(\psi\wedge\chi)[f] &\text{ iff } \mathcal{S},\mathbf{w}\models_i\psi[f]\text{ and }\mathcal{S},\mathbf{w}\models_i\chi[f]\\
	\mathcal{S},\mathbf{w}\models_i(\psi\vee\chi)[f] &\text{ iff } \mathcal{S},\mathbf{w}\models_i\psi[f]\text{ or }\mathcal{S},\mathbf{w}\models_i\chi[f]\\
	\mathcal{S},\mathbf{w}\models_i(\psi\to\chi)[f] &\text{ iff } (\forall \mathbf{v}\geq \mathbf{w})(\mathcal{S},\mathbf{v}\not\models_i\psi[f\circ\mathtt{H}_{\mathbf{w}\mathbf{v}}]\text{ or }\mathcal{S}, \mathbf{v}\models_i\chi[f\circ\mathtt{H}_{\mathbf{w}\mathbf{v}}])\\
	\mathcal{S},\mathbf{w}\models_i(\exists x\psi)[f] &\text{ iff } (\exists a\in U^{\mathtt{M}_\mathbf{w}})(\mathcal{S},\mathbf{w}\models_i\psi[f[x/a]])\\
	\mathcal{S},\mathbf{w}\models_i(\forall x\psi)[f] &\text{ iff } (\forall \mathbf{v}\geq \mathbf{w})(\forall a\in U^{\mathtt{M}_\mathbf{v}})(\mathcal{S},\mathbf{v}\models_i\psi[(f\circ\mathtt{H}_{\mathbf{w}\mathbf{v}})[x/a]])
\end{align*}
We now define $\mathsf{QIL}^+(\Omega) := \mathbb{L}(EP_{i}(\Omega), \models_i)$ for any $\Omega\subseteq Sign$. See \cite[Section 3.6 ff]{gss} for a proof that we indeed get a correct semantics for the positive intuitionistic logic in this way, the fact is also mentioned in \cite[Cor. 5.3.16]{vandalen}.

Sheaf semantics is a generalization of Kripke semantics in that the latter can be obtained from sheaf semantics as long as we assume in Definition \ref{D:intuitionistic-sheaf} that $\mathtt{H}_{\mathbf{w}\mathbf{v}} = id[U^{\mathtt{M}_\mathbf{w}}]$ for all $\mathbf{w},\mathbf{v} \in W$ such that $\mathbf{w}\leq \mathbf{v}$. Therefore, most of the properties and constructions available in the usual Kripke semantics have obvious counterparts in the sheaf semantics. The following lemma mentions some of these properties:
\begin{lemma}\label{L:intuitionistic-standard}
	For every $\Omega \subseteq Sign$, $(\mathcal{S}, \mathbf{w}, f)\in EP_{i}(\Omega)$, and $\phi \in \mathcal{FO}^+(\Omega)$, we have:
	\begin{enumerate}
		\item If $\mathbf{v}\geq \mathbf{w}$,  and $\mathcal{S},\mathbf{w}\models_{i}\phi[f]$, then $\mathcal{S},\mathbf{v}\models_{i}\phi[f\circ\mathtt{H}_{\mathbf{w}\mathbf{v}}]$.
		
		\item The \textit{generated sub-sheaf} $\mathcal{S}|_\mathbf{w} = (W|_\mathbf{w}, \leq|_\mathbf{w}, \mathtt{M}|_\mathbf{w}, \mathtt{H}|_\mathbf{w})\in \mathbb{I}(\Omega)$ is defined by $W|_\mathbf{w}:= \{\mathbf{v}\in W\mid \mathbf{v}\geq \mathbf{w}\}$, $\leq|_\mathbf{w}:= \leq \cap (W|_\mathbf{w}\times W|_\mathbf{w})$, $\mathtt{M}|_\mathbf{w}:= \mathtt{M}\upharpoonright(W|_\mathbf{w})$, and $\mathtt{H}|_\mathbf{w}:= \mathtt{H}\upharpoonright(W|_\mathbf{w}\times W|_\mathbf{w})$. With this definition, we get 
		$$
		\mathcal{S}|_\mathbf{w},\mathbf{v}\models_{i}\phi[f]\text{ iff } \mathcal{S},\mathbf{v}\models_{i}\phi[f]
		$$ 
		for every $\mathbf{v}\geq \mathbf{w}$ and every $f:Ind\to U^{\mathtt{M}_\mathbf{v}} =  U^{(\mathtt{M}|_\mathbf{w})_\mathbf{v}}$.
	\end{enumerate}
\end{lemma}
\begin{proof}
	We proceed by a straightforward induction on the construction of $\phi \in \mathcal{FO}^+(\Omega)$ in both cases, the reasoning is quite similar to the case of intuitionistic Kripke semantics.
\end{proof}

Alternatively, one can define $\mathsf{QIL}^+$ by its complete Hilbert-style axiomatization. More precisely, consider the following set of axiomatic schemes:
\begin{align}
	\forall x\phi&\to \phi[x/y]\label{Ax:9}\tag{$\alpha_9$}\\
	\phi[x/y]&\to \exists x\phi\label{Ax:10}\tag{$\alpha_{10}$}\\
	&x \equiv x\label{Ax:11}\tag{$\alpha_{11}$}\\
	y\equiv z &\to (\phi[x/y]\to\phi[x/z])\label{Ax:12}\tag{$\alpha_{12}$}
\end{align}
plus the following rules of inference:
\begin{align}
	\text{From }\psi\to\phi[x/y]&\text{ infer }\psi \to \forall x\phi\label{E:Rall}\tag{R$\forall$}\\
	\text{From }\phi[x/y]\to\psi&\text{ infer }\exists x\phi \to \psi\label{E:Rex}\tag{R$\exists$}	
\end{align}
where $x, z \in Ind$ and $y \in Ind\setminus FV(\psi)$  are such that $z, y$ are substitutable for $x$ in $\phi$. We then let $\mathfrak{QIL}^+:= \mathfrak{IL}^++(\eqref{Ax:9},\ldots,\eqref{Ax:12};\eqref{E:Rall},\eqref{E:Rex})$. It is well-known that for every $\Omega \subseteq Sign$, we have $\mathsf{QIL}^+(\Omega) = \mathfrak{QIL}^+(\mathcal{FO}(\Omega))$.

We now turn to the paraconsistent variant of Nelson's logic of strong negation which we denote by $\mathsf{QN4}$ and which we only define over $\mathcal{FO}(\Omega)$ for $\Omega \subseteq \Pi$. Our main goal in this section is to set up a sheaf semantics also for $\mathsf{QN4}$. Since no variants of sheaf semantics were yet (to the best of our knowledge) proposed for this logic in the existing literature, we cannot use our proposed semantics to \textit{define} $\mathsf{QN4}$. Instead, we extend $\mathfrak{QIL}^+$ with the following axiomatic schemes:
\begin{align}
	\sim\sim\phi &\leftrightarrow \phi\label{E:a1}\tag{An1}\\
	\sim(\phi\wedge \psi) 
	&\leftrightarrow (\sim\phi\vee \sim\psi)\label{E:a2}\tag{An2}\\
	\sim(\phi\vee \psi) &\leftrightarrow (\sim\phi\wedge \sim\psi)\label{E:a3}\tag{An3}\\
	\sim(\phi\to \psi) &\leftrightarrow (\phi\wedge \sim\psi)\label{E:a4}\tag{An4}\\
	\sim\exists x\theta &\leftrightarrow \forall x\sim\theta\label{E:a5}\tag{An5}\\
	\sim\forall x\theta &\leftrightarrow \exists x\sim\theta\label{E:a6}\tag{An6}
\end{align} 
In doing so, we obtain the system $\mathfrak{QN4}:= \mathfrak{QIL}^++(\eqref{E:a1},\ldots,\eqref{E:a6};)$ which represents the most standard way to axiomatize $\mathsf{QN4}$ known in the existing literature on the subject, see, e.g. \cite[p. 313]{odintsovwansing-trends}. We can therefore set $\mathsf{QN4}(\Omega):= \mathfrak{QN4}(\mathcal{FO}(\Omega))$ for every $\Omega\subseteq \Pi$.

This definition makes it obvious that $\mathsf{QN4}$ is a sublogic of $\mathsf{QCL}$ and extends $\mathsf{QIL}^+$. We retain these observations for future reference:
\begin{lemma}\label{L:intuitionistic-inclusion}
	Let $\Omega \subseteq \Pi$. Every substitution instance of an $\mathsf{QIL}^+(\Omega)$-theorem  is a theorem of $\mathsf{QN4}(\Omega)$ and every inference rule that is deducible in $\mathsf{QIL}^+(\Omega)$ is also deducible in $\mathsf{QN4}(\Omega)$. In particular, the following derived rules hold:
	\begin{align}
		\Gamma \cup \{\phi\}\vdash_{\mathsf{QN4}} \psi&\text{ iff }\Gamma\vdash_{\mathsf{QN4}} \phi\to\psi\label{R:DT}\tag{DT}\\
		\Gamma\vdash_{\mathsf{QN4}} \phi\to \psi&\text{ implies }\Gamma\vdash_{\mathsf{QN4}} \phi\to\forall x\psi&&x\notin FV(\Gamma \cup \{\phi\})\label{R:A}\tag{B$\forall$}\\
		\Gamma\vdash_{\mathsf{QN4}} \phi\to \psi&\text{ implies }\Gamma\vdash_{\mathsf{QN4}} \exists x\phi\to\forall x\psi&&x\notin FV(\Gamma \cup \{\psi\})\label{R:E}\tag{B$\exists$}\\
		\Gamma\vdash_{\mathsf{QN4}}\phi&\text{ implies }\Gamma\vdash_{\mathsf{QN4}}\forall x\phi\label{R:Gen} &&x \notin FV(\Gamma)\tag{Gen}
	\end{align}
\end{lemma}
\begin{lemma}\label{L:classical-inclusion}
		Let $\Omega \subseteq \Pi$. Then $\mathsf{QN4}(\Omega)\subseteq \mathsf{QCL}(\Omega)$ and every inference rule that is deducible $\mathsf{QN4}$ is also deducible in $\mathsf{QCL}$.
\end{lemma}
We observe, further, that, for every $\Omega \subseteq \Pi$, $\mathsf{QN4}(\Omega)$ is embeddable into $\mathsf{QIL}^+(\Omega^\pm \cup \{\epsilon^2\})$ and that the corresponding embedding $Tr:\mathcal{FO}(\Omega)\to \mathcal{FO}^+(\Omega^\pm \cup \{\epsilon^2\})$ can be defined by the following induction on the construction of $\phi \in \mathcal{FO}(\Pi)$: 
\begin{align*}
	Tr(P(\bar{x}_n))&:= P_+(\bar{x}_n);&&Tr(\sim P(\bar{x}_n)):= P_-(\bar{x}_n);\\
	Tr(x\equiv y)&:= (x\equiv y);&&Tr(\sim(x \equiv y)):= \epsilon(x,y);\\
	Tr(\sim\sim\phi)&:= Tr(\phi);\\
	Tr(\phi\star\psi)&:= Tr(\phi)\star Tr(\psi);&&Tr(\sim(\phi\star\psi)):= Tr(\sim\phi)\ast Tr(\sim\psi);\\
	Tr(\phi\to\psi)&:= Tr(\phi)\to Tr(\psi);&& Tr(\sim(\phi\to\psi)):= Tr(\phi)\wedge Tr(\sim\psi);\\
	Tr(Qx\phi)&:= QxTr(\phi);&& Tr(\sim Qx\phi):= Q'xTr(\sim\phi).
\end{align*} 
for all $n \in \omega$, all $P^n \in \Pi$,  all $\bar{x}_n, x, y \in Ind$, and all $\star, \ast, Q,$ and $Q'$ such that both $\{\star, \ast\} = \{\wedge, \vee\}$ and $\{Q,Q'\} = \{\forall, \exists\}$. More precisely, the following proposition holds:
\begin{proposition}\label{P:intuitionistic-embedding}
	For all  $\Gamma, \Delta \subseteq \mathcal{FO}(\Pi)$, $\Gamma\models_{\mathsf{QN4}}\Delta$ iff $Tr(\Gamma)\models_{\mathsf{QIL}^+}Tr(\Delta)$.
\end{proposition}
Proposition \ref{P:intuitionistic-embedding} is a well-known result about $\mathsf{QN4}$, see e.g. \cite[Proposition 7]{odintsovwansing-trends} for a sketch of a proof. To keep this paper reasonably self-contained, we also give its proof in Appendix \ref{App:intuitionistic-embedding}.

We proceed to define a sheaf semantics for $\mathsf{QN4}$ that we will presently show to be adequate for this logic. We start by defining the structures that we will call Nelsonian sheaves:
\begin{definition}\label{D:nelsonian-sheaf}
	Let $\Omega \subseteq \Pi$. A \textit{Nelsonian  $\Omega$-sheaf} is any structure of the form $\mathcal{S} = (W, \leq, \mathtt{M}^+, \mathtt{M}^-, \mathtt{H})$, such that:
		\begin{enumerate}
		\item $\mathcal{S}^+ = (W, \leq, \mathtt{M}^+, \mathtt{H}) \in \mathbb{I}(\Omega)$.
		
		\item $\mathcal{S}^- =(W, \leq, \mathtt{M}^-, \mathtt{H}) \in \mathbb{I}(\Omega\cup \{\epsilon^2\})$.
		
		The intuitionistic sheaves $\mathcal{S}^+$ and $\mathcal{S}^-$, defined above will be called the positive and the negative component of the Nelsonian sheaf $\mathcal{S}$.
%
	\end{enumerate}
	The class of all Nelsonian $\Omega$-sheaves will be denoted by $\mathbb{N}4(\Omega)$.
\end{definition}
Since the family $\mathtt{H}$ of canonical homomorphisms is shared by $\mathcal{S}^+$ and $\mathcal{S}^-$, it follows that $U^{\mathtt{M}^+_\mathbf{w}} = U^{\mathtt{M}^-_\mathbf{w}}$ for every $\mathbf{w} \in W$; we can therefore set $U_\mathbf{w}:= U^{\mathtt{M}^+_\mathbf{w}} = U^{\mathtt{M}^-_\mathbf{w}}$ for every $\mathbf{w} \in W$. As for the Nelsonian sheaf evaluation points, they are defined similarly to the intuitionistic ones: $EP_{n}(\Omega):= \{(\mathcal{S}, \mathbf{w}, f)\mid \mathcal{M}\in \mathbb{N}4(\Omega),\,(\mathcal{S}^+, \mathbf{w}, f)\in EP_{i}(\Omega)\}$.
 
One peculiarity of $\mathsf{QN4}$ consists in the fact that its semantics is usually constructed on the basis of two satisfaction relations, $\models^+_{n}$ and $\models^-_{n}$, instead of just one.\footnote{This feature is shared by every variant of Nelson's logic of strong negation, cf. the semantics of the conditional logic $\mathsf{N4CK}$ defined in Section \ref{S:N4CK}.} The informal interpretation of the two relations is that whenever $\mathcal{S}, \mathbf{w}\models^+_{n}\phi[f]$ holds, $\phi$ is \textit{verified} at $(\mathcal{S}, \mathbf{w}, f)$, and when $\mathcal{M}, \mathbf{w}\models^-_{n}\phi[f]$ holds, $\phi$ is \textit{falsified} at the same triple. These two relations are defined by simultaneous induction on the construction of $\phi\in \mathcal{FO}(\Omega)$:
\begin{align*}
	\mathcal{S},\mathbf{w}\models^\ast_nP(\bar{x}_n)[f] &\text{ iff } \mathcal{S}^\ast,\mathbf{w}\models_iP(\bar{x}_n)[f]\text{ iff } \mathtt{M}^\ast_\mathbf{w}\models_c P(\bar{x}_n)[f]\text{ iff } f(\bar{x}_n)\in P^{\mathtt{M}^\ast_\mathbf{w}}  &&P^n\in \Omega,\,\ast\in \{+,-\}\\
	\mathcal{S},\mathbf{w}\models^+_n(x \equiv y)[f] &\text{ iff } \mathcal{S}^+,\mathbf{w}\models_i(x \equiv y)[f]\text{ iff } \mathtt{M}^+_\mathbf{w}\models_{c}(x \equiv y)[f]\text{ iff } f(x) = f(y)\\
	\mathcal{S},\mathbf{w}\models^-_n(x \equiv y)[f] &\text{ iff } \mathcal{S}^-,\mathbf{w}\models_i\epsilon(x,y)[f]\text{ iff } \mathtt{M}^-_\mathbf{w}\models_{c}\epsilon(x,y)[f]\\
	\mathcal{S},\mathbf{w}\models^+_n(\psi\wedge\chi)[f] &\text{ iff } \mathcal{S},\mathbf{w}\models^+_n\psi[f]\text{ and }\mathcal{S},\mathbf{w}\models^+_n\chi[f]\\
	\mathcal{S},\mathbf{w}\models^-_n(\psi\wedge\chi)[f] &\text{ iff } \mathcal{S},\mathbf{w}\models^-_n\psi[f]\text{ or }\mathcal{S},\mathbf{w}\models^-_n\chi[f]\\
	\mathcal{S},\mathbf{w}\models^+_n(\psi\vee\chi)[f] &\text{ iff } \mathcal{S},\mathbf{w}\models^+_n\psi[f]\text{ or }\mathcal{S},\mathbf{w}\models^+_n\chi[f]\\
	\mathcal{S},\mathbf{w}\models^-_n(\psi\vee\chi)[f] &\text{ iff } \mathcal{S},\mathbf{w}\models^-_n\psi[f]\text{ and }\mathcal{S},\mathbf{w}\models^-_n\chi[f]\\
	\mathcal{S},\mathbf{w}\models^+_n(\psi\to\chi)[f] &\text{ iff } (\forall \mathbf{v}\geq \mathbf{w})(\mathcal{S},\mathbf{v}\not\models^+_n\psi[f\circ\mathtt{H}_{\mathbf{w}\mathbf{v}}]\text{ or }\mathcal{S}, \mathbf{v}\models^+_n\chi[f\circ\mathtt{H}_{\mathbf{w}\mathbf{v}}])\\
	\mathcal{S},\mathbf{w}\models^-_n(\psi\to\chi)[f] &\text{ iff } \mathcal{S},\mathbf{w}\models^+_n\psi[f]\text{ and }\mathcal{S},\mathbf{w}\models^-_n\chi[f]\\
	\mathcal{S},\mathbf{w}\models^+_n\sim\psi[f] &\text{ iff } \mathcal{S},\mathbf{w}\models^-_n\psi[f]\\
	\mathcal{S},\mathbf{w}\models^-_n\sim\psi[f] &\text{ iff } \mathcal{S},\mathbf{w}\models^+_n\psi[f]\\
	\mathcal{S},\mathbf{w}\models^+_n(\exists x\psi)[f] &\text{ iff } (\exists a\in U_{\mathbf{w}})(\mathcal{S},\mathbf{w}\models^+_n\psi[f[x/a]])\\
	\mathcal{S},\mathbf{w}\models^-_n(\exists x\psi)[f] &\text{ iff } (\forall \mathbf{v}\geq \mathbf{w})(\forall a\in U_{\mathbf{v}})(\mathcal{S},\mathbf{v}\models^-_n\psi[(f\circ\mathtt{H}_{\mathbf{w}\mathbf{v}})[x/a]])\\
	\mathcal{S},\mathbf{w}\models^+_n(\forall x\psi)[f] &\text{ iff } (\forall \mathbf{v}\geq \mathbf{w})(\forall a\in U_{\mathbf{v}})(\mathcal{S},\mathbf{v}\models^+_n\psi[(f\circ\mathtt{H}_{\mathbf{w}\mathbf{v}})[x/a]])\\
	\mathcal{S},\mathbf{w}\models^-_n(\forall x\psi)[f] &\text{ iff } (\exists a\in U_{\mathbf{w}})(\mathcal{S},\mathbf{w}\models^-_n\psi[f[x/a]])
\end{align*}
However, the negative satisfaction relation $\models^-_{n}$ is often viewed within this pair as a subsidiary one, which, among other things, is due to the fact that the satisfaction clauses for negation allow to completely reflect the structure of $\models^-_{n}$ within $\models^+_{n}$. Therefore, one can still capture $\mathsf{QN4}$ according to our usual pattern. In other words, we are going to prove the following:
\begin{proposition}\label{P:nelsonian-sheaves}
		For every $\Omega \subseteq \Pi$, $\mathsf{QN4}(\Omega) = \mathbb{L}(EP_n(\Omega), \models^+_n)$.
\end{proposition}
The proof of this proposition makes a substantial use of the embedding $Tr$ defined above. We sketch it in Appendix \ref{App:nelsonian-sheaves}. Another consequence of the tight relation between $\mathsf{QN4}$ and $\mathsf{QIL}^+$ is that Lemma \ref{L:intuitionistic-standard} carries over to Nelsonian sheaves:
\begin{lemma}\label{L:n4-standard}
Let $\Omega \subseteq \Pi$. For every $\ast\in \{+, -\}$, $(\mathcal{S}, \mathbf{w}, f)\in EP_{n}(\Omega)$, and $\phi \in \mathcal{FO}(\Omega)$, we have:
\begin{enumerate}
	\item If $\mathbf{v}\geq \mathbf{w}$,  and $\mathcal{S},\mathbf{w}\models^\ast_{n}\phi[f]$, then $\mathcal{S},\mathbf{v}\models^\ast_{n}\phi[f\circ\mathtt{H}_{\mathbf{w}\mathbf{v}}]$.
	
	\item The \textit{Nelsonian generated sub-sheaf} $\mathcal{S}|_\mathbf{w} = (W|_\mathbf{w}, \leq|_\mathbf{w}, \mathtt{M}^+|_\mathbf{w}, \mathtt{M}^-|_\mathbf{w}, \mathtt{H}|_\mathbf{w})\in \mathbb{N}4(\Omega)$ is such that both $(\mathcal{S}|_\mathbf{w})^+$ and $(\mathcal{S}|_\mathbf{w})^-$ are the generated sub-sheaves of $\mathcal{S}^+$ and $\mathcal{S}^-$, respectively. Then:
	$$
	\mathcal{S}|_\mathbf{w},\mathbf{v}\models^\ast_{n}\phi[f]\text{ iff } \mathcal{S},\mathbf{v}\models^\ast_{n}\phi[f]
	$$ 
	for every $\mathbf{v}\geq \mathbf{w}$ and every $f:Ind\to U_\mathbf{v}$.
\end{enumerate}	 
\end{lemma}
Proof of the Lemma is relegated to Appendix \ref{App:n4-standard}.
\begin{remark}\label{Rm:simplified}
	For the rest of the paper, we will be suppressing $\Pi$ in notations like $\mathcal{FO}(\Pi)$, $\mathbb{N}4(\Pi)$, $EP_n(\Pi)$, and $\mathsf{QN4}(\Pi)$.
\end{remark}
In the remaining part of the section we develop $\mathsf{QN4}$ to an extent that is sufficient for the subsequent sections.  As usual, it follows from our semantic definitions that the truth value of a formula $\phi \in \mathcal{FO}$ only depends on the values assigned by $f$ to the values of the variables in $FV(\phi)$. We will therefore write, for any $\ast \in \{+, -\}$, $\mathcal{S}, \mathbf{w}\models^\ast_{n}\phi[x_1/a_1,\ldots,x_n/a_n]$ iff $\mathcal{S} \in \mathbb{N}4$, $\mathbf{w} \in W$, $\phi \in \mathcal{FO}^{\bar{x}_n}$, and $\mathcal{S}, \mathbf{w}\models^\ast_{n}\phi[f]$ for every (equivalently, any) $f$ such that both $(\mathcal{S}, \mathbf{w}, f) \in EP_{n}$ and  $f(x_i) = a_i$ for every $1 \leq i \leq n$. In particular, we will write $\mathcal{S}, \mathbf{w}\models^\ast_{n}\phi$ iff $\phi \in \mathcal{FO}^\emptyset$ and we have $\mathcal{S}, \mathbf{w}\models^\ast_{n}\phi[f]$ for every (equivalently, any) function $f$ such that $(\mathcal{S}, \mathbf{w}, f) \in EP_{n}$; we will write $\mathcal{S}\models^\ast_{n}\phi$ iff  $\mathcal{S}, \mathbf{w}\models^\ast_{n}\phi$ for every $\mathbf{w}\in W$. The conventions of this paragraph also extend to sets and bi-sets of formulas in an obvious way.

Next, we introduce the following abbreviations for all $\phi, \psi\in\mathcal{FO}$:
\begin{itemize}
	\item $\phi \Rightarrow \psi$ (strong implication) for $(\phi \to \psi)\wedge(\sim\psi\to \sim\phi)$.
	
	\item $\phi \Leftrightarrow \psi$ (strong equivalence) for $(\phi \Rightarrow \psi)\wedge(\psi\Rightarrow \phi)$, or, (equivalently, in view of Lemma \ref{L:intuitionistic-inclusion}), for $(\phi\leftrightarrow\psi)\wedge(\sim\phi\leftrightarrow\sim\psi)$.
	
	\item $\phi\,\&\,\psi$ (ampersand) for $\sim(\phi\to\sim\psi)$.
\end{itemize}
The following lemma lists some properties of these derived connectives:
\begin{lemma}\label{L:derived-connectives}
	Let $\phi, \psi, \chi, \theta \in \mathcal{FO}$ be chosen arbitrarily. Then all of the following theorems hold in $\mathsf{QN4}$:
	\begin{align}
		\vdash_{\mathsf{QN4}}(\phi \Rightarrow \psi)&\to(\phi\to\psi)\label{E:T1}\tag{T1}\\
		\vdash_{\mathsf{QN4}}(\phi \Leftrightarrow \psi)&\to(\phi\leftrightarrow\psi)\label{E:T2}\tag{T2}\\
		\vdash_{\mathsf{QN4}}(\phi \Rightarrow \psi)&\Leftrightarrow(\sim\psi\Rightarrow\sim\phi)\label{E:T3}\tag{T3}\\
		\vdash_{\mathsf{QN4}}(\phi \Leftrightarrow \psi)&\Leftrightarrow(\sim\phi \Leftrightarrow\sim\psi)\label{E:T4}\tag{T4}\\
		\vdash_{\mathsf{QN4}}\phi &\Leftrightarrow \phi\label{E:T5}\tag{T5}\\
		\vdash_{\mathsf{QN4}}(\phi \Leftrightarrow \psi)&\Leftrightarrow(\psi \Leftrightarrow\phi)\label{E:T6}\tag{T6}\\
		\vdash_{\mathsf{QN4}}(\phi \Leftrightarrow \psi)&\Leftrightarrow((\psi \Leftrightarrow\chi)\Leftrightarrow(\phi \Leftrightarrow\chi))\label{E:T7}\tag{T7}\\
		\vdash_{\mathsf{QN4}}\sim\sim\phi &\Leftrightarrow \phi\label{E:T8}\tag{T8}\\
		\vdash_{\mathsf{QN4}}\sim(\phi\wedge \psi) 
		&\Leftrightarrow (\sim\phi\vee \sim\psi)\label{E:T9}\tag{T9}\\
		\vdash_{\mathsf{QN4}}\sim(\phi\vee \psi) &\Leftrightarrow (\sim\phi\wedge \sim\psi)\label{E:T10}\tag{T10}\\
		\vdash_{\mathsf{QN4}}\sim\exists x\theta &\Leftrightarrow \forall x\sim\theta\label{E:T11}\tag{T11}\\
		\vdash_{\mathsf{QN4}}\sim\forall x\theta &\Leftrightarrow \exists x\sim\theta\label{E:T12}\tag{T12}\\	
		\vdash_{\mathsf{QN4}}((\phi \Leftrightarrow \psi)  \wedge (\chi \Leftrightarrow \theta))&\to ((\phi\ast\chi)\Leftrightarrow(\psi\ast\theta))\label{E:T13}\tag{T13} &&\ast\in \{\wedge,\vee,\to\}\\
		\vdash_{\mathsf{QN4}}(\phi \Leftrightarrow \psi)&\to(Qx\phi\Leftrightarrow Qx\psi)\label{E:T14}\tag{T14} &&Q\in \{\forall, \exists\}\\
		\vdash_{\mathsf{QN4}}(\phi\,\&\,\psi)&\leftrightarrow (\phi\wedge\psi)\label{E:T15}\tag{T15}\\
		\vdash_{\mathsf{QN4}}\sim(\phi\,\&\,\psi)&\leftrightarrow (\phi\to\sim\psi)\label{E:T16}\tag{T16}
	\end{align}
Observe that \eqref{E:T9}-- \eqref{E:T12} strengthen (in view of \eqref{E:T2}) the axioms \eqref{E:a1}-- \eqref{E:a2} and \eqref{E:a5}-- \eqref{E:a6}, respectively. That this strengthening is non-trivial, follows from the fact that the converses of both \eqref{E:T1} and \eqref{E:T2} fail in general; moreover, one cannot replace $\leftrightarrow$ with $\Leftrightarrow$ in \eqref{E:a4}, \eqref{E:T15}, or \eqref{E:T16}.

Finally, note that the following schemes are not, in general, valid in $\mathsf{QN4}$:
\begin{align*}
(\phi \to \psi)\to(\sim\psi\to\sim\phi),\,
(\phi \leftrightarrow \psi)\to(\sim\phi \leftrightarrow\sim\psi)	
\end{align*}
\end{lemma}
We sketch the proof in Appendix \ref{App:derived-connectives}.

Thus, even though our defined connectives wouldn't make much sense in $\mathsf{QCL}$ as they are classically equivalent to $\to$, $\leftrightarrow$ and $\wedge$, respectively, we see that the situation is different in the context of $\mathsf{QN4}$. Indeed, whereas $\to$  and $\leftrightarrow$ in $\mathsf{QN4}$ cannot be contraposed, $\Rightarrow$ and $\Leftrightarrow$ define stronger (and contraposable) analogues to $\to$ and $\leftrightarrow$, respectively. Finally, the ampersand is especially convenient for handling the falsity conditions of restricted existential quantification in the context of $\mathsf{QN4}$. We provide a motivation for this use of ampersand in Appendix \ref{App:ampersand}.

Note, furthermore, that the failure of theorems like $(\phi \leftrightarrow \psi)\to(\sim\phi \leftrightarrow\sim\psi)$ implies that provably equivalent formulas in general fail to be substitutable for one another in $\mathsf{QN4}$. However, this failure does not extend to the strong equivalence; to the contrary, one can even prove that a substitution of strongly equivalent formulas results in a formula that is also strongly equivalent to the original one. 

As is well-known, a substitution of one formula for another in a first-order context can lead to rather complicated definitions. For the purposes of our paper it is sufficient to confine ourselves to the following one simple case.

If $\phi, \psi \in \mathcal{FO}$ then we denote by $\phi/\psi$ the result of replacing every occurrence of $v_0\equiv v_0$ in $\phi$ by $\psi$. The replacement operation then commutes with all the connectives and quantifiers, and, for $\phi \in At$, we stipulate that:
$$
\phi/\psi:=\begin{cases}
	\psi,\text{ if }\phi = (v_0\equiv v_0)\\
	\phi,\text{ otherwise.}
\end{cases}
$$
Then the following holds (See Appendix \ref{App:substitution} for a proof):
\begin{lemma}\label{L:substitution}
	For all $\phi, \psi, \chi \in \mathcal{FO}$, we have:
	\begin{equation}\label{E:T17}\tag{T17}
		\vdash_{\mathsf{QN4}}(\phi\Leftrightarrow\psi)\to((\chi/\phi)\Leftrightarrow(\chi/\psi))
	\end{equation}
\end{lemma}
An immediate consequence of Lemma \ref{L:substitution} is that disjunction can also be understood as defined connective:
\begin{corollary}\label{C:disjunstion}
	For all $\phi, \psi, \chi \in \mathcal{FO}$, we have $\vdash_{\mathsf{QN4}}(\chi/(\phi\vee\psi))\Leftrightarrow(\chi/(\sim(\sim\phi\wedge\sim\psi)))$
\end{corollary}
\begin{proof}
	By \eqref{E:T8}, \eqref{E:T10}, and Lemma \ref{L:substitution}.
\end{proof}
The formulas of the form $(\forall x)(\chi \to (\phi\Leftrightarrow \psi))$ will play a prominent role in the subsequent sections. We would like, therefore, also to take a moment to spell out their truth conditions in terms of Nelsonian sheaf semantics:
\begin{corollary}\label{C:set-encoding}
	Let $x,y \in Ind$ be pairwise distinct, let $\phi\in \mathcal{FO}^{x,y}$ and let $\psi, \chi \in \mathcal{FO}^y$. For every $\mathcal{S}\in \mathbb{N}4$, every $\mathbf{w} \in W$, and every $a \in U_\mathbf{w}$, we have
	$$
	\mathcal{S}, \mathbf{w} \models^+_n (\forall y)(\chi(y) \to (\phi(y,x)\Leftrightarrow \psi(y)))[x/a]
	$$ 
	iff, for every $\mathbf{v} \geq \mathbf{w}$ and every $b \in U_\mathbf{v}$ such that $\mathcal{S}, \mathbf{v}\models^+_n \chi[x/b]$ we have both
	$$
	\mathcal{S}, \mathbf{v}\models^+_n \phi[y/b, x/\mathtt{H}_{\mathbf{w}\mathbf{v}}(a)]\text{ iff }\mathcal{S}, \mathbf{v}\models^+_n \psi[y/b]
	$$
	and
	$$
	\mathcal{S}, \mathbf{v}\models^-_n \phi[y/b, x/\mathtt{H}_{\mathbf{w}\mathbf{v}}(a)]\text{ iff }\mathcal{S}, \mathbf{v}\models^-_n \psi[y/b].
	$$
\end{corollary}
\begin{proof}
	By application of the definitions.
\end{proof}
Moreover, we would like to state an application of Lemma \ref{L:substitution} to formulas of this type as a separate corollary:
\begin{corollary}\label{C:substitution}
	For all $\phi, \psi, \chi, \theta \in \mathcal{FO}$, we have
	\begin{equation}\label{E:T18}\tag{T18}
		\vdash_{\mathsf{QN4}}\forall x(\chi\to (\phi\Leftrightarrow\psi))\to(\forall x(\chi \to (\theta/\phi))\leftrightarrow\forall x(\chi \to (\theta/\psi))
	\end{equation} 
\end{corollary}
\begin{proof}
We reason as follows:
\begin{align}
		\vdash_{\mathsf{QN4}}(\chi\to (\phi\Leftrightarrow\psi))&\to (\chi\to((\theta/\phi)\Leftrightarrow(\theta/\psi)))\label{E:cr1}&&\text{by \eqref{E:T17}, Lm \ref{L:intuitionistic-inclusion}}\\
			\vdash_{\mathsf{QN4}}(\chi\to((\theta/\phi)\Leftrightarrow(\theta/\psi)))&\to(\chi\to((\theta/\phi)\leftrightarrow(\theta/\psi)))\label{E:cr2}&&\text{by \eqref{E:T2}, Lm \ref{L:intuitionistic-inclusion}}\\
			\vdash_{\mathsf{QN4}}(\chi\to((\theta/\phi)\leftrightarrow(\theta/\psi)))&\to((\chi\to(\theta/\phi))\leftrightarrow(\chi\to(\theta/\phi)))\label{E:cr3}&&\text{by Lm \ref{L:intuitionistic-inclusion}}\\
		\vdash_{\mathsf{QN4}}(\chi\to (\phi\Leftrightarrow\psi))&\to((\chi\to(\theta/\phi))\leftrightarrow(\chi\to(\theta/\phi)))\label{E:cr4}&&\text{by \eqref{E:cr1}--\eqref{E:cr3},   Lm\ref{L:intuitionistic-inclusion}}		
\end{align}
Now \eqref{E:T18} follows from \eqref{E:cr3} by Lemma \ref{L:intuitionistic-inclusion}.
\end{proof}

\section{Modal logics and their standard translation embeddings}\label{S:idea}
Although modal logics are not the main subject of this paper, they are so tightly connected to conditional logics based on the so-called Chellas semantics that a review of their standard translation embedding properties provides the best possible introduction to our main result. We start by setting $Prop:= \Pi \setminus \{S^1, O^1, E^2, R^3\} = \{p^1_n\mid n \in \omega\}$. The modal language $\mathcal{MD}$ is generated from $Prop$ by the following BNF:
$$
\phi::= p\mid \phi\wedge\phi\mid\phi\vee\phi\mid\phi\to\phi\mid\sim\phi\mid\Box\phi,
$$ 
where $p^1 \in Prop$. In case we want to have $\Diamond\phi$ as an elementary modality rather than an abbreviation for $\sim\Box\sim\phi$, we obtain the language $\mathcal{MD}^\Diamond$. The Kripke semantics for this language uses the class $pMod$ of pointed Kripke models as evaluation points and the modal  satisfaction relation $\models_m$; the definitions of both can be easily found in any of the numerous textbooks, e.g. in \cite[Ch 3.2]{chagrov}. The minimal normal modal logic $\mathsf{K}$ can then be defined by $\mathsf{K}:= \mathbb{L}(pMod, \models_m)$.

If we now set $\mu:= Prop\cup\{E^2\}$ and assume that, for every $x\in Ind$ we have fixed an $y\in Ind$ which is distinct from $x$, then the mapping $\sigma\tau_x:\mathcal{MD}\to\mathcal{FO}(\mu)^x$ is defined\footnote{One normally denotes the binary relation by $R$ rather than $E$ and writes $ST_x$ instead of $\sigma\tau_x$. However, in our paper we prefer to reserve $R$ and $ST_x$ for the formulation of our main result on the standard translation embedding of the conditional logic.} by induction on the construction of $\phi\in\mathcal{MD}$ and is called the modal standard $x$-translation:
\begin{align*}
	\sigma\tau_x(p)&:= p(x)&&p^1 \in Prop\\
	\sigma\tau_x(\sim\psi)&:= \sim \sigma\tau_x(\psi)\\
	\sigma\tau_x(\psi\ast\chi)&:= \sigma\tau_x(\psi)\ast \sigma\tau_x(\chi)&&\ast\in \{\wedge, \vee, \to\}\\
	\sigma\tau_x(\Box\psi) &:= \forall y(Exy\to \sigma\tau_{y}(\psi))
\end{align*}
The following lemma is then often  presented without any proof in the existing literature:
\begin{lemma}\label{L:K}
	For every $x\in Ind$, every $\phi\in\mathcal{MD}$ and every $(\mathcal{M}, \mathbf{w})\in pMod$, we have $\mathcal{M}, \mathbf{w}\models_m\phi$ iff $\mathcal{M} \models_c\sigma\tau_x(\phi)[x/\mathbf{w}]$.
\end{lemma}
Lemma \ref{L:K} comes across as very natural, indeed, almost self-evident: the class of Kripke models for modal logic is just another name for $\mathbb{C}(\mu)$, $x\in Ind$ chooses a point, turning it into a pointed structure, and the definition of $\sigma\tau_x$ is just a direct formalization of the inductive definition of $\models_m$ in first-order logic.

It is then just a small step to get from Lemma \ref{L:K} to the following
\begin{proposition}\label{P:K}
	For all $\Gamma, \Delta \subseteq \mathcal{MD}$ and for every $x \in Ind$, we have $\Gamma \models_{\mathsf{K}} \Delta$ iff $\sigma\tau_x(\Gamma)\models_{\mathsf{QCL}}\sigma\tau_x(\Delta)$.	
\end{proposition}
In other words, we find that $\sigma\tau_x$ faithfully embeds $\mathsf{K}$ into $\mathsf{QCL}(\mu)$ for every $x\in Ind$. 

The main line of thought presented in the paper can be viewed as a natural continuation of this very idea. However, before we get to it, we need to observe that things become more complicated as soon as we replace the classical propositional basis of $\mathsf{K}$ with some non-classical logic. For example, the basic intuitionistic modal logic $\mathsf{IK}$ (see, e.g. \cite{fischer-servi}) is supplied with a Kripke semantics of bi-relational models over a classical metalanguage. This immediately results in an analogue of Proposition \ref{P:K} saying that a straightforward formalization of this semantics amounts to defining an embedding of $\mathsf{IK}$ into a classical first-order correspondence language based on two binary predicates.\footnote{Note, however, that already in the basic case of $\mathsf{IK}$ this embedding is only faithful modulo a certain first-order theory encoding the interaction conditions between the two accessibility relations in the Kripke models that are assumed in the semantics of $\mathsf{IK}$.}. However, and much more interestingly, $\sigma\tau_x$ turns out to be useful for $\mathsf{IK}$, too, as it happens to embed it into $\mathsf{QIL}(\mu)$. This result, which is far from being trivial (one version of its proof can found in \cite[Ch. 5]{simpson}), suggests that $\mathsf{IK}$ carves out exactly the same fragment of $\mathsf{QIL}(\mu)$ as the one carved out by $\mathsf{K}$ in $\mathsf{QCL}(\mu)$; and the latter circumstance provides a strong argument in favor of viewing $\mathsf{IK}$ as the correct intuitionistic analogue of $\mathsf{K}$.

One may picture this situation as follows: assume that an intuitionist gets interested in modal logic, and asks a classicist colleague to explain them $\mathsf{K}$. It is natural to expect that this explanation, if it is to be given in precise terms, will end up mentioning $\sigma\tau_x$ at some point. Of course, the intuitionist will read the definition of $\sigma\tau_x$ in terms of $\mathsf{QIL}$ rather than $\mathsf{QCL}$ and thus will end up with $\mathsf{IK}$. In a sense then, $\mathsf{IK}$ is nothing but $\mathsf{K}$ \textit{read intuitionistically}.

Moreover, it is easy to see that at most one intuitionistic modal logic can be faithfully embedded by $\sigma\tau_x$ into $\mathsf{QIL}(\mu)$. In this sense, $\mathsf{IK}$ can be called \textit{the} correct intuitionistic counterpart to $\mathsf{K}$. We may even try to elevate this to a general criterion: \textit{given a  non-classical logic $Q\nu$ such that $Q\nu$ is a proper sublogic of $\mathsf{QCL}$ and its propositional fragment $\nu$ is a proper sublogic of $\mathsf{CL}$, a modal extension $\nu\mathsf{K}$ of $\nu$ is the $\nu$-based counterpart of $\mathsf{K}$ iff  $\sigma\tau_x$ faithfully embeds $\nu\mathsf{K}$ into $\mathsf{Q}\nu(\mu)$.}   

Yet, the very deep and enlightening result about $\mathsf{IK}$ is not without its little annoying wrinkles. Observe that $\mathsf{IK}$ is formulated over $\mathcal{MD}^\Diamond$ rather than $\mathcal{MD}$ as the diamond is no longer definable in terms of box. Therefore, the above definition of $\sigma\tau_x$ is no longer sufficient, as it has to include 
\begin{align}
	\sigma\tau^i_x(\Diamond\psi) &:= \exists y(Exy\wedge \sigma\tau^i_{y}(\psi))\label{E:diam1}
\end{align}  
as an additional clause. Of course, this clause is \textit{implied} by the above definition of $\sigma\tau_x$ over $\mathsf{QCL}$ and thus adding it to the definition of $\sigma\tau_x$ results in a classically equivalent reformulation $\sigma\tau^i_x$ of $\sigma\tau_x$. However, other classically equivalent formulations of $\sigma\tau_x$ are also possible, and some of them make the analogue of Proposition \ref{P:K} fail for $\mathsf{IK}$ and $\mathsf{QIL}$. Think, for example, of the extension $\sigma\tau^j_x$ of $\sigma\tau_x$ with the following clause:
\begin{align}
	\sigma\tau^j_x(\Diamond\psi) &:= \sim\forall y(Exy\to \sim\sigma\tau^j_{y}(\psi))\label{E:diam2}
\end{align}
The choice of a right classical reformulation of $\sigma\tau_x$ can therefore affect the truth of our analogue of Proposition \ref{P:K} in the intuitionistic case. Its precise formulation, then, goes as follows
\begin{proposition}\label{P:IK}
	There exists a classically equivalent reformulation $\sigma\tau^i_x$ of $\sigma\tau_x$ such that, for all $\Gamma, \Delta \subseteq \mathcal{MD}^\Diamond$ and for every $x \in Ind$, we have $\Gamma \models_{\mathsf{IK}} \Delta$ iff $\sigma\tau^i_x(\Gamma)\models_{\mathsf{QIL}}\sigma\tau^i_x(\Delta)$.	
\end{proposition}
Of course, the required classical reformulation is pretty clear in the case of $\mathsf{IK}$; in fact, it is so clear that one is tempted to dismiss the existential quantifier in the formulation of Proposition \ref{P:IK} as mere pedantry and to speak of just $\sigma\tau_x$ instead. However, we need to be mindful of the fact that our caveat `up to a classically equivalent reformulation' reflects a very real possibility that (to continue with our metaphorical scenario) in explaining the semantics of $\mathsf{K}$ to an interested intuitionist, the classicist modal logician might slip up and explain the semantics of diamond according to \eqref{E:diam2} rather than \eqref{E:diam1}, which might lead, on the side of the intuitionist colleague, to a logic that is distinct from $\mathsf{IK}$. In other words, our caveat uncovers the fact that $\mathsf{IK}$ is the result of the intuitionistic reading of $\mathsf{K}$ \textit{only up to a certain wording} of $\mathsf{K}$.

Given these considerations, our tentative principle for seeking out non-classical counterparts for $\mathsf{K}$ needs to be corrected as follows: \textit{given a non-classical logic $Q\nu$ such that $Q\nu$ is a proper sublogic of $\mathsf{QCL}$ and its propositional fragment $\nu$ is a proper sublogic of $\mathsf{CL}$, a modal extension $\nu\mathsf{K}$ of $\nu$ is a $\nu$-based counterpart of $\mathsf{K}$ iff some classically equivalent reformulation of $\sigma\tau_x$ faithfully embeds $\nu\mathsf{K}$ into $\mathsf{Q}\nu(\mu)$; the comparative merits of different $\nu$-based counterparts of $\mathsf{K}$ will then have to be judged on the basis of other properties of $Q\nu$.}

Apart from $\mathsf{IL}$, $\mathsf{N4}$ provides another possible instantiation of $\nu$ in the above principle; in fact we are not the first to notice this, as the paper \cite{odintsovwansing} by S. Odintsov and H. Wansing both defines several $\mathsf{N4}$-based analogues of $\mathsf{K}$ and looks into their standard translation embeddings into $\mathsf{QN4}(\mu)$. In the context of this paper, the most interesting of these results is \cite[Proposition 7]{odintsovwansing}, which shows that the $\mathsf{N4}$-based modal logic $\mathsf{FSK}^d$ is faithfully embedded into $\mathsf{QN4}(\mu)$ by a standard translation $T'_x$ which provides yet another classical equivalent to $\sigma\tau_x$. Although \cite{odintsovwansing} defines $\mathsf{FSK}^d$ over $\mathcal{MD}^\Diamond$, the logic also derives the classical definition of $\Diamond$ in terms of $\Box$ which makes it possible to alternatively define $\mathsf{FSK}^d$ and $T'_x$ over $\mathcal{MD}$\footnote{See \cite[Section 5.2]{nelsonian} for details.} The result of \cite{odintsovwansing} can then be reformulated in the terminology of this paper as
\begin{proposition}\label{P:ow}
		There exists a classically equivalent reformulation $T'_x$ of $\sigma\tau_x$ such that, for all $\Gamma, \Delta \subseteq \mathcal{MD}$ and for every $x \in Ind$, we have $\Gamma \models_{\mathsf{FSK}^d} \Delta$ iff $T'_x(\Gamma)\models_{\mathsf{QN4}}T'_x(\Delta)$.
\end{proposition}
Unfortunately, $T'_x$ looks a bit clumsy in that it fails to commute (up to a strong equivalence) with the propositional connectives of $\mathsf{N4}$. For example, given a $p^1 \in Prop$, we have $T'_x(\Box p) = \forall y(Rxy \to py)$,
but also $T'_x(\sim\Box p) = \exists y(Rxy \wedge \sim py)$,
whereas $\sim T'_x(\Box p) = \sim\forall y(Rxy \to py)$,
and the latter formula is clearly equivalent to $\exists y(Rxy\,\&\,\sim py)$ rather than $T'_x(\sim\Box p)$. Whereas \eqref{E:T15} ensures that $\exists y(Rxy\,\&\,\sim py)\leftrightarrow \exists y(Rxy \wedge \sim py)$, we fail to get the strong equivalence between the two formulas.

Still, Proposition \ref{P:ow} shows that some sort of an $\mathsf{N4}$-based analogue to Propositions \ref{P:K} and \ref{P:IK} is possible and that $\mathsf{FSK}^d$, at least to some extent, can be viewed as an $\mathsf{N4}$-based counterpart of $\mathsf{K}$.

\section{$\mathsf{CK}$ the minimal classical logic of conditionals}\label{S:conditional}
In this section we will introduce the language of conditional propositional logic (\textit{conditional language} for short) which we will denote by $\mathcal{CN}$. The language is generated from $Prop$ by the following BNF:
$$
\phi::= p\mid \phi\wedge\phi\mid\phi\vee\phi\mid\phi\to\phi\mid\sim\phi\mid\phi\boxto\phi,
$$ 
where $p^1 \in Prop$. The new connective $\boxto$ will be referred to as \textit{would-conditional} or simply \textit{conditional}; the elements of $\mathcal{CN}$ will be called \textit{conditional formulas}. We will apply to the conditional formulas all of the abbreviations introduced earlier for the connectives in $\mathcal{FO}$, plus the following new one:
\begin{itemize}
	\item $\phi\diamondto\psi$ (might-conditional) for $\sim(\phi\boxto\sim\psi)$. 
\end{itemize}
One relatively popular and well-researched semantics for $\mathcal{CN}$ is the so-called Chellas semantics. In this paper, we will use Chellas semantics to define two logics over $\mathcal{CN}$, denoted by $\mathsf{CK}$ and $\mathsf{N4CK}$, respectively. We begin by defining our models:
\begin{definition}\label{D:ccmodel}
		A \textit{classical conditional model} is a structure of the form $\mathcal{M} = (W, R, V)$, such that:
	\begin{enumerate}
		\item $W \neq \emptyset$ is a set of worlds, or nodes.
		
		\item $R \subseteq W \times \mathcal{P}(W)\times W$ is the accessibility relation. Thus, for every $X \subseteq W$, $R$ induces a binary relation $R_X$ on $W$ such that, for all $\mathbf{w},\mathbf{v} \in W$, $R_X(\mathbf{w},\mathbf{v})$ iff $R(\mathbf{w}, X, \mathbf{v})$.
		
		\item $V:Prop\to \mathcal{P}(W)$, called evaluation function.
	\end{enumerate}
The class of all classical conditional models will be denoted by $\mathbb{CK}$.
\end{definition}
The absence of bound variables in $\mathcal{CN}$ allows for a shorter definition of evaluation points compared to the one we had in the first-order case: we simply set
$$
EP_{ck}:= \{(\mathcal{M}, \mathbf{w})\mid \mathcal{M} \in \mathbb{CK},\,\mathbf{w} \in W\}.
$$
The satisfaction relation (denoted by $\models_{ck}$) is then supplied by the following definition:
\begin{align*}
	\mathcal{M}, \mathbf{w}&\models_{ck} p \text{ iff } \mathbf{w} \in V(p)&& p^1 \in Prop\\
	\mathcal{M}, \mathbf{w}&\models_{ck} \sim\psi \text{ iff } \mathcal{M}, \mathbf{w}\not\models_{ck} \psi\\
	\mathcal{M}, \mathbf{w}&\models_{ck} \psi \wedge \chi \text{ iff } \mathcal{M}, \mathbf{w}\models_{ck} \psi\text{ and }\mathcal{M}, \mathbf{w}\models_{ck} \chi\\
	\mathcal{M}, \mathbf{w}&\models_{ck} \psi \vee \chi \text{ iff } \mathcal{M}, \mathbf{w}\models_{ck} \psi\text{ or }\mathcal{M}, \mathbf{w}\models_{ck} \chi\\
	\mathcal{M}, \mathbf{w}&\models_{ck} \psi \to \chi \text{ iff } \mathcal{M}, \mathbf{v}\not\models_{ck} \psi\text{ or }\mathcal{M}, \mathbf{v}\models_{ck} \chi\\
	\mathcal{M}, \mathbf{w}&\models_{ck}\psi \boxto \chi \text{ iff } (\forall \mathbf{v} \in W)(R_{\|\psi\|^{ck}_\mathcal{M}}(\mathbf{w},\mathbf{v})\text{ implies }\mathcal{M}, \mathbf{v}\models_{ck} \chi)
\end{align*}
where $\|\psi\|^{ck}_\mathcal{M} := \{\mathbf{w}\in W\mid 	\mathcal{M}, \mathbf{w}\models_{ck}\psi\}$. Having thus defined the intended semantics for the classical conditional logic, we set $\mathsf{CK} = \mathbb{L}(EP_{ck}, \models_{ck})$.

It is easy to see that $\mathsf{CK}$ extends $\mathsf{CL}$. Furthermore, our semantics represents conditionals as a sort of modal sentences where the box is indexed by the antecedent; in other words, it generates a binary $\phi$-accessibility relation for every $\phi\in \mathcal{CN}$ and reads a conditional `if $\phi$ then $\psi$' as something like `it is $\phi$-necessary that $\psi$'. This connection leads to some obvious correspondences between normal modal logics and conditional logics based on Chellas semantics: indeed, one can meaningfully ask \textit{what kind of modalities} are  we dealing with when we view the conditionals of a given logic as nothing but formula-indexed modal boxes, and these considerations lead us to the notion of a modal companion of a given conditional logic. For example, the formula-indexed boxes assumed by the conditionals in $\mathsf{CK}$ can be viewed as $\mathsf{K}$-boxes, which, for every $\phi \in \mathcal{CN}$, is attested by the existence of the faithful embedding $\eta_\phi$ of $\mathsf{K}$ into $\mathsf{CK}$ which commutes with the propositional connectives and reads modal boxes as conditionals with $\phi$ as the antecedent.


Thus the idea of standard translation, at least for the conditional logics based on Chellas semantics, readily suggests itself. Of course, our accessibility relation is now ternary, so we have to use $R^3$ instead of $E^2$ in the correspondence language. An additional difficulty is that our binary accessibility relations are generated from $R^3$ by truth-sets of the conditional formulas, therefore, our semantics has to say something not only about particular nodes in a Kripke model, but also about certain subsets of this model. Thus we must allow for two kinds of items in the first-order correspondence language, the nodes of the corresponding Kripke model and the subsets thereof; we need, therefore, to add two additional unary predicates to our correspondence language. In our paper we denote them  $S^1$ and $O^1$ (for `sets' and `objects'). Next, we must be able to express that some sets are in fact truth sets of certain conditional formulas, that is to say that an item of the object kind is their element iff this item verifies a given conditional formula, which further means that we have to allow the elementhood relation into the correspondence language. In what follows, we will denote this relation by $E^2$. To sum up, our correspondence language turns out to be just $\mathcal{FO}(\Pi) = \mathcal{FO}$. 

Given this informal interpretation of the predicates in $\Pi$, the following definition is just a straightforward first-order formalization of the inductive definition of $\models_{ck}$. For every $x \in Ind$, let us fix $y,z,w \in Ind$ such that $x,y,z,w$ are pairwise distinct. Moreover, for any given $\phi \in \mathcal{FO}$ let $(\forall x)_O\phi$ abbreviate $\forall x(Ox\to \phi)$. The \textit{classical standard} $x$-\textit{translation} $st_x:\mathcal{CN} \to \mathcal{FO}^x$ is given by induction on the construction of a $\phi \in \mathcal{CN}$:
\begin{align*}
	st_x(p)&:= p(x)&&p^1 \in Prop\\
	st_x(\sim\psi)&:= \sim st_x(\psi)\\
	st_x(\psi\ast\chi)&:= st_x(\psi)\ast st_x(\chi)&&\ast\in \{\wedge, \vee, \to\}\\
	st_x(\psi\boxto\chi) &:= \exists y(Sy \wedge (\forall z)_O (Ezy\leftrightarrow st_z(\psi))\wedge\forall w(Rxyw\to st_{w}(\chi)))
\end{align*}
Of course, in order for $st_x$ to work correctly, the interaction between $S$, $O$ and $E$ must in fact resemble the interaction between sets of worlds and worlds in a classical conditional model. Although $E$ does not need to be the `real' elementhood verifying anything like ZFC, we can at least expect that $E$ satisfies a variant of the extensionality axiom in that we do not have multiple copies of the same truth-set; moreover, sufficiently many instances of set-theoretic comprehension must hold in order to guarantee that at least one copy of a truth-set for any given $\phi  \in \mathcal{CN}$ is contained in the extension of the $S$ predicate. The following first-order theory $Th_{ck}\subseteq \mathcal{FO}^\emptyset$ sums up these requirements:
\begin{align}
	&\forall x\forall y(Sx\wedge Sy \wedge (\forall z)_O(Ezx\leftrightarrow Ezy) \to x \equiv y)\label{E:thc5}\tag{Thc1}\\
	&\exists x(Sx\wedge (\forall y)_O(Eyx\leftrightarrow py))\qquad\qquad\qquad\qquad\qquad\qquad\qquad\qquad(p^1\in Prop)\label{E:thc6}\tag{Thc2}\\
	&\forall x(Sx\rightarrow\exists y(Sy\wedge (\forall z)_O(Ezy\leftrightarrow \sim Ezx))\label{E:thc7}\tag{Thc3}\\
	&\forall x\forall y((Sx\wedge Sy)\rightarrow\exists z(Sz\wedge (\forall w)_O(Ewz\leftrightarrow (Ewx\wedge Ewy))))\label{E:thc8}\tag{Thc4}\\
	&\forall x\forall y((Sx\wedge Sy)\rightarrow\exists z(Sz\wedge (\forall w)_O(Ewz\leftrightarrow \forall u(Rwxu\to Euy))))\label{E:thc9}\tag{Thc5}
\end{align}
It is easy to show that this, relatively lightweight, encoding of the subset structure of a classical conditional model is sufficient to turn $st_x$ into a faithful embedding of $\mathsf{CK}$ into $\mathsf{QCL}$. Although the proof of the fact that $st_x$ faithfully embeds $\mathsf{CK}$  into $\mathsf{QCL}$ under the assumption of $Th_{ck}$ provides no principal difficulty, we could not find it in the existing literature, so we devote the rest of this section to sketching it. In other words,\footnote{That Proposition \ref{P:ck-into-cl} only shows the faithfulness of the standard translation embedding modulo a certain set of first-order assumptions reflects the complexity of the semantics of conditionals as compared to modal semantics. However, this type of complications is also well-known in the modal case. For example $\sigma\tau_x$ faithfully embeds the classical modal logic $\mathsf{S4}$ into $\mathsf{QCL}(\mu)$ only modulo the assumption that $\forall xExx \wedge \forall xyz(Exy\wedge Eyz \to Exz)$.} we claim that:
\begin{proposition}\label{P:ck-into-cl}
	For all $\Gamma, \Delta \subseteq \mathcal{CN}$ and for every $x \in Ind$, $\Gamma\models_{\mathsf{CK}}\Delta$ iff $Th_{ck}\cup st_x(\Gamma)\models_{\mathsf{QCL}}st_x(\Delta)$.
\end{proposition} 
The proof of this proposition employs the following technical lemmas:
\begin{lemma}\label{L:cl-comprehension}
For every $\phi \in \mathcal{CN}$, we have $Th_{ck}\models_{\mathsf{QCL}}\exists y(Sy\wedge (\forall x)_O(Exy \leftrightarrow st_y(\phi)))$.	
\end{lemma}
\begin{proof}
We can simply repeat the proof that we give for Lemma \ref{L:n4-comprehension} below, and observe that $Th$ is clearly equivalent to $Th_{ck}$ over $\mathsf{QCL}$. Our Lemma then follows by Lemma \ref{L:classical-inclusion}.
\end{proof}

\begin{lemma}\label{L:ck-mod-cl}
	Let $\mathcal{M}\in \mathbb{CK}$ and let $\mathcal{M}^{cl} \in \mathbb{C}$ be such that:
	\begin{itemize}
		\item $U^{\mathcal{M}^{cl}} := W \cup \mathcal{P}(W)$.
		
		\item  For every $p^1 \in Prop$, $p^{\mathcal{M}^{cl}} := V(p)$.
		
		\item $O^{\mathcal{M}^{cl}}:= W$.
		
		\item $S^{\mathcal{M}^{cl}}:= \mathcal{P}(W)$.
		
		\item $E^{\mathcal{M}^{cl}}:= \{(w, X)\in W\times\mathcal{P}(W)\mid w \in X\}$.
		
		\item $R^{\mathcal{M}^{cl}}:= \{(w, X,v)\in W\times\mathcal{P}(W)\times W\mid R(w, X, v)\}$.
	\end{itemize}
Then the following statements hold:
\begin{enumerate}
	\item $\mathcal{M}^{cl}\models_{c} Th_{ck}$.
	
	\item For every $\phi\in \mathcal{CN}$, every $x \in Ind$, and every $\mathbf{w} \in W$, we have $\mathcal{M}, \mathbf{w} \models_{ck} \phi$ iff $\mathcal{M}^{cl}\models_{c}st_x(\phi)[x/\mathbf{w}]$.
\end{enumerate}
\end{lemma}
\begin{proof}[Proof (a sketch)]
	Part 1 is straightforward (even though somewhat tedious) to check. As for Part 2, we argue by induction on the construction of  $\phi\in \mathcal{CN}$ in which we only consider a couple of cases.
	
	\textit{Basis}. Assume that $\phi = p$ where $p^1 \in Prop$. Then we have:
	$$
	\mathcal{M}, \mathbf{w} \models_{ck} \phi \text{ iff } \mathbf{w} \in V(p) \text{ iff } \mathbf{w} \in p^{\mathcal{M}^{cl}} \text{ iff } \mathcal{M}^{cl}\models_{c} p(x) = st_x(\phi)[x/\mathbf{w}].
	$$
	
	\textit{Induction step}. The Boolean cases are straightforward. In case $\phi = (\psi \boxto \chi)$ for some $\psi, \chi \in \mathcal{CN}$, and  $x,y,z, w \in Ind$ are pairwise distinct, we begin by noting that IH for $\psi$ implies that:
	\begin{equation}\label{E:ckcl1}
		\mathcal{M}^{cl}\models_{c} Sy \wedge (\forall z)_O (Ezy\leftrightarrow st_z(\psi))[y/\|\psi\|^{ck}_\mathcal{M}].
	\end{equation}
We now reason as follows:
\begin{align*}
	\mathcal{M}, \mathbf{w} \models_{ck} \phi &\text{ iff } (\forall \mathbf{v} \in W)(R_{\|\psi\|^{ck}_\mathcal{M}}(\mathbf{w},\mathbf{v})\text{ implies }\mathcal{M}, \mathbf{v}\models_{ck} \chi)\\
	&\text{ iff } (\forall \mathbf{v} \in W)(R^{\mathcal{M}^{cl}}(\mathbf{w},\|\psi\|^{ck}_\mathcal{M},\mathbf{v})\text{ implies }\mathcal{M}^{cl}\models_{c} st_{w}(\chi)[w/\mathbf{v}])&&\text{by IH}\\
	&\text{ iff } (\forall \mathbf{v} \in W\ \cup \mathcal{P}(W))(R^{\mathcal{M}^{cl}}(\mathbf{w},\|\psi\|^{ck}_\mathcal{M},\mathbf{v})\text{ implies }\mathcal{M}^{cl}\models_{c} st_{w}(\chi)[w/\mathbf{v}])&&\text{by def. of $R^{\mathcal{M}^{cl}}$}\\
	&\text{ iff } (\forall \mathbf{v} \in U^{\mathcal{M}^{cl}})(R^{\mathcal{M}^{cl}}(\mathbf{w},\|\psi\|^{ck}_\mathcal{M},\mathbf{v})\text{ implies }\mathcal{M}^{cl}\models_{c} st_{w}(\chi)[w/\mathbf{v}])\\
	&\text{ iff } \mathcal{M}^{cl}\models_{c} \forall w(Rxyw \to  st_{w}(\chi))[x/\mathbf{w}, y/\|\psi\|^{ck}_\mathcal{M}]
\end{align*}
The latter implies, by \eqref{E:ckcl1}, that
\begin{equation}\label{E:extras1}
	\mathcal{M}^{cl}\models_{c} \exists y(Sy \wedge (\forall z)_O (Ezy\leftrightarrow st_z(\psi))\wedge\forall w(Rxyw\to st_{w}(\chi)))[x/\mathbf{w}]
\end{equation}
in other words, that $\mathcal{M}^{cl}\models_{c} st_x(\psi\boxto\chi)[x/\mathbf{w}]$. Conversely, assuming \eqref{E:extras1}, we can find an $X \in U^{\mathcal{M}^{cl}}$ such that $X \in S^{\mathcal{M}^{cl}} = \mathcal{P}(W)$ and we have
\begin{equation}\label{E:extras2}
	\mathcal{M}^{cl}\models_{c} (\forall z)_O (Ezy\leftrightarrow st_z(\psi))\wedge\forall w(Rxyw\to st_{w}(\chi)))[x/\mathbf{w}, y/X]
\end{equation}
But then, by \eqref{E:ckcl1} and Part 1, it follows that $X = \|\psi\|^{ck}_\mathcal{M}$ so that $ \mathcal{M}^{cl}\models_{c} \forall w(Rxyw \to  st_{w}(\chi))[x/\mathbf{w}, y/\|\psi\|^{ck}_\mathcal{M}]$ follows. 	  
\end{proof}
Yet another lemma provides for the converse direction of Proposition \ref{P:ck-into-cl}:
\begin{lemma}\label{L:cl-mod-ck}
Let $\mathcal{M}\in \mathbb{C}$ be such that $\mathcal{M}\models_{c} Th_{ck}$, and let $\mathcal{M}^{ck} \in \mathbb{CK}$ be such that:
\begin{itemize}
	\item $W^{ck} := U^{\mathcal{M}}$.
	
	\item $R^{ck}(\mathbf{w}, X, \mathbf{v})\text{ iff } (\exists a \in S^{\mathcal{M}})(X \cap O^{\mathcal{M}} = \{b \in O^{\mathcal{M}}\mid E^{\mathcal{M}}(b,a)\}\text{ and }R^{\mathcal{M}}(\mathbf{w}, a, \mathbf{v}))$.
	
	\item For every $p^1 \in Prop$, $V^{ck}(p) := p^{\mathcal{M}}$.	
\end{itemize}
Then, for every $\phi\in \mathcal{CN}$, every $x \in Ind$, and every $\mathbf{w} \in W^{ck} = U^{\mathcal{M}}$, we have $\mathcal{M}^{ck}, \mathbf{w} \models_{ck} \phi$ iff $\mathcal{M}\models_{c}st_x(\phi)[x/\mathbf{w}]$.
\end{lemma}
\begin{proof}[Proof (a sketch)]
	It is clear that $\mathcal{M}^{ck} \in \mathbb{CK}$; as for the main statement of the Lemma, we proceed by induction on the construction of $\phi\in \mathcal{CN}$.
	
	\textit{Basis}. Assume that $\phi = p$ where $p^1 \in Prop$. Then we have:
	$$
	\mathcal{M}^{ck}, w \models_{ck} \phi = p \text{ iff } w \in V(p) \text{ iff } w \in p^{\mathcal{M}} \text{ iff } \mathcal{M}\models_{c} p(x) = st_x(\phi)[x/\mathbf{w}].
	$$
	\textit{Induction step}. The Boolean cases are straightforward by IH. If $\phi = (\psi \boxto \chi)$, assume  $x,y,z, w \in Ind$ to be pairwise distinct. By Lemma \ref{L:cl-comprehension}, we know that for some $a \in S^\mathcal{M}$ we have:
\begin{equation}\label{E:clck1}
	\mathcal{M}\models_{c} Sy \wedge (\forall z)_O (Ezy\leftrightarrow st_z(\psi))[y/a].
\end{equation}
For the right-to-left half of the Lemma, assume that $\mathcal{M}^{ck}, \mathbf{w} \models_{ck} \phi$. Then $(\forall \mathbf{v} \in W^{ck})(R^{ck}_{\|\psi\|^{ck}_{\mathcal{M}^{ck}}}(\mathbf{w},\mathbf{v})\text{ implies }\mathcal{M}^{ck}, \mathbf{v}\models_{ck} \chi)$. If now $\mathbf{v} \in U^\mathcal{M} = W^{ck}$ is such that $R^\mathcal{M}(\mathbf{w}, a, \mathbf{v})$, then, by IH and the choice of $a$, we must have $R^{ck}_{\|\psi\|^{ck}_{\mathcal{M}^{ck}}}(\mathbf{w},\mathbf{v})$, whence also $\mathcal{M}^{ck}, \mathbf{v}\models_{ck} \chi$, and, by IH, $\mathcal{M}\models_{c} st_{w}(\chi)[w/\mathbf{v}]$.
Thus we have shown that $\mathcal{M}\models_{c}\forall w(Rxyw\to st_{w}(\chi)))[x/\mathbf{w}, y/a]$, which, together with \eqref{E:clck1}, yields that $\mathcal{M}\models_{c} st_x(\psi\boxto\chi)[x/\mathbf{w}]$.

For the converse, assume that $\mathcal{M}\models_{c} st_x(\psi\boxto\chi)[x/w]$. Then there must exist a $b \in S^\mathcal{M}$ such that  both
\begin{equation}\label{E:clck2}
	\mathcal{M}\models_{c} Sy \wedge (\forall z)_O (Ezy\leftrightarrow st_z(\psi))[y/b]
\end{equation}
and
\begin{equation}\label{E:clck3}
	\mathcal{M}\models_{c} \forall w(Rxyw\to st_{w}(\chi))[x/\mathbf{w}, y/b]
\end{equation}
If now $\mathbf{v} \in W^{ck} = U^\mathcal{M}$ is such that $R^{ck}_{\|\psi\|^{ck}_{\mathcal{M}^{ck}}}(\mathbf{w},\mathbf{v})$, then we must have $R^\mathcal{M}(\mathbf{w},b,\mathbf{v})$ by IH and the choice of $a$, whence it follows that $\mathcal{M}\models_{c} st_{w}(\chi)[w/\mathbf{v}]$ by \eqref{E:clck3}. By IH, we must have $\mathcal{M}^{ck}, \mathbf{v}\models_{ck} \chi$. Thus we have shown that $\mathcal{M}^{ck}, \mathbf{w}\models_{ck} \psi\boxto\chi$.

\end{proof}
If we attempt to adapt the general criterion for seeking out non-classical counterparts to $\mathsf{K}$ that was formulated in the previous section, to the context of conditional logic, the following principle suggests itself: \textit{given a non-classical logic $Q\nu$ such that $Q\nu$ is a proper sublogic of $\mathsf{QCL}$ and its propositional fragment $\nu$ is a proper sublogic of $\mathsf{CL}$, a conditional extension $\nu\mathsf{CK}$ of $\nu$ is a $\nu$-based counterpart of $\mathsf{CK}$ iff some classically equivalent reformulation of $\sigma\tau_x$ faithfully embeds $\nu\mathsf{CK}$ into $\mathsf{Q}\nu(\Pi)$ modulo the assumption of some classically equivalent reformulation of $Th_{ck}$; the comparative merits of different $\nu$-based counterparts of $\mathsf{CK}$ will then have to be judged on the basis of other properties of $Q\nu$.}

Approaching with this principle to $\mathsf{QIL}$, we have defined an intuitionistic counterpart $\mathsf{IntCK}$ of $\mathsf{CK}$ in \cite{olkhovikov}, and we have shown, in \cite[Theorem 2]{olkhovikov}, that $\mathsf{IntCK}$ is an intuitionistic counterpart to $\mathsf{CK}$ according to the above criterion.\footnote{The theory that that is used in \cite{olkhovikov} is not classically equivalent to $Th_{ck}$; however, the proof given in \cite{olkhovikov} allows to cut this theory down to a subtheory that is classically equivalent to $Th_{ck}$. This discrepancy is just another example of polymorphism of first-order conditional theories that is further discussed in Section \ref{S:conclusion} below; cf. especially Corollaries \ref{C:foil} and \ref{C:ck-into-cl}.} Very conveniently, the reformulation $st^i_x$ of $st_x$ used in \cite{olkhovikov} is similar to the reformulation $\sigma\tau^i_x$ of $\sigma\tau_x$, described in the previous section in that the change consists in adding the standard translation of $\Diamond$ as a separate clause. Moreover, \cite[Proposition 4]{olkhovikov} shows that, if we also add a clause for diamonds to the mapping $\eta_\phi$ mentioned in the beginning of this section, then we get an embedding of $\mathsf{IK}$ into $\mathsf{IntCK}$. Thus, $\mathsf{K}$, $\mathsf{CK}$, $\mathsf{IK}$, and $\mathsf{IntCK}$ are tied to one another by natural embeddings forming, as it were, a tightly connected system. The main aim of the next section, and of this paper in general, is to extend this system with $\mathsf{N4}$-based modal and conditional logics.

\section{$\mathsf{N4CK}$, the basic Nelsonian conditional logic}\label{S:N4CK}
Another logic based on $\mathcal{CN}$ extends $\mathsf{N4}$ rather than $\mathsf{CL}$. It was introduced in \cite{nelsonian}, and its semantics is based on the following notion:
\begin{definition}\label{D:model}
	A \textit{Nelsonian conditional model} is a structure of the form $\mathcal{M} = (W, \leq, R, V^+, V^-)$ such that:
	\begin{enumerate}
		\item $W \neq \emptyset$ is the set of worlds, or nodes.
		
		\item $\leq$ is a pre-order on $W$.
		
		\item $V^+, V^-:Prop \to \mathcal{P}^\leq(W) = \{X\subseteq W\mid (\forall\mathbf{w}\in X)(\forall\mathbf{v}\geq\mathbf{w})(\mathbf{v}\in X)\}$.
		
		\item $R \subseteq W \times (\mathcal{P}(W)\times \mathcal{P}(W))\times W$ is the conditional accessibility relation. Thus, for all $X,Y \subseteq W$, $R$ induces a binary relation $R_{(X,Y)}$ on $W$ such that, for all $\mathbf{w},\mathbf{v} \in W$, $R_{(X,Y)}(\mathbf{w},\mathbf{v})$ iff $R(\mathbf{w}, (X,Y),\mathbf{v})$. Moreover, the following conditions must be satisfied for all $X, Y \subseteq W$:
		\begin{align}
			(\leq^{-1}\circ R_{(X,Y)}) &\subseteq (R_{(X,Y)}\circ\leq^{-1})\label{Cond:1}\tag{c1}\\
			(R_{(X,Y)}\circ\leq) &\subseteq (\leq\circ R_{(X,Y)})\label{Cond:2}\tag{c2}
		\end{align}	
	\end{enumerate}
The class of all Nelsonian conditional models will be denoted by $\mathbb{NC}$. 
\end{definition}
Conditions \eqref{Cond:1} and \eqref{Cond:2} can be reformulated as requirements to complete the dotted parts of each of the two diagrams given in Figure \ref{Fig:completion-patterns} once the respective straight-line part is given.
\begin{figure}
	\begin{center}
		\begin{tikzpicture}[scale=.7]
			\node (w1) at (-4,-1) {$w$};
			\node (w'1) at (-4,2) {$w'$};
			\node (v1) at (-1,-1) {$v$};
			\node (v'1) at (-1,2) {$v'$};
			\draw[-latex] (w1) to  node[midway, below] {$R_{(X,Y)}$} (v1);
			\draw[-latex] (w1) to node[midway, left] {$\leq$} (w'1);
			\draw[dotted, -latex] (w'1) to  node[midway, above] {$R_{(X,Y)}$} (v'1);
			\draw[dotted, -latex] (v1) to  node[midway, right] {$\leq$} (v'1);
			\node (w2) at (2,-1) {$w$};
			\node (w'2) at (2,2) {$w'$};
			\node (v2) at (5,-1) {$v$};
			\node (v'2) at (5,2) {$v'$};
			\draw[-latex] (w2) to  node[midway, below] {$R_{(X,Y)}$} (v2);
			\draw[dotted, -latex] (w2) to node[midway, left] {$\leq$} (w'2);
			\draw[dotted, -latex] (w'2) to  node[midway, above] {$R_{(X,Y)}$} (v'2);
			\draw[-latex] (v2) to  node[midway, right] {$\leq$} (v'2);
		\end{tikzpicture}
		\caption{Conditions \eqref{Cond:1} and \eqref{Cond:2}}\label{Fig:completion-patterns}
	\end{center}
\end{figure}
The evaluation points for $\mathsf{N4CK}$ are then defined as in $\mathsf{CK}$, in other words, we set $
EP:= \{(\mathcal{M}, \mathbf{w})\mid \mathcal{M}\in \mathbb{NC},\,\mathbf{w} \in W\}$.

Just as in other Nelsonian logics, we find in $\mathsf{N4CK}$ two satisfaction relations $\models^+$ and $\models^-$, representing verifications and falsifications of the conditional formulas, respectively. Their inductive definition runs as follows:
\begin{align*}
	\mathcal{M}, \mathbf{w}&\models^\star p \text{ iff } w \in V^\star(p)\qquad\qquad\text{for $p^1 \in Prop$ and $\star\in \{+, -\}$}\\
	\mathcal{M}, \mathbf{w}&\models^+ \sim\psi \text{ iff } \mathcal{M}, \mathbf{w}\models^- \psi\\
	\mathcal{M}, \mathbf{w}&\models^- \sim\psi \text{ iff } \mathcal{M}, \mathbf{w}\models^+ \psi\\
	\mathcal{M}, \mathbf{w}&\models^+ \psi \wedge \chi \text{ iff } \mathcal{M}, \mathbf{w}\models^+ \psi\text{ and }\mathcal{M}, \mathbf{w}\models^+ \chi\\
	\mathcal{M}, \mathbf{w}&\models^- \psi \wedge \chi \text{ iff } \mathcal{M}, \mathbf{w}\models^- \psi\text{ or }\mathcal{M}, \mathbf{w}\models^- \chi\\
	\mathcal{M}, \mathbf{w}&\models^+ \psi \vee \chi \text{ iff } \mathcal{M}, \mathbf{w}\models^+ \psi\text{ or }\mathcal{M}, \mathbf{w}\models^+ \chi\\
	\mathcal{M}, \mathbf{w}&\models^- \psi \vee \chi \text{ iff } \mathcal{M}, \mathbf{w}\models^- \psi\text{ and }\mathcal{M}, \mathbf{w}\models^- \chi\\
	\mathcal{M}, \mathbf{w}&\models^+ \psi \to \chi \text{ iff } (\forall \mathbf{v} \geq \mathbf{w})(\mathcal{M}, \mathbf{v}\models^+ \psi\Rightarrow\mathcal{M}, \mathbf{v}\models^+ \chi)\\
	\mathcal{M}, \mathbf{w}&\models^- \psi \to \chi \text{ iff } \mathcal{M}, \mathbf{w}\models^+ \psi\text{ and }\mathcal{M}, \mathbf{w}\models^- \chi\\
	\mathcal{M}, \mathbf{w}&\models^+ \psi \boxto \chi \text{ iff } (\forall \mathbf{v} \geq \mathbf{w})(\forall \mathbf{u} \in W)(R_{\|\psi\|_\mathcal{M}}(\mathbf{w},\mathbf{u}) \Rightarrow\mathcal{M}, \mathbf{u}\models^+ \chi)\\
	\mathcal{M}, \mathbf{w}&\models^- \psi \boxto \chi \text{ iff } (\exists \mathbf{u} \in W)(R_{\|\psi\|_\mathcal{M}}(\mathbf{w},\mathbf{u})\text{ and }\mathcal{M}, \mathbf{u}\models^- \chi)	
\end{align*}
where we assume, for any given $\phi \in \mathcal{CN}$, that:
$$
\|\phi\|_\mathcal{M}: = (\|\phi\|^+_\mathcal{M}, \|\phi\|^-_\mathcal{M}) = (\{\mathbf{w} \in W\mid \mathcal{M}, \mathbf{w}\models^+\phi\}, \{\mathbf{w} \in W\mid \mathcal{M}, \mathbf{w}\models^-\phi\}).
$$
To set up the semantics for $\mathsf{N4CK}$ it only remains to define that $\mathsf{N4CK}:= \mathbb{L}(EP, \models^+)$. 

The following Lemma is then straightforward to prove:
\begin{lemma}\label{L:n4ck-monotonicity}
	For every $\mathcal{M}\in \mathbb{N}4$, all $\mathbf{w},\mathbf{v}\in W$ such that $\mathbf{w}\leq\mathbf{v}$, for every $\ast\in \{+, -\}$, and for every $\phi \in \mathcal{CN}$, $\mathbf{w}\in \|\phi\|^\ast_\mathcal{M}$ implies $\mathbf{v}\in \|\phi\|^\ast_\mathcal{M}$.
\end{lemma}
Among other things, \cite{nelsonian} considered the question of a complete Hilbert-style axiomatization of $\mathsf{N4CK}$. It turns out (see \cite[Theorem 1]{nelsonian}) that we have $\mathsf{N4CK} = \mathfrak{N4CK}(\mathcal{CN})$ for $\mathfrak{N4CK} = \mathfrak{IL}^++(\eqref{E:a1}-\eqref{E:a6},\eqref{E:ax1}-\eqref{E:ax4};\eqref{E:RAbox},\eqref{E:RCbox1},\eqref{E:RCbox2})$, where we assume that:
\begin{align}
	((\phi \boxto \psi)\wedge(\phi \boxto \chi))&\Leftrightarrow(\phi \boxto (\psi \wedge \chi))\label{E:ax1}\tag{Ax1}\\
	(\sim(\phi\boxto\psi)\wedge (\phi \boxto \chi))&\to\sim(\phi \boxto(\psi\vee\sim\chi))\label{E:ax2}\tag{Ax2}\\
	((\phi \diamondto \psi)\to(\phi \boxto \chi))&\to(\phi \boxto (\psi \to \chi))\label{E:ax3}\tag{Ax3}\\
	\phi\boxto (\psi&\to\psi)\label{E:ax4}\tag{Ax4}\\
	\text{From }\phi\Leftrightarrow\psi &\text{ infer } (\phi\boxto\chi)\Leftrightarrow(\psi\boxto\chi)\label{E:RAbox}\tag{RA$\Box$}\\
	\text{From }\phi\leftrightarrow\psi &\text{ infer } (\chi\boxto\phi)\leftrightarrow(\chi\boxto\psi)\label{E:RCbox1}\tag{RC$\Box$1}\\
	\text{From }\sim\phi\leftrightarrow\sim\psi &\text{ infer } \sim(\chi\boxto\phi)\leftrightarrow\sim(\chi\boxto\psi)\label{E:RCbox2}\tag{RC$\Box$2}	
\end{align}
In what follows, we will write $\vdash$ to denote $\vdash_{\mathfrak{N4CK}}$.

The completeness of $\mathfrak{N4CK}$ relative to $\mathsf{N4CK}$ was shown in \cite{nelsonian} by the rather standard method of constructing a universal model; we would like to recall this construction in some detail as we plan to use it below. More precisely,  we say that, for any given $\Gamma, \Delta \subseteq \mathcal{CN}$, the pair $(\Gamma, \Delta)$ is maximal iff $\Gamma \not\vdash \Delta$ and $\Gamma \cup \Delta = \mathcal{CN}$. Our universal model can now be defined as follows
\begin{definition}\label{D:canonical-model}
	The structure $\mathcal{M}_c$ is the tuple $(W_c, \leq_c, R_c, V^+_c, V^-_c)$ is such that:
	\begin{itemize}
		\item $W_c:=\{(\Gamma, \Delta)\in \mathcal{P}(\mathcal{CN})\times\mathcal{P}(\mathcal{CN})\mid (\Gamma, \Delta)\text{ is maximal}\}$.
		
		\item $(\Gamma_0,\Delta_0)\leq_c(\Gamma_1,\Delta_1)$ iff $\Gamma_0\subseteq\Gamma_1$ for all $(\Gamma_0,\Delta_0),(\Gamma_1,\Delta_1)\in W_c$.
		
		\item For all $(\Gamma_0,\Delta_0),(\Gamma_1,\Delta_1)\in W_c$ and $X, Y \subseteq W_c$, we have $((\Gamma_0,\Delta_0),(X,Y),(\Gamma_1,\Delta_1)) \in R_c$ iff there exists a $\phi\in\mathcal{CN}$, such that all of the following holds:
		\begin{itemize}
			\item $X = \{(\Gamma,\Delta)\in W_c\mid\phi\in\Gamma\}$.
			
			\item $Y = \{(\Gamma,\Delta)\in W_c\mid\sim\phi\in\Gamma\}$
			
			\item $\{\psi\mid\phi\boxto\psi\in \Gamma_0\}\subseteq \Gamma_1$.
			
			\item $\{\sim(\phi\boxto\psi)\mid\sim\psi\in \Gamma_1\}\subseteq \Gamma_0$.
		\end{itemize}
		
		\item $V^+_c(p):=\{(\Gamma,\Delta)\in W_c\mid p\in\Gamma\}$ for every $p^1 \in Prop$.
		
		\item $V^-_c(p):=\{(\Gamma,\Delta)\in W_c\mid \sim p\in\Gamma\}$ for every $p^1 \in Prop$.	
	\end{itemize}
\end{definition}
It follows from \cite[Proposition 7]{nelsonian} that:
\begin{proposition}\label{P:truth-lemma}
	For every  $\phi\in\mathcal{CN}$, the following statements are true:
	\begin{enumerate}
		\item $\mathcal{M}_c \in \mathbb{NC}$.
		
		\item For every $(\Gamma, \Delta)\in W_c$, we have:
		\begin{enumerate}
			\item $
			\mathcal{M}_c,(\Gamma,\Delta)\models^+\phi$ iff $\phi \in \Gamma$.
			
			\item $
			\mathcal{M}_c,(\Gamma,\Delta)\models^-\phi$ iff $\sim\phi \in \Gamma$. 
		\end{enumerate}
	\end{enumerate}
\end{proposition}
The main result of this paper is that an analogue of Proposition \ref{P:ck-into-cl} can be proven for $\mathsf{N4CK}$ and $\mathsf{QN4}$, respectively, which considerably strengthens the claim that $\mathsf{N4CK}$ is the correct minimal conditional logic on the Nelsonian propositional basis. In order to do that, we first need to define the right classical equivalents for  $Th_{ck}$ and $st_x$, respectively. Since these equivalents are at least syntactically different from  $Th_{ck}$ and $st_x$, we will introduce for them separate notations, namely, $Th$ and $ST_x$.

As for the first component in this pair, we assume that $Th \subseteq \mathcal{FO}^\emptyset$ contains all and only the following first-order sentences:
\begin{align}
	&\forall x\forall y(Sx\wedge Sy \wedge (\forall z)_O(Ezx\Leftrightarrow Ezy) \to x \equiv y)\label{E:th5}\tag{Th1}\\
	&\exists x(Sx\wedge (\forall y)_O(Eyx\Leftrightarrow py))\qquad\qquad\qquad\qquad\qquad\qquad\qquad\qquad(p^1\in Prop)\label{E:th6}\tag{Th2}\\
	&\forall x(Sx\rightarrow\exists y(Sy\wedge (\forall z)_O(Ezy\Leftrightarrow \sim Ezx))\label{E:th7}\tag{Th3}\\
	&\forall x\forall y((Sx\wedge Sy)\rightarrow\exists z(Sz\wedge (\forall w)_O(Ewz\Leftrightarrow (Ewx\ast Ewy))))\qquad(\ast\in\{\wedge, \to\})\label{E:th8}\tag{Th4}\\
	&\forall x\forall y((Sx\wedge Sy)\rightarrow\exists z(Sz\wedge (\forall w)_O(Ewz\Leftrightarrow \forall u(Rwxu\to Euy))))\label{E:th9}\tag{Th5}
\end{align}
It is easy to see that $Th$ is classically equivalent to $Th_{ck}$. The only difference between the two theories is the replacement of equivalences encoding the set extensions in $Th_{ck}$ with the strong equivalences which allows for a replacement of a definable set with its definition in $\mathsf{QN4}$. In view of the considerations given in Appendix \ref{App:ampersand}, one could also argue for replacing conjunctions with ampersands in \eqref{E:th6}--\eqref{E:th9}; however, this is not necessary, since $Th$ only imposes the truth of its components and never says anything about their falsity conditions. We have seen, however, that the differences between $\wedge$ and $\&$ only become apparent when the treatment of the falsity component of a restricted existential quantification comes into question. Therefore, the distinction between $\wedge$ and $\&$ can no longer be overlooked when it comes to the definition of the standard translation which is supposed to faithfully represent both the truth and the falsity conditions of a translated formula.  As a result, we obtain, for every given $x \in Ind$, the standard translation $ST_x:\mathcal{CN}\to \mathcal{FO}^x$ defined by the following induction on the construction of a $\phi \in \mathcal{CN}$ (assuming the same choice of $x,y,z,w$ as in $st_x$):
\begin{align*}
	ST_x(p)&:= p(x)&&p^1 \in Prop\\
	ST_x(\sim\psi)&:= \sim ST_x(\psi)\\
	ST_x(\psi\ast\chi)&:= ST_x(\psi)\ast ST_x(\chi)&&\ast\in \{\wedge, \vee, \to\}\\
	ST_x(\psi\boxto\chi) &:= \exists y(Sy \wedge (\forall z)_O (Ezy\Leftrightarrow ST_z(\psi))\,\&\,\forall w(Rxyw\to ST_{w}(\chi)))
\end{align*}
The following proposition provides for the `easy' direction of the faithfulness claim:
\begin{proposition}\label{P:easy}
	Let $\Gamma, \Delta \subseteq \mathcal{CN}$, and let $x \in Ind$. If $\Gamma\models_{\mathsf{N4CK}} \Delta$, then $Th\cup ST_x(\Gamma)\models_{\mathsf{QN4}} ST_x(\Delta)$.
\end{proposition}
Its proof is analogous to the proof of the corresponding half of Proposition \ref{P:ck-into-cl} given in the previous section. We start by proving the following technical lemmas:
\begin{lemma}\label{L:n4-comprehension}
	For every $\phi \in \mathcal{CN}$, we have $Th\models_{\mathsf{QN4}}\exists y(Sy\wedge (\forall x)_O(Exy \Leftrightarrow ST_x(\phi)))$.	
\end{lemma}
The proof of the Lemma is relegated to Appendix \ref{App:n4-comprehension}.

\begin{definition}\label{D:n4-mod-n4ck}
	Given an $\mathcal{S}\in \mathbb{N}4$, we define $\mathcal{M}_{\mathcal{S}}$ as follows:
	\begin{enumerate}
		\item $W_{\mathcal{S}} := \{(\mathbf{w}, a)\mid \mathbf{w} \in W,\,a\in U_\mathbf{w}\}$.
		
		\item $(\mathbf{w},a)\leq_{\mathcal{S}}(\mathbf{v},b)$ iff $\mathbf{w}\leq \mathbf{v}$ and $\mathtt{H}_{\mathbf{w}\mathbf{v}}(a) = b$.
		
		\item $R_{\mathcal{S}}((\mathbf{w},a), (X, Y), (\mathbf{v},b))$ iff $\mathbf{w} = \mathbf{v}$ and 
		\begin{align*}
				(\exists c \in S^{\mathtt{M}^+_\mathbf{w}})&(\forall \mathbf{u} \geq \mathbf{w})(\forall d \in O^{\mathtt{M}^+_\mathbf{u}})(((\mathbf{u},d)\in X\text{ iff }E^{\mathtt{M}^+_\mathbf{u}}(d,\mathtt{H}_{\mathbf{w}\mathbf{u}}(c)))\\
				&\text{ and }((\mathbf{u},d)\in Y\text{ iff }E^{\mathtt{M}^-_\mathbf{u}}(d,\mathtt{H}_{\mathbf{w}\mathbf{u}}(c)))\text{ and }R^{\mathtt{M}^+_\mathbf{w}}(a,c,b)).
		\end{align*}
	\item For every $p^1 \in Prop$, $V^+_\mathcal{S}(p) := \{(\mathbf{w},a)\in W_\mathcal{S}\mid a \in p^{\mathtt{M}^+_\mathbf{w}}\}$ and $V^-_\mathcal{S}(p) := \{(\mathbf{w},a)\in W_\mathcal{S}\mid a \in p^{\mathtt{M}^-_\mathbf{w}}\}$.
	\end{enumerate}
\end{definition}
In order to avoid the clutter, we introduce the notation $\Xi((\mathbf{w},c), (X,Y))$ for  the main bi-conditional in Definition \ref{D:n4-mod-n4ck}.3; this part of the definition can then be re-written as
$$
R_{\mathcal{S}}((\mathbf{w},a), X, Y, (\mathbf{v},b))\text{ iff }\mathbf{w} = \mathbf{v}\text{ and }(\exists c \in S^{\mathtt{M}^+_\mathbf{w}})(\Xi((\mathbf{w},c), (X,Y))\text{ and }R^{\mathtt{M}^+_\mathbf{w}}(a,c,b)).
$$
\begin{lemma}\label{L:n4-mod-n4ck1}
	Let $\mathcal{S}\in \mathbb{N}4$ be such that $\mathcal{S}\models^+_{n} Th$, and let $\mathcal{M}_{\mathcal{S}}$ be given according to Definition \ref{D:n4-mod-n4ck}. Then  $\mathcal{M}_{\mathcal{S}}\in \mathbb{NC}$.
\end{lemma}
\begin{proof}
	The transitivity and reflexivity of $\leq_{\mathcal{S}}$ easily follow from the same properties of $\leq$. As for the monotonicity of $V^+_\mathcal{S}$ relative to $\leq_\mathcal{S}$, assume that $p^1 \in Prop$, and that $(\mathbf{w},a), (\mathbf{v},b) \in W_\mathcal{S}$ are such that $(\mathbf{w},a)\leq_{\mathcal{S}}(\mathbf{v},b)$. Then $\mathbf{w} \leq \mathbf{v}$ and  $\mathtt{H}_{\mathbf{w}\mathbf{v}}(a) = b$. If now $(\mathbf{w},a) \in V^+_\mathcal{S}(p)$, then $a \in p^{\mathtt{M}^+_\mathbf{w}}$, and, since $\mathtt{H}_{\mathbf{w}\mathbf{v}}\in Hom(\mathtt{M}^+_\mathbf{w}, \mathtt{M}^+_\mathbf{v})$, also $b = \mathtt{H}_{\mathbf{w}\mathbf{v}}(a) \in p^{\mathtt{M}^+_\mathbf{v}}$, which, in turn, means that $(\mathbf{v},b) \in V^+_\mathcal{S}(p)$. We argue similarly for the monotonicity of $V^-_\mathcal{S}$.
	
	It only remains to check the conditions imposed on $R$. As for \eqref{Cond:1}, assume that $(\mathbf{w},a), (\mathbf{v},b), (\mathbf{u},c) \in W_\mathcal{S}$ and $X, Y \subseteq  W_\mathcal{S}$ are such that both $(\mathbf{w},a)\leq_{\mathcal{S}}(\mathbf{v},b)$ and $R_{\mathcal{S}}((\mathbf{w},a), X, Y, (\mathbf{u},c))$. Then $\mathbf{w} = \mathbf{u}$ and $\mathbf{w} \leq \mathbf{v}$, whence $(\mathbf{u},c) = (\mathbf{w},c)\leq_{\mathcal{S}}(\mathbf{v},\mathtt{H}_{\mathbf{w}\mathbf{v}}(c))\in  W_\mathcal{S}$. On the other hand, $\mathtt{H}_{\mathbf{w}\mathbf{v}}(a) = b$, and there exists a $d \in S^{\mathtt{M}^+_\mathbf{w}}$ such that both $R^{\mathtt{M}^+_\mathbf{w}}(a,d,c)$ and, for all $\mathbf{w}'\geq \mathbf{w}$ and all $g \in  O^{\mathtt{M}^+_{\mathbf{w}'}}$ we have:
	\begin{equation}\label{E:n4tock1}
		((\mathbf{w}',g)\in X\text{ iff }E^{\mathtt{M}^+_{\mathbf{w}'}}(g,\mathtt{H}_{\mathbf{w}\mathbf{w}'}(d)))
		\text{ and }((\mathbf{w}',g)\in Y\text{ iff }E^{\mathtt{M}^-_{\mathbf{w}'}}(g,\mathtt{H}_{\mathbf{w}\mathbf{w}'}(d)))
	\end{equation}
	Now, consider $\mathtt{H}_{\mathbf{w}\mathbf{v}}(d)\in S^{\mathtt{M}^+_\mathbf{v}}$. Since $\mathtt{H}_{\mathbf{w}\mathbf{v}}\in Hom(\mathtt{M}^+_\mathbf{w}, \mathtt{M}^+_\mathbf{v})$, we must have that $R^{\mathtt{M}^+_\mathbf{v}}(\mathtt{H}_{\mathbf{w}\mathbf{v}}(a),\mathtt{H}_{\mathbf{w}\mathbf{v}}(d),\mathtt{H}_{\mathbf{w}\mathbf{v}}(c))$, in other words, that $R^{\mathtt{M}^+_\mathbf{v}}(b,\mathtt{H}_{\mathbf{w}\mathbf{v}}(d),\mathtt{H}_{\mathbf{w}\mathbf{v}}(c))$. Moreover, if $\mathbf{v}'\geq \mathbf{v}$ and $k\in O^{\mathtt{M}^+_{\mathbf{v}'}}$, then transitivity implies $\mathbf{v}'\geq \mathbf{w}$, and Definition \ref{D:nelsonian-sheaf} implies that $\mathtt{H}_{\mathbf{v}\mathbf{v}'}(\mathtt{H}_{\mathbf{w}\mathbf{v}}(d)) = \mathtt{H}_{\mathbf{w}\mathbf{v}'}(d)$, whence, by \eqref{E:n4tock1}, we must have:
	\begin{equation*}
		((\mathbf{v}',k)\in X\text{ iff }E^{\mathtt{M}^+_{\mathbf{v}'}}(k,\mathtt{H}_{\mathbf{v}\mathbf{v}'}(\mathtt{H}_{\mathbf{w}\mathbf{v}}(d))))
		\text{ and }((\mathbf{v}',k)\in Y\text{ iff }E^{\mathtt{M}^-_{\mathbf{v}'}}(k,\mathtt{H}_{\mathbf{v}\mathbf{v}'}(\mathtt{H}_{\mathbf{w}\mathbf{v}}(d))))
	\end{equation*}
	Therefore, we must also have $R_{\mathcal{S}}((\mathbf{v},b), X, Y, (\mathbf{v}, \mathtt{H}_{\mathbf{w}\mathbf{v}}(c)))$, and \eqref{Cond:1} is shown to hold.
	
	The argument for \eqref{Cond:2} is similar.
\end{proof}
Since Definition \ref{D:n4-mod-n4ck}.3 is relatively involved, we look a bit deeper into its consequences, before proceeding further:
\begin{lemma}\label{L:n4-mod-n4ck-aux}
	Let $\mathcal{S}\in \mathbb{N}4$ and let $\phi \in \mathcal{CN}$. Then the following statements hold:
\begin{enumerate}
	\item Assume that, for every $\star \in \{+, -\}$, every $x \in Ind$, and every $(\mathbf{w},a) \in W_\mathcal{S}$, we have $\mathcal{M}_\mathcal{S}, (\mathbf{w},a) \models^\star \phi$ iff $\mathcal{S}, \mathbf{w}\models^\star_{n}ST_x(\phi)[x/a]$. Then, for every $(\mathbf{v}, b)\in W_\mathcal{S}$ and all distinct $y,z \in Ind$ we have 
	$$
	\mathcal{S}, \mathbf{v}\models^+_n (\forall z)_O(Ezy \Leftrightarrow ST_z(\phi))[y/b] \text{ iff }\Xi((\mathbf{v},b), \|\phi\|_{\mathcal{M}_\mathcal{S}}).
	$$
	\item For all $(\mathbf{w},a), (\mathbf{v}, b) \in W_\mathcal{S}$, if $(\mathbf{w},a)\leq_{\mathcal{S}}(\mathbf{v}, b)$ and $\Xi((\mathbf{w},a), \|\phi\|_{\mathcal{M}_\mathcal{S}})$, then  $\Xi((\mathbf{v},b), \|\phi\|_{\mathcal{M}_\mathcal{S}})$.
\end{enumerate}	
\begin{proof}
	(Part 1) Assume the premises. We have $\mathcal{S}, \mathbf{v}\models^+_n (\forall z)_O(Ezy \Leftrightarrow ST_z(\phi))[y/b]$ iff, by Corollary \ref{C:set-encoding}, we have:
	\begin{align*}
		(\forall \mathbf{u} \geq \mathbf{v})(\forall d \in O^{\mathtt{M}^+_\mathbf{u}})&((\mathcal{S},\mathbf{u}\models^+_n ST_z(\phi)[z/d]\text{ iff }\mathcal{S},\mathbf{u}\models^+_n Ezy[y/\mathtt{H}_{\mathbf{w}\mathbf{u}}(b), z/d])\notag\\
		&\text{ and }(\mathcal{S},\mathbf{u}\models^-_n ST_z(\phi)[z/d]\text{ iff }\mathcal{S},\mathbf{u}\models^-_n Ezy[y/\mathtt{H}_{\mathbf{w}\mathbf{u}}(b), z/d]))
	\end{align*}
	By our hypothesis about $\phi$, the latter equation is equivalent to 
	\begin{align*}
		(\forall \mathbf{u} \geq \mathbf{v})(\forall d \in O^{\mathtt{M}^+_\mathbf{u}})(((\mathbf{u}, z)&\in \|\phi\|^+_{\mathcal{M}_\mathcal{S}}\text{ iff } E^{\mathtt{M}^+_\mathbf{u}}(d,\mathtt{H}_{\mathbf{w}\mathbf{u}}(b)))\notag\\
		&\text{ and }((\mathbf{u}, z)\in \|\phi\|^-_{\mathcal{M}_\mathcal{S}}\text{ iff }E^{\mathtt{M}^-_\mathbf{u}}(d,\mathtt{H}_{\mathbf{w}\mathbf{u}}(b))))
	\end{align*}
that is to say, to $\Xi((\mathbf{v},b), \|\phi\|_{\mathcal{M}_\mathcal{S}})$. 

(Part 2) Assume that $(\mathbf{w},a)\leq_{\mathcal{S}}(\mathbf{v}, b)$; then $\mathbf{w} \leq \mathbf{v}$ and $\mathtt{H}_{\mathbf{w}\mathbf{v}}(a) = b$. If now $\Xi((\mathbf{w},a), \|\phi\|_{\mathcal{M}_\mathcal{S}})$, this means, by Part 1, that $\mathcal{S}, \mathbf{w}\models^+_n (\forall z)_O(Ezy \Leftrightarrow ST_z(\phi))[y/a]$, whence, by Lemma \ref{L:n4-standard}.1, $\mathcal{S}, \mathbf{v}\models^+_n (\forall z)_O(Ezy \Leftrightarrow ST_z(\phi))[y/\mathtt{H}_{\mathbf{w}\mathbf{v}}(a)]$, whence, again by Part 1, $\Xi((\mathbf{v},b), \|\phi\|_{\mathcal{M}_\mathcal{S}})$. 
\end{proof}	
\end{lemma}
\begin{lemma}\label{L:n4-mod-n4ck2}
	Let $\mathcal{S}\in \mathbb{N}4$ be such that $\mathcal{S}\models^+_{n} Th$, and let $\mathcal{M}_{\mathcal{S}}\in \mathbb{NC}$ be given according to Definition \ref{D:n4-mod-n4ck}. Then, for every $\phi\in \mathcal{CN}$, $\star \in \{+, -\}$, $x \in Ind$, and $(\mathbf{w},a) \in W_\mathcal{S}$, we have $\mathcal{M}_\mathcal{S}, (\mathbf{w},a) \models^\star \phi$ iff $\mathcal{S}, \mathbf{w}\models^\star_{n}ST_x(\phi)[x/a]$.	
\end{lemma}
\begin{proof}
 We argue by induction on the construction of $\phi\in \mathcal{CN}$. 

\textit{Basis}. If $\phi = p$ for some $p^1 \in Prop$, then we have, for every $\star \in \{+, -\}$ and $(\mathbf{w},a) \in W_\mathcal{S}$, that
$\mathcal{M}_\mathcal{S}, (\mathbf{w},a) \models^\star p$ iff $(\mathbf{w},a)\in V^\star_\mathcal{S}(p)$ iff $a \in p^{\mathtt{M}^\star_\mathbf{w}}$ iff $\mathcal{S}, \mathbf{w}\models^\star_n p(x)[x/a]$.

\textit{Induction step}. The following cases arise:

\textit{Case 1}. If $\phi = \psi \wedge \chi$, then we have, on the one hand:
\begin{align*}
	\mathcal{M}_\mathcal{S}, (\mathbf{w},a) \models^+ \psi \wedge \chi &\text{ iff } \mathcal{M}_\mathcal{S}, (\mathbf{w},a) \models^+ \psi\text{ and }\mathcal{M}_\mathcal{S}, (\mathbf{w},a) \models^+\chi&&\text{ by IH}\\
	&\text{ iff }  \mathcal{S}, \mathbf{w}\models^+_n ST_x(\psi)[x/a]\text{ and } \mathcal{S}, \mathbf{w}\models^+_n ST_x(\chi)[x/a]\\
	&\text{ iff }  \mathcal{S}, \mathbf{w}\models^+_n ST_x(\psi)\wedge 
	ST_x(\chi)[x/a]\\
	&\text{ iff }  \mathcal{S}, \mathbf{w}\models^+_n ST_x(\psi\wedge\chi)[x/a]
\end{align*}
On the other hand, we have:
\begin{align*}
	\mathcal{M}_\mathcal{S}, (\mathbf{w},a) \models^- \psi \wedge \chi &\text{ iff } \mathcal{M}_\mathcal{S}, (\mathbf{w},a) \models^- \psi\text{ or }\mathcal{M}_\mathcal{S}, (\mathbf{w},a) \models^-\chi\\
	&\text{ iff }  \mathcal{S}, \mathbf{w}\models^-_n ST_x(\psi)[x/a]\text{ or } \mathcal{S}, \mathbf{w}\models^-_n ST_x(\chi)[x/a]&&\text{ by IH}\\
	&\text{ iff }  \mathcal{S}, \mathbf{w}\models^-_n ST_x(\psi)\wedge 
	ST_x(\chi)[x/a]\\
	&\text{ iff }  \mathcal{S}, \mathbf{w}\models^-_n ST_x(\psi\wedge\chi)[x/a]
\end{align*}

\textit{Case 2} and \textit{Case 3}, where we assume that $\phi = \psi \vee \chi$ and $\phi = \sim\psi$, respectively, are solved similarly to Case 1.

\textit{Case 4}. If $\phi = \psi \to \chi$, then assume that $	\mathcal{M}_\mathcal{S}, (\mathbf{w},a) \models^+ \psi \to \chi$, and let $\mathbf{v}\geq\mathbf{w}$ be such that $\mathcal{S}, \mathbf{v}\models^+_n ST_x(\psi)[x/\mathtt{H}_{\mathbf{w}\mathbf{v}}(a)]$. Then, by IH, $\mathcal{M}_\mathcal{S}, (\mathbf{v},\mathtt{H}_{\mathbf{w}\mathbf{v}}(a)) \models^+ \psi$, and we also have $(\mathbf{w},a)\leq_{\mathcal{S}}(\mathbf{v},\mathtt{H}_{\mathbf{w}\mathbf{v}}(a))$, whence, by our assumption, $\mathcal{M}_\mathcal{S}, (\mathbf{v},\mathtt{H}_{\mathbf{w}\mathbf{v}}(a)) \models^+ \chi$. Again by IH, $\mathcal{S}, \mathbf{v}\models^+_n ST_x(\chi)[x/\mathtt{H}_{\mathbf{w}\mathbf{v}}(a)]$. Since $\mathbf{v}\geq\mathbf{w}$ was chosen arbitrarily, we conclude that
\begin{equation}\label{E:assumption}
	\mathcal{S}, \mathbf{w}\models^+_n (ST_x(\psi)\to ST_x(\chi)) = ST_x(\psi\to\chi)[x/a]
\end{equation}

In the other direction, assume that \eqref{E:assumption} holds. If $(\mathbf{v},b) \in W_\mathcal{S}$ is such that $(\mathbf{w},a)\leq_{\mathcal{S}}(\mathbf{v},b)$, then $\mathbf{w} \leq \mathbf{v}$ and $\mathtt{H}_{\mathbf{w}\mathbf{v}}(a) = b$. If, moreover, we have $\mathcal{M}_\mathcal{S}, (\mathbf{v},b) \models^+ \psi$, then, by IH, $\mathcal{S}, \mathbf{v}\models^+_n ST_x(\psi)[x/\mathtt{H}_{\mathbf{w}\mathbf{v}}(a)]$, whence, by \eqref{E:assumption}, $\mathcal{S}, \mathbf{v}\models^+_n ST_x(\chi)[x/\mathtt{H}_{\mathbf{w}\mathbf{v}}(a)]$. Applying IH again, we get that $\mathcal{M}_\mathcal{S}, (\mathbf{v},\mathtt{H}_{\mathbf{w}\mathbf{v}}(a)) \models^+ \psi$, in other words, that $\mathcal{M}_\mathcal{S}, (\mathbf{v},b) \models^+ \psi$. By the choice of $(\mathbf{v},b) \in W_\mathcal{S}$, we obtain that $	\mathcal{M}_\mathcal{S}, (\mathbf{w},a) \models^+ \psi \to \chi$, as desired.

Finally, observe that we have:
\begin{align*}
	\mathcal{M}_\mathcal{S}, (\mathbf{w},a) \models^- \psi \to \chi &\text{ iff } \mathcal{M}_\mathcal{S}, (\mathbf{w},a) \models^+ \psi\text{ and }\mathcal{M}_\mathcal{S}, (\mathbf{w},a) \models^-\chi\\
	&\text{ iff }  \mathcal{S}, \mathbf{w}\models^+_n ST_x(\psi)[x/a]\text{ and } \mathcal{S}, \mathbf{w}\models^-_n ST_x(\chi)[x/a]&&\text{ by IH}\\
	&\text{ iff }  \mathcal{S}, \mathbf{w}\models^-_n ST_x(\psi)\to 
	ST_x(\chi)[x/a]\\
	&\text{ iff }  \mathcal{S}, \mathbf{w}\models^-_n ST_x(\psi\to\chi)[x/a]
\end{align*}

\textit{Case 5}. If $\phi = \psi \boxto \chi$, then assume that
\begin{equation}\label{E:assum1}
	\mathcal{M}_\mathcal{S}, (\mathbf{w},a) \models^+ \psi \boxto \chi
\end{equation}
By Lemma \ref{L:n4-comprehension}, we can choose a $b \in U_\mathbf{w}$ such that:
\begin{equation}\label{E:assum2}
	\mathcal{S}, \mathbf{w}\models^+_n Sy\wedge (\forall z)_O(Ezy \Leftrightarrow ST_z(\psi))[y/b]
\end{equation}
Moreover, choose any $\mathbf{v} \in W$ and any $c \in U_\mathbf{v}$ such that:
\begin{align}
	&\mathbf{v}\geq\mathbf{w}\label{E:asm0}\\
	&\mathcal{S},\mathbf{v}\models^+_n R(x,y,w)[x/\mathtt{H}_{\mathbf{w}\mathbf{v}}(a), y/\mathtt{H}_{\mathbf{w}\mathbf{v}}(b), w/c]\label{E:asm0a}
\end{align}
We now reason as follows:
\begin{align}
	&(\mathbf{w},b)\leq_\mathcal{S}(\mathbf{v},\mathtt{H}_{\mathbf{w}\mathbf{v}}(b))\label{E:asm1a}&&\text{ by }\eqref{E:asm0}\\
	&b \in S^{\mathtt{M}^+_\mathbf{w}}\label{E:asm1}&&\text{ by }\eqref{E:assum2}\\
	&\Xi((\mathbf{w},b), \|\psi\|_{\mathcal{M}_\mathcal{S}})\label{E:asm2}&&\text{ by }\eqref{E:assum2}, Lm. \ref{L:n4-mod-n4ck-aux}.1,\text{ IH for }\psi\\
	&\Xi((\mathbf{v},\mathtt{H}_{\mathbf{w}\mathbf{v}}(b)), \|\psi\|_{\mathcal{M}_\mathcal{S}})\label{E:asm3}&&\text{ by }\eqref{E:asm0},\eqref{E:asm2}, Lm. \ref{L:n4-mod-n4ck-aux}.2\\
	&\mathtt{H}_{\mathbf{w}\mathbf{v}}(b)\in S^{\mathtt{M}^+_\mathbf{v}}\label{E:asm4}&&\text{ by }\eqref{E:asm1}, \mathtt{H}_{\mathbf{w}\mathbf{v}}\in Hom(\mathtt{M}^+_\mathbf{w}, \mathtt{M}^+_\mathbf{v})\\
	&R_\mathcal{S}((\mathbf{v},\mathtt{H}_{\mathbf{w}\mathbf{v}}(a)), \|\psi\|_{\mathcal{M}_\mathcal{S}}, (\mathbf{v},c))\label{E:asm6}&&\text{ by }\eqref{E:asm3}, \eqref{E:asm4}, \eqref{E:asm0a}\\
	&(\mathbf{w},a)\leq_\mathcal{S}(\mathbf{v},\mathtt{H}_{\mathbf{w}\mathbf{v}}(a))\label{E:asm7}&&\text{ by }\eqref{E:asm0}\\
	&\mathcal{M}_\mathcal{S}, (\mathbf{v},c)\models^+ \chi\label{E:asm8}&&\text{ by }\eqref{E:assum1}, \eqref{E:asm6}, \eqref{E:asm7}\\
	&\mathcal{S}, \mathbf{v}\models^+_n ST_w(\chi)[w/c]\label{E:asm9}&&\text{ by }\eqref{E:asm8},\text{ IH for }\chi
\end{align}
By the choice of $\mathbf{v}$ and $c$, we obtain that:
\begin{equation}\label{E:assum6}
	\mathcal{S}, \mathbf{w}\models^+_n \forall w(R(x,y,w)\to ST_w(\chi))[x/a, y/b]
\end{equation}
The latter equation, together with \eqref{E:assum2}, allows us to conclude that
\begin{equation}\label{E:assum7}
		\mathcal{S}, \mathbf{w}\models^+_n ST_x(\psi\boxto\chi)[x/a]
\end{equation}
For the converse, assume \eqref{E:assum7} and choose any $b \in U_\mathbf{w}$ such that both \eqref{E:assum2} and \eqref{E:assum6} hold. Consider any 
 $(\mathbf{v},c),(\mathbf{u},d) \in W_\mathcal{S}$ such that
\begin{align}
	&(\mathbf{w},a)\leq_\mathcal{S}(\mathbf{v},c)\label{E:asmp1}\\
	&R_\mathcal{S}((\mathbf{v},c,), \|\psi\|_{\mathcal{M}_\mathcal{S}}, (\mathbf{u},d))\label{E:asmp2}
\end{align} 
By \eqref{E:asmp2} and Definition \ref{D:n4-mod-n4ck}.3, there must be a $e \in U_\mathbf{v}$ such that all of the following holds
\begin{align}
&e \in S^{\mathtt{M}^+_\mathbf{v}}\label{E:asmp3}\\
&\Xi((\mathbf{v}, e), \|\psi\|_{\mathcal{M}_\mathcal{S}})\label{E:asmp4}\\
&R^{\mathtt{M}^+_\mathbf{v}}(c, e, d)\label{E:asmp5}
\end{align}
We now reason as follows:
\begin{align}
	&\mathcal{S}, \mathbf{v}\models^+_n  (\forall z)_O(Ezw \Leftrightarrow ST_z(\psi))[w/e]\label{E:asmp6}&&\text{ by \eqref{E:asmp4}, Lm \ref{L:n4-mod-n4ck-aux}.1, IH}\\
	&\mathcal{S}, \mathbf{v}\models^+_n  (\forall z)_O(Ezy \Leftrightarrow ST_z(\psi))[y/\mathtt{H}_{\mathbf{w}\mathbf{v}}(b)]\label{E:asmp7}&&\text{ by Lm \ref{L:n4-standard}.1, \eqref{E:assum2}}\\
	&\mathcal{S}, \mathbf{v}\models^+_n  (\forall z)_O(Ezw \Leftrightarrow Ezy)[y/\mathtt{H}_{\mathbf{w}\mathbf{v}}(b), w/e]\label{E:asmp7a}&&\text{ by \eqref{E:asmp6}, \eqref{E:asmp7}, \eqref{E:T18}}\\
	&e = \mathtt{H}_{\mathbf{w}\mathbf{v}}(b)\label{E:asmp8}&&\text{ by }\eqref{E:th5}, \eqref{E:asmp7a}\\
	&c = \mathtt{H}_{\mathbf{w}\mathbf{v}}(a)\label{E:asmp9}&&\text{ by }\eqref{E:asmp1}\\
	&\mathcal{S}, \mathbf{v}\models^+_n R(x,y,w)[x/\mathtt{H}_{\mathbf{w}\mathbf{v}}(a), y/\mathtt{H}_{\mathbf{w}\mathbf{v}}(b), w/d]\label{E:asmp10}&&\text{ by }\eqref{E:asmp5}, \eqref{E:asmp8}, \eqref{E:asmp9}\\
	&\mathcal{S}, \mathbf{v}\models^+_n ST_w(\chi)[w/d]\label{E:asmp12}&&\text{ by }\eqref{E:asmp10}, \eqref{E:assum6}\\
	&\mathcal{M}_\mathcal{S}, (\mathbf{v}, d)\models^+ \chi\label{E:asmp13}&&\text{ by IH, }\eqref{E:asmp12}\\
	&\mathbf{v} = \mathbf{u}\label{E:asmp14}&&\text{ by \eqref{E:asmp2}, Df \ref{D:n4-mod-n4ck}.3}\\
	&\mathcal{M}_\mathcal{S}, (\mathbf{u}, d)\models^+ \chi\label{E:asmp15}&&\text{ by }\eqref{E:asmp13}, \eqref{E:asmp14}
\end{align}
 Since $(\mathbf{v},c),(\mathbf{u},d) \in W_\mathcal{S}$ were chosen arbitrarily under the conditions \eqref{E:asmp1} and \eqref{E:asmp2}, it follows that $\mathcal{M}_\mathcal{S}, (\mathbf{w},a) \models^+ \psi \boxto \chi$, as desired.
 
 Next, assume that 
 \begin{equation}\label{E:hyp1}
 	\mathcal{M}_\mathcal{S}, (\mathbf{w},a) \models^- \psi \boxto \chi
 \end{equation}
This means that, for some $(\mathbf{v}, b) \in W_\mathcal{S}$, we have:
\begin{align}
	&R_\mathcal{S}((\mathbf{w},a), \|\psi\|_{\mathcal{M}_\mathcal{S}}, (\mathbf{v},b))\label{E:hyp2}\\
	&\mathcal{M}_\mathcal{S}, (\mathbf{v},b) \models^- \chi\label{E:hyp3}
\end{align}
By \eqref{E:hyp2}, we can choose a $c \in U_\mathbf{w}$ such that all of the following holds
\begin{align}
	&c \in S^{\mathtt{M}^+_\mathbf{w}}\label{E:hyp4}\\
	&\Xi((\mathbf{w}, c), \|\psi\|_{\mathcal{M}_\mathcal{S}})\label{E:hyp5}\\
	&R^{\mathtt{M}^+_\mathbf{w}}(a, c, b)\label{E:hyp6}
\end{align}
Assume, next, that some $\mathbf{u}'\geq\mathbf{u}\geq \mathbf{w}$ and some $d\in U_\mathbf{u}$ are such that
\begin{equation}\label{E:hyp7}
	\mathcal{S}, \mathbf{u}\models^+_n Sy\wedge (\forall z)_O(Ezy \Leftrightarrow ST_z(\psi))[y/d]
\end{equation}
We now reason as follows:
\begin{align}
	&\mathbf{w} = \mathbf{v}\label{E:hyp8}&&\text{ by }\eqref{E:hyp2}\\
	&\mathcal{S}, \mathbf{w}\models^-_n ST_w(\chi)[w/b]\label{E:hyp9}&&\text{ by IH for $\chi$, }\eqref{E:hyp3},\eqref{E:hyp8}\\
	&\mathcal{S}, \mathbf{u}'\models^-_n ST_w(\chi)[w/\mathtt{H}_{\mathbf{w}\mathbf{u}'}(b)]\label{E:hyp10}&&\text{ by Lm \ref{L:n4-standard}.1, }\eqref{E:hyp9}\\
	&\mathcal{S}, \mathbf{u}'\models^+_n	Rxyw[x/\mathtt{H}_{\mathbf{w}\mathbf{u}'}(a), y/\mathtt{H}_{\mathbf{w}\mathbf{u}'}(c), w/\mathtt{H}_{\mathbf{w}\mathbf{u}'}(b)]\label{E:hyp11}&&\text{ by Lm \ref{L:n4-standard}.1, }\eqref{E:hyp6}\\
	&\mathcal{S}, \mathbf{u}'\models^-_n Rxyw \to ST_w(\chi)[x/\mathtt{H}_{\mathbf{w}\mathbf{u}'}(a), y/\mathtt{H}_{\mathbf{w}\mathbf{u}'}(c), w/\mathtt{H}_{\mathbf{w}\mathbf{u}'}(b)]\label{E:hyp12}&&\text{ by }\eqref{E:hyp10}, \eqref{E:hyp11}\\
	&\mathcal{S}, \mathbf{u}'\models^-_n \forall w(Rxyw \to ST_w(\chi))[x/\mathtt{H}_{\mathbf{w}\mathbf{u}'}(a), y/\mathtt{H}_{\mathbf{w}\mathbf{u}'}(c)]\label{E:hyp13}&&\text{ by }\eqref{E:hyp12}\\
	&\mathcal{S}, \mathbf{w}\models^+_n Sx\wedge (\forall z)_O(Ezx \Leftrightarrow ST_z(\psi))[x/c]\label{E:hyp14}&&\text{ by Lm \ref{L:n4-mod-n4ck-aux}.1, }\eqref{E:hyp4}, \eqref{E:hyp5}\\
	&\mathcal{S}, \mathbf{u}'\models^+_n Sx\wedge (\forall z)_O(Ezx \Leftrightarrow ST_z(\psi))[x/\mathtt{H}_{\mathbf{w}\mathbf{u}'}(c)]\label{E:hyp15}&&\text{ by Lm \ref{L:n4-standard}.1, }\eqref{E:hyp14}\\
	&\mathcal{S}, \mathbf{u}'\models^+_n Sy\wedge (\forall z)_O(Ezy \Leftrightarrow ST_z(\psi))[y/\mathtt{H}_{\mathbf{u}\mathbf{u}'}(d)]\label{E:hyp15a}&&\text{ by Lm \ref{L:n4-standard}.1, }\eqref{E:hyp7}\\
	&\mathcal{S}, \mathbf{u}'\models^+_n Sx\wedge Sy\wedge (\forall z)_O(Ezx \Leftrightarrow Ezy)[x/\mathtt{H}_{\mathbf{w}\mathbf{u}'}(c), y/\mathtt{H}_{\mathbf{u}\mathbf{u}'}(d)]\label{E:hyp16}&&\text{ by }\eqref{E:T18}, \eqref{E:hyp15}, \eqref{E:hyp15a}\\
	&\mathtt{H}_{\mathbf{u}\mathbf{u}'}(d) = \mathtt{H}_{\mathbf{w}\mathbf{u}'}(c)\label{E:hyp17}&&\text{ by }\eqref{E:th5}, \eqref{E:hyp16}\\
	&\mathcal{S}, \mathbf{u}\models^-_n \forall w(Rxyw \to ST_w(\chi))[x/\mathtt{H}_{\mathbf{w}\mathbf{u}'}(a), y/\mathtt{H}_{\mathbf{u}\mathbf{u}'}(d)]\label{E:hyp18}&&\text{ by }\eqref{E:hyp13}, \eqref{E:hyp17}\\
	&\mathcal{S}, \mathbf{u}\models^+_n \sim\forall w(Rxyw \to ST_w(\chi))[x/\mathtt{H}_{\mathbf{u}\mathbf{u}'}(\mathtt{H}_{\mathbf{w}\mathbf{u}}(a)), y/\mathtt{H}_{\mathbf{u}\mathbf{u}'}(d)]\label{E:hyp19}&&\text{ by Df. \ref{D:nelsonian-sheaf}, }\eqref{E:hyp18}
\end{align}
It follows from the above reasoning, by the choice of $\mathbf{u}'\geq \mathbf{u}$, that we must have 
$$
\mathcal{S}, \mathbf{u}\models^+_n (Sy\wedge (\forall z)_O(Ezy \Leftrightarrow ST_z(\psi)))\to \sim\forall w(Rxyw \to ST_w(\chi))[x/\mathtt{H}_{\mathbf{w}\mathbf{u}}(a), y/d];
$$
the latter implies, by the definition of $\&$, that
$$
\mathcal{S}, \mathbf{u}\models^-_n (Sy\wedge (\forall z)_O(Ezy \Leftrightarrow ST_z(\psi)))\,\&\,\forall w(Rxyw \to ST_w(\chi))[x/\mathtt{H}_{\mathbf{w}\mathbf{u}}(a), y/d],
$$
whence, by the choice of $\mathbf{u}$ and $d$, it follows that 
$$
\mathcal{S}, \mathbf{u}\models^-_n \exists y((Sy\wedge (\forall z)_O(Ezy \Leftrightarrow ST_z(\psi)))\,\&\,\forall w(Rxyw \to ST_w(\chi)))[x/a],
$$
or, in other words, that $\mathcal{S}, \mathbf{w}\models^-_n ST_x(\psi\boxto\chi)[x/a]$, as desired.

On the other hand, if we have
\begin{equation}\label{E:hh1}
	\mathcal{M}_\mathcal{S}, (\mathbf{w},a) \not\models^- \psi \boxto \chi
\end{equation}
this means that, for every  $(\mathbf{v}, b) \in W_\mathcal{S}$, we have:
\begin{align}
	&R_\mathcal{S}((\mathbf{w},a), \|\psi\|_{\mathcal{M}_\mathcal{S}}, (\mathbf{v},b))\text{ implies }\mathcal{M}_\mathcal{S}, (\mathbf{v},b) \not\models^- \chi\label{E:hh2}
\end{align}
By Lemma \ref{L:n4-comprehension}, we can choose a $c \in U_\mathbf{w}$, such that:
\begin{equation}\label{E:hh3}
	\mathcal{S}, \mathbf{w}\models^+_n Sy\wedge (\forall z)_O(Ezy \Leftrightarrow ST_z(\psi))[y/c]
\end{equation}
We choose, next, any $d \in U_\mathbf{w}$ and reason as follows:
\begin{align}
	&R^{\mathtt{M}^+_\mathbf{w}}(a, c, d)\label{E:hh4}&&\text{premise}\\
	&\Xi((\mathbf{w}, c), \|\psi\|_{\mathcal{M}_\mathcal{S}})\label{E:hh5}&&\text{ by Lm \ref{L:n4-mod-n4ck-aux}.1, }\eqref{E:hh3}\\
	&R_\mathcal{S}((\mathbf{w},a), \|\psi\|_{\mathcal{M}_\mathcal{S}}, (\mathbf{w},d))\label{E:hh6}&&\text{ by }\eqref{E:hh4}, \eqref{E:hh5}\\
	&\mathcal{M}_\mathcal{S}, (\mathbf{w},d) \not\models^- \chi\label{E:hh7}&&\text{ by }\eqref{E:hh2}, \eqref{E:hh6}\\
	&\mathcal{S}, \mathbf{w}\not\models^-_n ST_w(\chi)[w/d]\label{E:hh8}&&\text{ by IH for }\chi, \eqref{E:hh7}\\
	&\mathcal{S}, \mathbf{w}\not\models^-_n Rxyw\to ST_w(\chi)[x/a, y/c, w/d]\label{E:hh9}&&\text{ by }\eqref{E:hh4}, \eqref{E:hh8}
\end{align}
This reasoning, by the choice of $d \in U_\mathbf{w}$, shows that we have:
$$
\mathcal{S}, \mathbf{w}\not\models^-_n \forall w(Rxyw\to ST_w(\chi))[x/a, y/c],
$$
which, together with \eqref{E:hh3}, implies that
$$
\mathcal{S}, \mathbf{w}\not\models^+_n (Sy\wedge (\forall z)_O(Ezy \Leftrightarrow ST_z(\psi))) \to \sim\forall w(Rxyw\to ST_w(\chi))[x/a, y/c],
$$
whence we get, by definition of $\&$, that
$$
\mathcal{S}, \mathbf{w}\not\models^-_n ((Sy\wedge (\forall z)_O(Ezy \Leftrightarrow ST_z(\psi)))\,\&\,\forall w(Rxyw\to ST_w(\chi)))[x/a, y/c];
$$
the latter implies that
$$
\mathcal{S}, \mathbf{w}\not\models^-_n \exists y((Sy\wedge (\forall z)_O(Ezy \Leftrightarrow ST_z(\psi)))\,\&\,\forall w(Rxyw\to ST_w(\chi)))[x/a],
$$
or, in other words, that $\mathcal{S}, \mathbf{w}\not\models^-_n ST_x(\psi\boxto\chi)[x/a]$, as desired.
\end{proof}
We are now in a position to prove Proposition \ref{P:easy}:
\begin{proof}[Proof of Proposition \ref{P:easy}]
	We argue by contraposition. Assume that $Th\cup ST_x(\Gamma)\not\models_{\mathsf{QN4}}ST_x(\Delta)$. Then there must exist an $\mathcal{S}\in \mathbb{N}4$, a $\mathbf{w} \in W$, and an $a\in U_\mathbf{w}$ such that $\mathcal{S}, \mathbf{w}\models^+_n(Th\cup ST_x(\Gamma), ST_x(\Delta))[x/a]$. By Lemma \ref{L:n4-standard}.2, we have $\mathcal{S}_\mathbf{w}\models^+_n Th$ and $\mathcal{S}_\mathbf{w}, \mathbf{w}\models^+_n(ST_x(\Gamma), ST_x(\Delta))[x/a]$. By Lemma \ref{L:n4-mod-n4ck1}, the structure $\mathcal{M}_{\mathcal{S}_\mathbf{w}}$ given as in Definition \ref{D:n4-mod-n4ck}, is a Nelsonian conditional model, and, by Lemma \ref{L:n4-mod-n4ck2}, we have $\mathcal{M}_{\mathcal{S}_\mathbf{w}},(\mathbf{w}, a)\models^+ (\Gamma, \Delta)$. But then also $\Gamma \not\models_{\mathsf{N4CK}}\Delta$.
\end{proof}
We now wish to prove a converse of Proposition \ref{P:easy}. This task also requires a series of preliminary constructions that develop the potential of the universal model $\mathcal{M}_c$ of $\mathsf{N4CK}$ given in Definition \ref{D:canonical-model}.

\begin{definition}\label{D:standard-sequence}
	For any $n \in \omega$, a sequence $\alpha = ((\Gamma_0,\Delta_0),\phi_1,\ldots,\phi_n, (\Gamma_n, \Delta_n))$ is called a $(\Gamma_0,\Delta_0)$-standard sequence of length $n + 1$ iff $(\Gamma_0,\Delta_0),\ldots,(\Gamma_n, \Delta_n)\in W_c$, $\phi_1,\ldots,\phi_n \in \mathcal{CN}$, and we have:
	$$
	(\forall i < n)(R_c((\Gamma_i,\Delta_i),\|\phi_{i + 1}\|_{\mathcal{M}_c}, (\Gamma_{i+1}, \Delta_{i+1}))).
	$$
	Given a $(\Gamma, \Delta)\in W_c$, the set of all $(\Gamma,\Delta)$-standard sequences will be denoted by $Seq(\Gamma,\Delta)$. The set $Seq$ of all standard sequences is then given by $
	Seq := \bigcup\{Seq(\Gamma,\Delta)\mid (\Gamma, \Delta)\in W_c\}$.
	
	Finally, given a $\beta = ((\Xi_0,\Theta_0),\psi_1,\ldots,\psi_m, (\Xi_m, \Theta_m))\in Seq$, we say that $\beta$ extends $\alpha$ and will write $
	\alpha \mathrel{\prec}\beta$ 
	iff (1) $m = n$, (2) $\phi_1 = \psi_1,\ldots, \phi_n = \psi_n (= \psi_m)$, and
	
	\noindent(3) $(\Gamma_0,\Delta_0)\leq_c (\Xi_0,\Theta_0),\ldots,  (\Gamma_n,\Delta_n)\leq_c (\Xi_n,\Theta_n)(=  (\Xi_m,\Theta_m))$. 
\end{definition}
Given any $(\Gamma, \Delta), (\Xi, \Theta)\in W_c$ such that $(\Gamma,\Delta)\leq_c (\Xi,\Theta)$, a function $f:Seq(\Gamma,\Delta)\to Seq(\Xi,\Theta)$ is called a \textit{local} $((\Gamma,\Delta),(\Xi,\Theta))$-\textit{choice function} iff $(\forall \alpha \in Seq(\Gamma,\Delta))(\alpha\mathrel{\prec}f(\alpha))$. The set of all local $((\Gamma,\Delta),(\Xi,\Theta))$-choice functions will be denoted by $\mathfrak{F}((\Gamma,\Delta),(\Xi,\Theta))$.

However, $((\Gamma,\Delta),(\Xi,\Theta))$-choice functions will mostly interest us as restrictions of \textit{global choice functions}. More precisely, $F: Seq \to Seq$ is a global choice function iff for every $(\Gamma, \Delta) \in W_c$ there exists a $(\Xi, \Theta)\in W_c$ such that $(\Gamma,\Delta)\leq_c (\Xi,\Theta)$ and $F\upharpoonright Seq(\Gamma,\Delta) \in \mathfrak{F}((\Gamma,\Delta),(\Xi,\Theta))$. The set of all global choice functions will be denoted by $\mathfrak{G}$. The following lemma sums up some useful properties of global choice functions:
\begin{lemma}\label{L:global-functions}
	Let  $(\Gamma, \Delta) \in W_c$. Then the following statements hold:
	\begin{enumerate}
		\item For every $(\Xi, \Theta)\in W_c$ such that $(\Gamma,\Delta)\leq_c (\Xi,\Theta)$ and for every $f \in \mathfrak{F}((\Gamma,\Delta),(\Xi,\Theta))$, there exists an $F \in \mathfrak{G}$ such that $F\upharpoonright Seq(\Gamma,\Delta) = f$.
		
		\item For every $\alpha \in Seq(\Gamma,\Delta)$ and every $(\Xi, \Theta)\in W_c$ such that $end(\alpha) \leq_c (\Xi, \Theta)$, there exists a $\beta \in Seq$ such that $end(\beta) = (\Xi, \Theta)$, and an $F\in\mathfrak{G}$ such that $F(\alpha) = \beta$.
		
		\item $id[Seq]\in \mathfrak{G}$.
		
		\item Given any $n \in \omega$ and any $F_1,\ldots,F_n \in \mathfrak{G}$ we also have that $F_1\circ\ldots\circ F_n \in \mathfrak{G}$.
		
		\item For every $\alpha \in Seq$ and every $F \in \mathfrak{G}$, we have $\alpha\mathrel{\prec}F(\alpha)$.
	\end{enumerate} 
\end{lemma}
The proof of Lemma \ref{L:global-functions} repeats the proof of \cite[Lemma 19]{olkhovikov} almost word-for-word; however, for the sake of completeness, we also include this proof in Appendix \ref{App:seq}.

The global choice functions are the basis for another type of sequences, that, along with the standard sequences, is necessary for the main model-theoretic construction of the present section. We will call them \textit{global sequences}. A global sequence is any sequence of the form $(F_1,\ldots, F_n) \in \mathfrak{G}^n$ where $n \in \omega$ (thus $\Lambda$ is also a global sequence with $n = 0$). Given two global sequences $(F_1,\ldots, F_k)$ and $(G_1,\ldots, G_m)$, we say that $(G_1,\ldots, G_m)$ \textit{extends} $(F_1,\ldots, F_k)$ and write $(F_1,\ldots, F_k)\mathrel{\sqsubseteq}(G_1,\ldots, G_m)$ iff $k \leq m$ and $F_1 = G_1,\ldots, F_k = G_k$. Furthermore, we will denote by $Glob$ the set $\bigcup_{n \in \omega}\mathfrak{G}^n$, that is to say, the set of all global sequences.

The final item in this series of preliminary model-theoretic constructions is a certain equivalence relation on $\mathcal{CN}$. Namely, given any $\phi,\psi \in \mathcal{CN}$, we define that:
$$
\phi\simeq\psi\text{ iff } (\phi \Leftrightarrow \psi \in \mathsf{N4CK}).
$$
For any $\phi \in \mathcal{CN}$, we will denote its equivalence class relative to $\simeq$ by $[\phi]_\simeq$.

We now proceed to define an $\mathcal{S}_c \in \mathbb{N}4$ induced by  $\mathcal{M}_c\in \mathbb{NC}$ of Definition \ref{D:canonical-model}.
\begin{definition}\label{D:canonical-sheaf}
	We set $\mathcal{S}_c:= (Glob, \sqsubseteq, \mathbb{M}^+, \mathbb{M}^-,\mathbb{F})$, where:
	\begin{itemize}
		\item For every $\bar{F} \in Glob$, $\mathbb{M}^+_{\bar{F}} := \mathbb{M}^+$, $\mathbb{M}^-_{\bar{F}} := \mathbb{M}^-$ i.e. every global sequence gets assigned the same positive model $\mathbb{M}^+\in \mathbb{C}(\Pi)$ (resp. the same negative model $\mathbb{M}^-\in \mathbb{C}(\Pi\cup\{\epsilon^2\})$). As for the components of these models, we set:
		\begin{itemize}
			\item $U^{\mathbb{M}^+} = U^{\mathbb{M}^-} = U := Seq \cup \{[\phi]_{\simeq}\mid \phi \in \mathcal{CN}\}$. 
			
			\item $p^{\mathbb{M}^+} := \{\beta \in Seq\mid p\in \pi^1(end(\beta))\}$ and $p^{\mathbb{M}^-} := \{\beta \in Seq\mid \sim p\in \pi^1(end(\beta))\}$ for every $p^1 \in Prop$.
			
			\item $S^{\mathbb{M}^+} :=\{[\phi]_{\simeq}\mid \phi \in \mathcal{CN}\}$ and $S^{\mathbb{M}^-} :=Seq$.
			
			\item $O^{\mathbb{M}^+} :=Seq$ and $O^{\mathbb{M}^-} := \{[\phi]_{\simeq}\mid \phi \in \mathcal{CN}\}$.
			
			\item $E^{\mathbb{M}^+} := \{(\beta, [\phi]_{\simeq})\in O^{\mathbb{M}^+}\times S^{\mathbb{M}^+}\mid \phi \in \pi^1(end(\beta))\}$ and $E^{\mathbb{M}^-} := \{(\beta, [\phi]_{\simeq})\in O^{\mathbb{M}^+}\times S^{\mathbb{M}^+}\mid \sim\phi \in \pi^1(end(\beta))\}$. 
			
			\item $R^{\mathbb{M}^+} := \{(\beta, [\phi]_{\simeq}, \gamma)\in  O^{\mathbb{M}^+}\times S^{\mathbb{M}^+}\times O^{\mathbb{M}^+}\mid (\exists (\Xi, \Theta) \in W_c)(\exists \psi \in [\phi]_{\simeq})(\gamma = \beta^\frown(\psi, (\Xi, \Theta)))\}$ and $R^{\mathbb{M}^-} := \emptyset$.
			
			\item $\epsilon^{\mathbb{M}^-}:= \{([\phi]_{\simeq},[\psi]_{\simeq})\in S^{\mathbb{M}^+}\times S^{\mathbb{M}^+}\mid [\phi]_{\simeq}\neq[\psi]_{\simeq}\}$.
		\end{itemize} 
		\item For  any $\bar{F}, \bar{G} \in Glob$ such that, for some $k \leq m$ we have  $\bar{F} = (F_1,\ldots, F_k)$ and  $\bar{G} = (F_1,\ldots, F_m)$ we have $\mathbb{F}_{\bar{F}\bar{G}}:U\to U$, where we set:
		$$
		\mathbb{F}_{\bar{F}\bar{G}}(\gamma):= \begin{cases}
			(F_{k + 1}\circ\ldots\circ F_m)(\gamma),\text{ if }\gamma \in  Seq;\\
			\gamma,\text{ otherwise. }
		\end{cases}
		$$
	\end{itemize}
\end{definition}
The first thing to show is that we, indeed, have  $\mathcal{S}_c\in \mathbb{N}4$. Again, we begin by establishing another technical fact:
\begin{lemma}\label{L:technical}
	For all $\bar{F}, \bar{G}\in  Glob$ such that $\bar{F}\mathrel{\sqsubseteq}\bar{G}$, it is true that:
	\begin{enumerate}
		\item $\mathbb{F}_{\bar{F}\bar{G}}\upharpoonright Seq \in \mathfrak{G}$.
		
		\item For every $\alpha \in Seq$ we have $\alpha\mathrel{\prec}\mathbb{F}_{\bar{F}\bar{G}}(\alpha)$; in particular, we have $end(\alpha)\leq_c end(\mathbb{F}_{\bar{F}\bar{G}}(\alpha))$, or, equivalently, $\pi^1(end(\alpha))\subseteq \pi^1(end(\mathbb{F}_{\bar{F}\bar{G}}(\alpha)))$.
	\end{enumerate} 
\end{lemma}
\begin{proof}
	Assume the hypothesis; we may also assume, wlog, that, for some $k \leq m < \omega$, $\bar{F}$ and $\bar{G}$ are given in the following form:
	\begin{align}
		\bar{F} &= (F_1,\ldots, F_k)\label{E:barf}\tag{def-$\bar{F}$}\\
		\bar{G} &= (F_1,\ldots, F_m)\label{E:barg}\tag{def-$\bar{G}$}
	\end{align}
	In this case, we also get the following representation for $\mathbb{F}_{\bar{F}\bar{G}}$:
	\begin{align}
		\mathbb{F}_{\bar{F}\bar{G}}\upharpoonright Seq = id[U]\circ F_{k + 1}\circ\ldots\circ F_m\label{E:fab1}\tag{def1-$\mathbb{F}_{\bar{F}\bar{G}}$}\\
		\mathbb{F}_{\bar{F}\bar{G}}\upharpoonright \{[\phi]_{\simeq}\mid \phi \in \mathcal{CN}\} = id[\{[\phi]_{\simeq}\mid \phi \in \mathcal{CN}\}]\label{E:fab2}\tag{def2-$\mathbb{F}_{\bar{F}\bar{G}}$}
	\end{align} 
	Part 1 now easily follows from \eqref{E:fab1} and Lemma \ref{L:global-functions}.4. As for Part 2, we observe that, if $\alpha \in Seq$, then $
	\mathbb{F}_{\bar{F}\bar{G}}(\alpha) =  (id[U]\circ F_{k + 1}\circ\ldots\circ F_m)(\alpha)\in Seq$.
	Now Part 1 and Lemma \ref{L:global-functions}.5 together imply that $\alpha\mathrel{\prec}\mathbb{F}_{\bar{F}\bar{G}}(\alpha)$. By Definition \ref{D:standard-sequence}, this means that $end(\alpha)\leq_c end(\mathbb{F}_{\bar{F}\bar{G}}(\alpha))$, or, equivalently, that $\pi^1(end(\alpha))\subseteq \pi^1(end(\mathbb{F}_{\bar{F}\bar{G}}(\alpha)))$. 	
\end{proof}

\begin{lemma}\label{L:canonical-sheaf}
	$\mathcal{S}_c$ is a Nelsonian sheaf.
\end{lemma}
\begin{proof}
	It is clear that $Glob$ is non-empty and that $\sqsubseteq$ defines a pre-order on $Glob$. It is also clear that $\mathbb{M}^+\in \mathbb{C}(\Theta)$, $\mathbb{M}^-\in \mathbb{C}(\Theta\cup \{\epsilon^2\})$, and that, for $\bar{F}, \bar{G}\in  Glob$, $\bar{F}\mathrel{\sqsubseteq}\bar{G}$ implies $\mathbb{F}_{\bar{F}\bar{G}}:U\to U$. We need to show that $\mathcal{S}^+_c = (Glob, \sqsubseteq, \mathbb{M}^+, \mathbb{F})\in \mathbb{I}(\Theta)$ and that $\mathcal{S}^-_c =(Glob, \sqsubseteq, \mathbb{M}^+, \mathbb{F})\in \mathbb{I}(\Theta\cup\{\epsilon^2\})$.
	
	As for the conditions imposed by Definition \ref{D:intuitionistic-sheaf}.4  on the functions of the form $\mathbb{F}_{\bar{F}\bar{G}}$, it is clear from Definition \ref{D:canonical-sheaf} and from our convention on compositions of empty families of functions that (1) for any $\bar{F}\in  Glob$ we will have in $\mathbb{F}_{\bar{F}\bar{F}} = id[U]$ and that (2) if $\bar{F}, \bar{G}, \bar{H}\in  Glob$ are such that $\bar{F}\mathrel{\sqsubseteq}\bar{G}\mathrel{\sqsubseteq}\bar{H}$, then $\mathbb{F}_{\bar{F}\bar{G}}\circ\mathbb{F}_{\bar{G}\bar{H}} = \mathbb{F}_{\bar{F}\bar{H}}$.
	
	It remains to establish that, for each pair $\bar{F}, \bar{G}\in  Glob$ such that  $\bar{F}\mathrel{\sqsubseteq}\bar{G}$, we have  $\mathbb{F}_{\bar{F}\bar{G}}\in Hom(\mathbb{M}^+,\mathbb{M}^+)\cap Hom(\mathbb{M}^-,\mathbb{M}^-)$. The latter claim boils down to showing that the $\mathbb{M}^+$-extension of every predicate in $\Pi$ and the $\mathbb{M}^-$-extension of every predicate in $\Pi\cup\{\epsilon^2\}$ is preserved by $\mathbb{F}_{\bar{F}\bar{G}}$. In doing so, we will assume that, for some appropriate $k \leq m < \omega$, $\bar{F}$, $\bar{G}$, and $\mathbb{F}_{\bar{F}\bar{G}}$ are given in a form that satisfies \eqref{E:barf}, \eqref{E:barg}, \eqref{E:fab1}, and \eqref{E:fab2}. So let $\mathbb{X} \in \Pi\cup\{\epsilon^2\}$; the following cases have to be considered:
	
	\textit{Case 1}. $\mathbb{X} \in \{O^1, S^1, \epsilon^2\}$. Trivial by \eqref{E:fab1} and \eqref{E:fab2}.
	
	\textit{Case 2}. $\mathbb{X} = p^1 \in Prop$. If $\alpha \in U$ is such that $\alpha \in p^{\mathbb{M}^+}$, then, by Definition \ref{D:canonical-sheaf} and Lemma \ref{L:technical}, we must have all of the following: (1) $\alpha, \mathbb{F}_{\bar{F}\bar{G}}(\alpha) \in Seq$, (2) $\pi^1(end(\alpha))\subseteq \pi^1(end(\mathbb{F}_{\bar{F}\bar{G}}(\alpha)))$, and (3) $p\in \pi^1(end(\alpha))$. But then clearly also $p \in \pi^1(end(\mathbb{F}_{\bar{F}\bar{G}}(\alpha)))$, whence, further, $\mathbb{F}_{\bar{F}\bar{G}}(\alpha) \in p^{\mathbb{M}^+}$, as desired. Similarly, if  $\alpha \in U$ is such that $\alpha \in p^{\mathbb{M}^-}$, then, by Definition \ref{D:canonical-sheaf} and Lemma \ref{L:technical}, we must have (1) and (2) as given above, plus $\sim p\in \pi^1(end(\alpha))$. But then clearly also $\sim p \in \pi^1(end(\mathbb{F}_{\bar{F}\bar{G}}(\alpha)))$, whence, further, $\mathbb{F}_{\bar{F}\bar{G}}(\alpha) \in p^{\mathbb{M}^-}$, as desired
	
	\textit{Case 3}. $\mathbb{X} = E^2$. If $\alpha, \beta \in U$ are such that $(\alpha, \beta) \in E^{\mathbb{M}^+}$, then, arguing as in Case 2, we must have: (1) $\alpha, \mathbb{F}_{\bar{F}\bar{G}}(\alpha) \in Seq$, (2) $\beta \in S^{\mathbb{M}^+}$, in other words, $\beta = [\phi]_{\simeq}$ for some $\phi \in \mathcal{CN}$, (3) $\pi^1(end(\alpha))\subseteq \pi^1(end(\mathbb{F}_{\bar{F}\bar{G}}(\alpha)))$, and (4) $\phi \in \pi^1(end(\alpha))\subseteq \pi^1(end(\mathbb{F}_{\bar{F}\bar{G}}(\alpha)))$, so that we also have, by \eqref{E:fab2}, that $
	(\mathbb{F}_{\bar{F}\bar{G}}(\alpha),\mathbb{F}_{\bar{F}\bar{G}}(\beta)) = (\mathbb{F}_{\bar{F}\bar{G}}(\alpha),\mathbb{F}_{\bar{F}\bar{G}}([\phi]_{\simeq})) = (\mathbb{F}_{\bar{F}\bar{G}}(\alpha),[\phi]_{\simeq}) \in E^{\mathbb{M}^+}$. On the other hand, if $(\alpha, \beta) \in E^{\mathbb{M}^-}$, then we must have (1)--(3) as above plus $\sim\phi \in \pi^1(end(\alpha))\subseteq \pi^1(end(\mathbb{F}_{\bar{F}\bar{G}}(\alpha)))$, so that we also have, by \eqref{E:fab2}, that $
	(\mathbb{F}_{\bar{F}\bar{G}}(\alpha),\mathbb{F}_{\bar{F}\bar{G}}(\beta)) = (\mathbb{F}_{\bar{F}\bar{G}}(\alpha),\mathbb{F}_{\bar{F}\bar{G}}([\phi]_{\simeq})) = (\mathbb{F}_{\bar{F}\bar{G}}(\alpha),[\phi]_{\simeq}) \in E^{\mathbb{M}^-}$.
	
	\textit{Case 4}. $\mathbb{X} = R^3$. Then for no $\alpha, \beta, \gamma \in U$ can we have $(\alpha, \beta, \gamma) \in R^{\mathbb{M}^-}$. If  $\alpha, \beta, \gamma \in U$ are such that $(\alpha, \beta, \gamma) \in R^{\mathbb{M}^+}$, then, arguing as in Case 2, we must have: (1) $\alpha, \gamma, \mathbb{F}_{\bar{F}\bar{G}}(\alpha), \mathbb{F}_{\bar{F}\bar{G}}(\gamma)\in Seq$, (2) $\beta \in S^{\mathbb{M}^+}$, in other words, $\beta = [\phi]_{\simeq}$ for some $\phi \in \mathcal{CN}$, (3) $\alpha\mathrel{\prec}\mathbb{F}_{\bar{F}\bar{G}}(\alpha)$, and $\gamma\mathrel{\prec}\mathbb{F}_{\bar{F}\bar{G}}(\gamma)$, and, finally, (4) for some $(\Xi,\Theta)\in W_c$ and some $\psi \in [\phi]_{\simeq}$, we must have $\gamma = \alpha^\frown(\psi, (\Xi, \Theta))$. Now, Definition \ref{D:standard-sequence} implies that, for some $(\Xi',\Theta')\in W_c$ such that $(\Xi,\Theta)\leq_c(\Xi',\Theta')$ we must have $\mathbb{F}_{\bar{F}\bar{G}}(\gamma) = \mathbb{F}_{\bar{F}\bar{G}}(\alpha)^\frown(\psi, (\Xi', \Theta'))$. Therefore, by Definition \ref{D:canonical-sheaf}, we must have
	$$
	(\mathbb{F}_{\bar{F}\bar{G}}(\alpha), \mathbb{F}_{\bar{F}\bar{G}}(\beta), \mathbb{F}_{\bar{F}\bar{G}}(\gamma)) =  (\mathbb{F}_{\bar{F}\bar{G}}(\alpha), [\phi]_{\simeq}, \mathbb{F}_{\bar{F}\bar{G}}(\gamma))\in R^{\mathbb{M}^+}.
	$$  
\end{proof}
We will eventually have to show that the Nelsonian sheaf $\mathcal{S}_c$ satisfies \eqref{E:th5}--\eqref{E:th9}. The  following lemma shows the satisfaction of \eqref{E:th5}:
\begin{lemma}\label{L:indiscernibles}
	$\mathcal{S}_c \models^+_{n}  \eqref{E:th5}$.
\end{lemma}
\begin{proof}
	Let $\bar{F} \in Glob$ and let $a, b \in U$ be such that $
	\mathcal{S}_c,\bar{F} \models^+_n Sx\wedge Sz\wedge (\forall y)_O (Eyx\Leftrightarrow Eyz)[x/a, z/b]$.
	Then $a,b \in S^{\mathbb{M}^+}$, that is to say, for some $\phi, \psi \in \mathcal{CN}$, we must have $a = [\phi]_{\simeq}$ and $b = [\psi]_{\simeq}$. Assume, towards contradiction, that $a = [\phi]_{\simeq}\neq [\psi]_{\simeq} = b$, then we must have $(\phi\Leftrightarrow \psi)\notin \mathsf{N4CK}$; in view of Proposition \ref{P:truth-lemma}, we can suppose, wlog, that for some $(\Gamma, \Delta) \in W_c$ we either have $\phi \in \Gamma$ but $\psi \notin \Gamma$ or we have $\sim\phi \in \Gamma$ but $\sim\psi \notin \Gamma$. Since we clearly have $(\Gamma, \Delta) \in Seq \subseteq U$, it follows from Definition \ref{D:canonical-sheaf} that in the former case we must have both $
	\mathcal{S}_c,\bar{F} \models^+_{n} Oy\wedge Eyx[x/a, y/(\Gamma, \Delta)]$,
	and $\mathcal{S}_c,\bar{F} \not\models^+_{n} Eyz[y/(\Gamma, \Delta), z/b]$, which contradicts our assumption. Similarly, in the latter case we must have both $
	\mathcal{S}_c,\bar{F} \models^+_{n} Oy\wedge \sim Eyx[x/a, y/(\Gamma, \Delta)]$,
	and $\mathcal{S}_c,\bar{F} \not\models^+_{n} \sim Eyz[y/(\Gamma, \Delta), z/b]$, which, again, contradicts our assumption. Therefore, we must have $a = b$.
\end{proof} 
We will also need the following corollary to Lemma \ref{L:indiscernibles}
\begin{corollary}\label{C:indiscernibles}
	Let $x,y,z \in Ind$ be pairwise distinct, let $\phi \in \mathcal{FO}^y$,  $\bar{F}\in Glob$, and $a,b \in U$ be such that both $\mathcal{S}_c,\bar{F}\models^+_{n} Sx\wedge (\forall y)_O(Eyx\Leftrightarrow \phi)[x/a]$ and $\mathcal{S}_c,\bar{F}\models^+_{n} Sx\wedge (\forall y)_O(Eyx\Leftrightarrow \phi)[x/b]$. Then $a = b$.
\end{corollary}
\begin{proof}
	Assume the premises. Renaming the variables, we get that $\mathcal{S}_c,\bar{F}\models^+_{n} Sz\wedge (\forall y)_O(Ezx\Leftrightarrow \phi)[z/b]$, whence, by \eqref{E:T18} we obtain that $\mathcal{S}_c,\bar{F}\models^+_{n} Sx \wedge Sz\wedge (\forall y)_O(Eyx\Leftrightarrow Eyz)[x/a,z/b]$. It follows, by Lemma \ref{L:indiscernibles}, that $a = b$.
\end{proof}
The next lemma can be seen as a version of a `truth lemma' for $\mathcal{S}_c$.
\begin{lemma}\label{L:sheaf-truth}
	Let $\bar{F} \in Glob$, $\alpha \in Seq$, $x \in Ind$, and let $\phi\in \mathcal{CN}$. Then the following statements hold:
	\begin{enumerate}
		\item $\mathcal{S}_c,\bar{F} \models^+_n ST_x(\phi)[x/\alpha] \text{ iff } \phi \in \pi^1(end(\alpha))$.
		
		\item $\mathcal{S}_c,\bar{F} \models^-_n ST_x(\phi)[x/\alpha] \text{ iff } \sim\phi \in \pi^1(end(\alpha))$.
		
		\item $\mathcal{S}_c,\bar{F}\models^+_{n} Sx\wedge (\forall y)_O(Eyx\Leftrightarrow ST_y(\phi))[x/[\phi]_{\simeq}]$.
	\end{enumerate}		
\end{lemma}
\begin{proof}
	We observe, first, that, for any given $\phi \in \mathcal{CN}$, Parts 1 and 2 together clearly imply Part 3. Indeed, we must have $\mathcal{S}_c,\bar{F}\models^+_{n} Sx[x/[\phi]_{\simeq}]$. As for the other conjunct, if Parts 1 and 2 hold for a given $\phi$ and for all instantiations of $\bar{F}$, $\alpha$, and $x$, then assume that a $\bar{G} \in Glob$ is such that $\bar{F}\mathrel{\sqsubseteq}\bar{G}$. Then, by Part 1, we must have, for every $\beta \in Seq$:
	$$
	E^{\mathbb{M}^+}(\beta,[\phi]_{\simeq}) \text{ iff } \phi \in \pi^1(end(\beta)) \text{ iff } \mathcal{S}_c,\bar{G}\models^+_{n} ST_x(\phi)[x/\beta],
	$$
	and, by Part 2, we will have, for every  $\beta \in Seq$:
	$$
	E^{\mathbb{M}^-}(\beta,[\phi]_{\simeq}) \text{ iff } \sim\phi \in \pi^1(end(\beta)) \text{ iff } \mathcal{S}_c,\bar{G}\models^+_{n} ST_x(\phi)[x/\beta].
	$$
	Since we have $[\phi]_{\simeq} = \mathbb{F}_{\bar{F}\bar{G}}([\phi]_{\simeq})$, Corollary \ref{C:set-encoding} allows us to conclude that $
	\mathcal{S}_c,\bar{F}\models^+_{n} (\forall y)_O(Eyx\Leftrightarrow ST_y(\phi)))[x/[\phi]_{\simeq}]$.
	
	We will therefore prove all the three parts simultaneously by induction on the construction of $\phi\in \mathcal{CN}$; but, in view of the foregoing observation, we will only argue for Parts 1 and 2.
	
	\textit{Basis}. Assume that $\phi = p$ for some $p^1 \in Prop$. Then $ST_x(\phi) = px$ and Definition \ref{D:canonical-sheaf} implies that we have (for Part 1) $\mathcal{S}_c,\bar{F}\models^+_{n} px[x/\alpha]$ iff $\alpha \in p^{\mathbb{M}^+}$ iff $p \in \pi^1(end(\alpha))$,
	and (for Part 2) $\mathcal{S}_c,\bar{F}\models^-_{n} px[x/\alpha]$ iff $\alpha \in p^{\mathbb{M}^-}$ iff $\sim p \in \pi^1(end(\alpha))$.
	
	\textit{Induction step}. Again, several cases are possible:
	
	\textit{Case 1}. $\phi = \psi \wedge \chi$. Then $ST_x(\phi) = ST_x(\psi)\wedge ST_x(\chi)$ and we have, for Part 1, by IH and Proposition \ref{P:truth-lemma}:
	\begin{align*}
		\mathcal{S}_c,\bar{F}\models^+_{n} (ST_x(\psi)\wedge ST_x(\chi))[x/\alpha] &\text{ iff } \mathcal{S}_c,\bar{F}\models^+_{n} ST_x(\psi)[x/\alpha]\text{ and }\mathcal{S}_c,\bar{F}\models^+_{n} ST_x(\chi)[x/\alpha]\\
		&\text{ iff }  \psi\in \pi^1(end(\alpha))\text{ and }\chi\in \pi^1(end(\alpha))\\
		&\text{ iff }  \mathcal{M}_c, end(\alpha)\models^+ \psi\text{ and } \mathcal{M}_c, end(\alpha)\models^+ \chi\\
		&\text{ iff }  \mathcal{M}_c, end(\alpha)\models^+ \psi\wedge\chi\\
		&\text{ iff }  \psi\wedge\chi\in \pi^1(end(\alpha)) 
	\end{align*}
and, for Part 2:
\begin{align*}
	\mathcal{S}_c,\bar{F}\models^-_{n} (ST_x(\psi)\wedge ST_x(\chi))[x/\alpha] &\text{ iff } \mathcal{S}_c,\bar{F}\models^-_{n} ST_x(\psi)[x/\alpha]\text{ or }\mathcal{S}_c,\bar{F}\models^-_{n} ST_x(\chi)[x/\alpha]\\
	&\text{ iff }  \sim\psi\in \pi^1(end(\alpha))\text{ or }\sim\chi\in \pi^1(end(\alpha))\\
	&\text{ iff }  \mathcal{M}_c, end(\alpha)\models^- \psi\text{ or } \mathcal{M}_c, end(\alpha)\models^- \chi\\
	&\text{ iff }  \mathcal{M}_c, end(\alpha)\models^- \psi\wedge\chi\\
	&\text{ iff }  \sim(\psi\wedge\chi)\in \pi^1(end(\alpha)) 
\end{align*}
	
	\textit{Case 2} $\phi = \psi \vee \chi$, and \textit{Case 3} $\phi = \sim\psi$ are similar to Case 1.
	
	\textit{Case 4}. $\phi = \psi \to \chi$. By Proposition \ref{P:truth-lemma}, we know that:
	\begin{align}
		\psi\to&\chi\in \pi^1(end(\alpha)) \Leftrightarrow \mathcal{M}_c, end(\alpha)\models^+ \psi\to\chi\notag\\
		&\text{ iff }(\forall(\Xi,\Theta)\in W_c)(end(\alpha)\leq_c(\Xi,\Theta)\text{ and }\mathcal{M}_c, (\Xi,\Theta)\models^+ \psi\text{ implies }\mathcal{M}_c, (\Xi,\Theta)\models^+ \chi)\notag\\
		&\text{ iff }(\forall(\Xi,\Theta)\in W_c)(end(\alpha)\leq_c(\Xi,\Theta)\text{ and }\psi \in \Xi \text{ implies }\chi\in \Xi)\label{E:to}\tag{$\to$}
	\end{align}
	We have $ST_x(\phi) = ST_x(\psi)\to ST_x(\chi)$, and we argue as follows for Part 1.
	
	First, assume that $\psi\to\chi\in \pi^1(end(\alpha))$, and let
	$\bar{G} \in Glob$ be such that $\bar{F}\mathrel{\sqsubseteq}\bar{G}$. Then, by Lemma \ref{L:technical}.2, $end(\alpha)\leq_c end(\mathbb{F}_{\bar{F}\bar{G}}(\alpha))$, so that, by \eqref{E:to}, $
	\psi \in \pi^1(end(\mathbb{F}_{\bar{F}\bar{G}}(\alpha)))$ entails $\chi\in \pi^1(end(\mathbb{F}_{\bar{F}\bar{G}}(\alpha)))$,
	whence, by IH, it follows that
	$$
	\mathcal{S}_c,\bar{G}\models^+_{n} ST_x(\psi)[x/\mathbb{F}_{\bar{F}\bar{G}}(\alpha)]\text{ implies } \mathcal{S}_c,\bar{G}\models^+_{n} ST_x(\chi)[x/\mathbb{F}_{\bar{F}\bar{G}}(\alpha)].
	$$
	Since $\bar{G}\sqsupseteq\bar{F}$ was chosen in $Glob$ arbitrarily, we have shown that $\mathcal{S}_c,\bar{F}\models^+_{n} (ST_x(\psi)\to ST_x(\chi))[x/\alpha]$, or, equivalently, that $\mathcal{S}_c,\bar{F}\models^+_{n} ST_x(\phi)[x/\alpha]$.
	
	For the converse, we argue by contraposition. Assume that $\psi\to\chi\notin \pi^1(end(\alpha))$. By \eqref{E:to}, there must be a  $(\Xi,\Theta)\in W_c$ such that $end(\alpha)\leq_c(\Xi,\Theta)$, $\psi \in \Xi$, and $\chi\notin \Xi$. By Lemma \ref{L:global-functions}.2, we can choose a $\beta \in Seq$ such that $end(\beta) = (\Xi,\Theta)$, and a $F \in \mathfrak{G}$ such that $F(\alpha) = \beta$. But then we can set $\bar{G} := \bar{F}^\frown(F) \in Glob$; we clearly have $\bar{G}\sqsupseteq\bar{F}$, $\mathbb{F}_{\bar{F}\bar{G}}(\alpha) = F(\alpha) = \beta$, and $\pi^1(end(\beta)) = \Xi$. Therefore, under these settings, we also get that $\psi \in \pi^1(end(\mathbb{F}_{\bar{F}\bar{G}}(\alpha)))$, but $\chi\notin \pi^1(end(\mathbb{F}_{\bar{F}\bar{G}}(\alpha)))$. Thus we must have, by IH, that $\mathcal{S}_c,\bar{G}\models^+_n ST_x(\psi)[x/\mathbb{F}_{\bar{F}\bar{G}}(\alpha)]$ but  $\mathcal{S}_c,\bar{G}\not\models^+_n ST_x(\chi)[x/\mathbb{F}_{\bar{F}\bar{G}}(\alpha)]$. The latter means that $\mathcal{S}_c,\bar{F}\not\models^+_n (ST_x(\psi)\to ST_x(\chi))[x/\alpha]$, or, equivalently, that $\mathcal{S}_c,\bar{F}\not\models^+_n ST_x(\phi)[x/\alpha]$. 
	
	As for Part 2, we have, by IH and Proposition \ref{P:truth-lemma}:
	\begin{align*}
		\mathcal{S}_c,\bar{F}\models^-_{n} (ST_x(\psi)\to ST_x(\chi))[x/\alpha] &\text{ iff } \mathcal{S}_c,\bar{F}\models^+_{n} ST_x(\psi)[x/\alpha]\text{ and }\mathcal{S}_c,\bar{F}\models^-_{n} ST_x(\chi)[x/\alpha]\\
		&\text{ iff }  \psi\in \pi^1(end(\alpha))\text{ and }\sim\chi\in \pi^1(end(\alpha))\\
		&\text{ iff }  \mathcal{M}_c, end(\alpha)\models^+ \psi\text{ and } \mathcal{M}_c, end(\alpha)\models^- \chi\\
		&\text{ iff }  \mathcal{M}_c, end(\alpha)\models^- \psi\to\chi\\
		&\text{ iff }  \sim(\psi\to\chi)\in \pi^1(end(\alpha)) 
	\end{align*}
	
	\textit{Case 5}. $\phi = \psi\boxto\chi$. Then we have 
	\begin{align*}
		ST_x(\phi) &:= \exists y((Sy \wedge (\forall z)_O(Ezy\Leftrightarrow ST_z(\psi)))\,\&\,\forall w(Rxyw\to ST_w(\chi))).
	\end{align*}
We argue for Part 1 first. By Proposition \ref{P:truth-lemma}, we know that:
	\begin{align}
		\psi\boxto&\chi\in \pi^1(end(\alpha)) \text{ iff } \mathcal{M}_c, end(\alpha)\models^+ \psi\boxto\chi\notag\\
		&\text{ iff }(\forall(\Xi,\Theta),(\Xi',\Theta')\in W_c)\notag\\
		&\qquad\quad(end(\alpha)\leq_c(\Xi,\Theta)\text{ and }R_c((\Xi,\Theta), \|\psi\|_{\mathcal{M}_c}, (\Xi',\Theta')) \text{ implies }\mathcal{M}_c, (\Xi',\Theta')\models^+ \chi)\notag\\
	&\text{ iff }(\forall(\Xi,\Theta),(\Xi',\Theta')\in W_c)\notag\\
		&\qquad(end(\alpha)\leq_c(\Xi,\Theta)\text{ and }R_c((\Xi,\Theta), \|\psi\|_{\mathcal{M}_c}, (\Xi',\Theta')) \text{ implies }\chi\in \Xi')\label{E:boxto}\tag{$\boxto$}
	\end{align}
	Assume first that $\psi\boxto\chi\in \pi^1(end(\alpha))$; let $\beta \in Seq$, and let
	$\bar{G} \in Glob$ be such that $\bar{F}\mathrel{\sqsubseteq}\bar{G}$. If $(\mathbb{F}_{\bar{F}\bar{G}}(\alpha), [\psi]_{\simeq},\beta) \in R^{\mathbb{M}^+}$, then we must have, by Definition \ref{D:canonical-sheaf}, that, for some $(\Gamma, \Delta)\in W_c$ and some $\theta\in [\psi]_{\simeq}$, we have $\beta = \mathbb{F}_{\bar{F}\bar{G}}(\alpha)^\frown(\theta, (\Gamma, \Delta))$. But then, by Definition \ref{D:standard-sequence}, also $R_c(end(\mathbb{F}_{\bar{F}\bar{G}}(\alpha)), \|\theta\|_{\mathcal{M}_c}, (\Gamma, \Delta))$; and, since $\theta\in [\psi]_{\simeq}$, Proposition \ref{P:truth-lemma} implies that $\|\theta\|_{\mathcal{M}_c} = \|\psi\|_{\mathcal{M}_c}$, so that $R_c(end(\mathbb{F}_{\bar{F}\bar{G}}(\alpha)), \|\psi\|_{\mathcal{M}_c}, (\Gamma, \Delta))$. By Lemma \ref{L:technical}.2, we must further have $end(\alpha)\leq_c end(\mathbb{F}_{\bar{F}\bar{G}}(\alpha))$, and now \eqref{E:boxto} yields that
	$\chi \in \Gamma = \pi^1(end(\beta))$. Therefore, it follows by IH that $\mathcal{S}_c,\bar{G}\models^+_{n} ST_w(\chi)[w/\beta]$.
	
	Since the choice of $\beta \in Seq$ and 
	$\bar{G} \in Glob$ such that $\bar{F}\mathrel{\sqsubseteq}\bar{G}$ was made arbitrarily, we have shown that $
	\mathcal{S}_c,\bar{F}\models^+_n \forall w(Rxyw\to ST_w(\chi))[x/\alpha, y/[\psi]_{\simeq}]$.
	Moreover, IH for Part 3 implies that $
	\mathcal{S}_c,\bar{F}\models^+_n Sy\wedge (\forall z)_O(Ezy\Leftrightarrow ST_z(\psi)))[y/[\psi]_{\simeq}]$. Summing up the two conjuncts and applying \eqref{Ax:10} and \eqref{E:T15}, we obtain that:
	\begin{align*}
		\mathcal{S}_c,\bar{F}\models^+_n \exists y((Sy \wedge (\forall z)_O(Ezy\Leftrightarrow ST_z(\psi)))\,\&\,\forall w(Rxyw\to ST_w(\chi)))[x/\alpha],
	\end{align*}
	or, in other words, that $\mathcal{S}_c,\bar{F}\models^+_n ST_x(\phi)[x/\alpha]$, as desired.
	
	For the converse, we argue by contraposition. Assume that $\psi\boxto\chi\notin \pi^1(end(\alpha))$. In this case, \eqref{E:boxto} implies the existence of $(\Xi,\Theta),(\Xi',\Theta')\in W_c$ such that $end(\alpha)\leq_c(\Xi,\Theta)$, $R_c((\Xi,\Theta), \|\psi\|_{\mathcal{M}_c}, (\Xi',\Theta'))$, and $\chi\notin \Xi'$. By Lemma \ref{L:global-functions}.2, there exists a $\beta \in Seq$ such that $end(\beta) = (\Xi,\Theta)$ and an $F \in \mathfrak{G}$ such that $F(\alpha) = \beta$. But then clearly $\bar{G} = \bar{F}^\frown(F) \in Glob$ and $\bar{F}\mathrel{\sqsubseteq}\bar{G}$. Moreover, Definition \ref{D:canonical-sheaf} implies that $\mathbb{F}_{\bar{F}\bar{G}}(\alpha) = \beta$. Next, $R_c((\Xi,\Theta), \|\psi\|_{\mathcal{M}_c}, (\Xi',\Theta'))$ implies that $\gamma = \beta^\frown(\psi, (\Xi',\Theta')) \in Seq$, whence, by Definition \ref{D:canonical-sheaf}, $(\beta, [\psi]_\simeq, \gamma)\in R^{\mathbb{M}^+}$. So, all in all we get that $
	\mathcal{S}_c,\bar{G}\models^+_n Rxyw[x/\mathbb{F}_{\bar{F}\bar{G}}(\alpha), y/[\psi]_\simeq, w/\gamma]$,
	and, on the other hand, since $\chi\notin \Xi' = \pi^1(end(\gamma))$, IH for Part 1 implies that $
	\mathcal{S}_c,\bar{G}\not\models^+_n ST_w(\chi)[w/\gamma]$.
	Therefore, given that $\bar{F}\mathrel{\sqsubseteq}\bar{G}$, it follows that:
	\begin{equation}\label{E:dagger}\tag{$\dagger$}
		\mathcal{S}_c,\bar{F}\not\models^+_{n} \forall w(Rxyw\to ST_w(\chi))[x/\alpha, y/[\psi]_\simeq]
	\end{equation}
	If now $a \in U$ is such that we have:
	\begin{equation}\label{E:ddagger}\tag{$\ddagger$}
		\mathcal{S}_c,\bar{F}\models^+_n Sy\wedge (\forall z)_O(Ezy\Leftrightarrow ST_z(\psi))[y/a]	
	\end{equation}
	then note that, by IH for Part 3 we also have $
	\mathcal{S}_c,\bar{F}\models^+_nSy\wedge(\forall z)_O(Ezy\Leftrightarrow ST_z(\psi))[y/[\psi]_\simeq]$. By Corollary \ref{C:indiscernibles} and \eqref{E:ddagger}, it follows that we must have $a = [\psi]_\simeq$ in this case. But then \eqref{E:dagger} implies that we must have
	$\mathcal{S}_c,\bar{F}\not\models^+_n \forall w(Rxyw\to ST_w(\chi))[x/\alpha, y/a]$.
	Since $a\in U$ was chosen arbitrarily under the condition given by \eqref{E:ddagger}, we conclude that
	\begin{align*}
		\mathcal{S}_c,\bar{F}\not\models^+_n \exists y(Sy \wedge (\forall z)_O(Ezy\Leftrightarrow ST_z(\psi))\wedge\forall w(Rxyw\to ST_w(\chi)))[x/\alpha],
	\end{align*}
	whence, by \eqref{E:T2}, \eqref{E:T14}, and \eqref{E:T15}, it follows that $\mathcal{S}_c,\bar{F}\not\models^+_n ST_x(\phi)[x/\alpha]$, as desired.
	
	It remains to consider Part 2. By Proposition \ref{P:truth-lemma}, we know that:
	\begin{align}
		\sim(\psi\boxto&\chi)\in \pi^1(end(\alpha)) \text{ iff } \mathcal{M}_c, end(\alpha)\models^- \psi\boxto\chi\notag\\
		&\text{ iff }(\exists(\Xi,\Theta)\in W_c)(R_c(end(\alpha), \|\psi\|_{\mathcal{M}_c}, (\Xi,\Theta))\text{ and }\mathcal{M}_c, (\Xi,\Theta)\models^- \chi)\notag\\
		&\text{ iff }(\exists(\Xi,\Theta)\in W_c)(R_c(end(\alpha), \|\psi\|_{\mathcal{M}_c}, (\Xi,\Theta))\text{ and }\sim\chi\in \Xi)\label{E:diamondto}\tag{$\boxto^-$}
	\end{align}
 Assume, first, that $\sim\psi\boxto\chi\in \pi^1(end(\alpha))$. Then, by \eqref{E:diamondto}, we can choose a $(\Xi,\Theta)\in W_c$ such that both $R_c(end(\alpha), \|\psi\|_{\mathcal{M}_c}, (\Xi,\Theta))$ and $\sim\chi\in \Xi$. The former means, by Definition \ref{D:standard-sequence}, that $\beta = \alpha^\frown(\psi, (\Xi,\Theta)) \in Seq$, and the latter means that $\sim\chi \in \pi^1(end(\beta))$ whence, by IH for Part 2, it follows that $\mathcal{S}_c,\bar{F}\models^-_n ST_w(\chi)[w/\beta]$.
	
Next, Definition \ref{D:canonical-sheaf} implies that $\mathcal{S}_c,\bar{F}\models^+_n Rxyw[x/\alpha, y/[\psi]_\simeq, w/\beta]$ and that, by IH for Part 3, we have $\mathcal{S}_c,\bar{F}\models^+_n Sy\wedge(\forall z)_O (Ezy\Leftrightarrow ST_z(\psi))[y/[\psi]_\simeq]$. 

If now $\bar{G}\in Glob$ is such that $\bar{G}\sqsupseteq\bar{F}$ and $a \in U$ is such that $\mathcal{S}_c,\bar{G}\models^+_n Sy\wedge(\forall z)_O (Ezy\Leftrightarrow ST_z(\psi))[y/a]$, then, by Lemma \ref{L:n4-standard}.1, we must have  $\mathcal{S}_c,\bar{G}\models^+_n Sy\wedge(\forall z)_O (Ezy\Leftrightarrow ST_z(\psi))[y/[\psi]_\simeq]$. By Corollary \ref{C:indiscernibles}, we get that $a = [\psi]_\simeq$. Moreover, Lemma \ref{L:n4-standard}.1 also implies that $\mathcal{S}_c,\bar{G}\models^+_n Rxyw[x/\mathbb{F}_{\bar{F}\bar{G}}(\alpha), y/a, w/\mathbb{F}_{\bar{F}\bar{G}}(\beta)]$ and $\mathcal{S}_c,\bar{G}\models^-_n ST_w(\chi)[w/\mathbb{F}_{\bar{F}\bar{G}}(\beta)]$. It follows that $
\mathcal{S}_c,\bar{G}\models^-_n \forall w(Rxyw\to ST_w(\chi))[x/\mathbb{F}_{\bar{F}\bar{G}}(\alpha), y/a]$, in other words, that:
$$
\mathcal{S}_c,\bar{G}\models^+_n \sim\forall w(Rxyw\to ST_w(\chi))[x/\mathbb{F}_{\bar{F}\bar{G}}(\alpha), y/a].
$$
In view of the choice of $\bar{G}\in Glob$ and $a \in U$, we have shown that:
$$
\mathcal{S}_c,\bar{F}\models^+_n \forall y(Sy\wedge(\forall z)_O (Ezy\Leftrightarrow ST_z(\psi)) \to \sim\forall w(Rxyw\to ST_w(\chi)))[x/\alpha],
$$
which, by \eqref{E:T8} and \eqref{E:T17} is the same as
$$
\mathcal{S}_c,\bar{F}\models^+_n \forall y\sim\sim(Sy\wedge(\forall z)_O (Ezy\Leftrightarrow ST_z(\psi)) \to \sim\forall w(Rxyw\to ST_w(\chi)))[x/\alpha],
$$
It remains to apply the definition of ampersand and \eqref{E:a5} to get that 
 $\mathcal{S}_c,\bar{F}\models^+_n \sim ST_x(\phi)[x/\alpha]$, or, equivalently that $\mathcal{S}_c,\bar{F}\models^-_n ST_x(\phi)[x/\alpha]$, as desired.
		
As for the converse of Part 2, we again argue by contraposition. Assume that $\sim(\psi\boxto\chi)\notin \pi^1(end(\alpha))$. In this case, \eqref{E:diamondto} implies that for every $(\Xi,\Theta)\in W_c$ such that $R_c(end(\alpha), \|\psi\|_{\mathcal{M}_c}, (\Xi,\Theta))$, we must have $\sim\chi\notin \Xi$. Now, if $\beta \in Seq$ is such that $(\alpha, [\psi]_\simeq, \beta) \in R^{\mathbb{M}^+}$, then, by Definition \ref{D:canonical-sheaf}, there must exist a $(\Gamma, \Delta) \in W_c$ and a $\theta \in [\psi]_\simeq$ such that $\beta = \alpha^\frown(\theta, (\Gamma, \Delta))$. Since $\beta \in Seq$, Definition \ref{D:standard-sequence} implies that $R_c(end(\alpha), \|\theta\|_{\mathcal{M}_c}, (\Gamma, \Delta))$; but, since $\theta \in [\psi]_\simeq$, the latter means that we must have $R_c(end(\alpha), \|\psi\|_{\mathcal{M}_c}, (\Gamma, \Delta))$. But then our assumption implies that $\sim\chi\notin \Gamma = \pi^1(end(\beta))$, whence, by IH for Part 2, we must have $\mathcal{S}_c,\bar{F}\not\models^-_n ST_w(\chi)[w/\beta]$. Thus we have shown that $\mathcal{S}_c,\bar{F}\not\models^+_n Rxyw\wedge\sim ST_w(\chi)[x/\alpha, y/[\psi]_\simeq, w/\beta]$ for every $\beta \in Seq$, which, by definition of $R^{\mathbb{M}^+}$, implies that 
$\mathcal{S}_c,\bar{F}\not\models^-_n Rxyw\to ST_w(\chi)[x/\alpha, y/[\psi]_\simeq, w/b]$ for every $b \in U$. In view of the $\mathsf{QN4}$ semantics of $\forall$ and $\sim$, the latter can be equivalently reformulated as $\mathcal{S}_c,\bar{F}\not\models^+_n \sim\forall w(Rxyw\to ST_w(\chi))[x/\alpha, y/[\psi]_\simeq]$.
Since, by IH for Part 3, we also have $\mathcal{S}_c,\bar{F}\models^+_n Sy\wedge(\forall z)_O (Ezy\Leftrightarrow ST_z(\psi))[y/[\psi]_\simeq]$, we have shown that
$$
\mathcal{S}_c,\bar{F}\not\models^+_n (Sy\wedge(\forall z)_O (Ezy\Leftrightarrow ST_z(\psi)))\to \sim\forall w(Rxyw\to ST_w(\chi))[x/\alpha, y/[\psi]_\simeq],
$$
which, in view of the definition of ampersand, is the same as
$$
\mathcal{S}_c,\bar{F}\not\models^-_n (Sy\wedge(\forall z)_O (Ezy\Leftrightarrow ST_z(\psi)))\,\&\,\forall w(Rxyw\to ST_w(\chi))[x/\alpha, y/[\psi]_\simeq].
$$ 
The latter, by the $\mathsf{QN4}$ semantics of the existential quantifier entails that 
$$
\mathcal{S}_c,\bar{F}\not\models^-_n \exists y(Sy\wedge(\forall z)_O (Ezy\Leftrightarrow ST_z(\psi)))\,\&\,\forall w(Rxyw\to ST_w(\chi))[x/\alpha],
$$ 
in other words, that $\mathcal{S}_c,\bar{F}\not\models^-_n ST_x(\phi)[x/\alpha]$, as desired.
\end{proof}
Before we move on, we would like to draw a corollary from our lemma:
\begin{corollary}\label{C:simplified}
	Let $\bar{F} \in Glob$, $\alpha \in Seq$, $x \in Ind$, and let $\psi,\chi\in \mathcal{CN}$. Then:
	$$
	\mathcal{S}_c,\bar{F} \models^+_n ST_x(\psi\boxto\chi) \Leftrightarrow\forall w(Rxyw\to ST_w(\chi))[x/\alpha, y/[\psi]_\simeq].
	$$			
\end{corollary}
\begin{proof}
	We show, first, that
	$$
	\mathcal{S}_c,\bar{F} \models^+_n ST_x(\psi\boxto\chi) \leftrightarrow\forall w(Rxyw\to ST_w(\chi))[x/\alpha, y/[\psi]_\simeq].
	$$
	Right-to-left direction is trivial by Lemma \ref{L:sheaf-truth}.3. As for the converse, if $\mathcal{S}_c,\bar{F} \models^+_n ST_x(\psi\boxto\chi)[x/\alpha]$, then we can choose an $a \in U$ such that:
	\begin{equation*}
		\mathcal{S}_c,\bar{F}\models^+_n Sy \wedge (\forall z)_O(Ezy\Leftrightarrow ST_z(\psi))\,\&\,\forall w(Rxyw\to ST_w(\chi))[x/\alpha, y/a]
	\end{equation*}
	But then, by Corollary \ref{C:indiscernibles} and Lemma \ref{L:sheaf-truth}.3, we must have $a = [\psi]_\simeq$, so that
	
	\noindent $\mathcal{S}_c,\bar{F} \models^+_n \forall w(R(x,y,w)\to ST_w(\chi))[x/\alpha, y/[\psi]_\simeq]$ follows. We need to show, next, that
	$$
	\mathcal{S}_c,\bar{F} \models^+_n \sim ST_x(\psi\boxto\chi) \leftrightarrow\sim\forall w(Rxyw\to ST_w(\chi))[x/\alpha, y/[\psi]_\simeq].
	$$ 
	The left-to-right direction easily follows by \eqref{E:a1}, \eqref{E:a5}, and Lemma \ref{L:sheaf-truth}.3. If 
	
	\noindent$\mathcal{S}_c,\bar{F} \models^-_n \forall w(R(x,y,w)\to ST_w(\chi))[x/\alpha, y/[\psi]_\simeq]$, then let $\bar{G} \in Glob$ and an $a \in U$ be such that $\bar{G}\sqsupseteq \bar{F}$ and $
	\mathcal{S}_c,\bar{G}\models^+_n Sy \wedge (\forall z)_O(Ezy\Leftrightarrow ST_z(\psi))[y/a]$. By Lemma \ref{L:n4-standard}.1, we also have $
	\mathcal{S}_c,\bar{G} \models^-_n \forall w(R(x,y,w)\to ST_w(\chi))[x/\mathbb{F}_{\bar{F}\bar{G}}(\alpha), y/[\psi]_\simeq]$. Next, Corollary \ref{C:indiscernibles} and Lemma \ref{L:sheaf-truth}.3 imply that $a = [\psi]_\simeq$, so that $\mathcal{S}_c,\bar{G} \models^-_n \forall w(R(x,y,w)\to ST_w(\chi))[x/\mathbb{F}_{\bar{F}\bar{G}}(\alpha), y/a]$ follows. Since $\bar{G}\sqsupseteq \bar{F}$ and $a \in U$ were chosen arbitrarily, it follows that
	$$
		\mathcal{S}_c,\bar{F} \models^+_n \forall y(Sy \wedge (\forall z)_O(Ezy\Leftrightarrow ST_z(\psi))\to \sim \forall w(R(x,y,w)\to ST_w(\chi)))[x/\alpha],
	$$
	or, equivalently, that $\mathcal{S}_c,\bar{F} \models^+_n \sim ST_x(\psi\boxto\chi)[x/\alpha]$.
\end{proof}
It only remains now to show that $\mathcal{S}_c$ is a model of $Th$:
\begin{lemma}\label{L:th}
	For every $\bar{F} \in Glob$, we have $\mathcal{S}_c, \bar{F}\models^+_n Th$.
\end{lemma}
\begin{proof}
	That the Nelsonian sheaf $\mathcal{S}_c$ satisfies \eqref{E:th5} was shown in Lemma \ref{L:indiscernibles}. The satisfaction of \eqref{E:th6} follows from Lemma \ref{L:sheaf-truth} (one needs to instantiate $y$ to $[p]_\simeq$). We consider the remaining parts of $Th$ in more detail below:
	
	\eqref{E:th8}. Let $\ast \in \{\wedge, \to\}$, $\bar{F} \in Glob$, and $a, b \in S^{\mathbb{M}^+}$. Then, by Definition \ref{D:canonical-sheaf}, there must exist some $\phi, \psi \in \mathcal{CN}$ such that $a = [\phi]_\simeq$ and $b = [\psi]_\simeq$. But then Lemma \ref{L:sheaf-truth}.3 implies that we have all of the following:
	\begin{align}
		\mathcal{S}_c, \bar{F}&\models^+_n Sx \wedge (\forall w)_O(Ewx\Leftrightarrow ST_w(\phi))[x/a]\label{E:phi}\tag{$\S$}\\
		\mathcal{S}_c, \bar{F}&\models^+_n Sy \wedge (\forall w)_O(Ewy\Leftrightarrow ST_w(\psi))[y/b]\label{E:psi}\tag{$\P$}\\
		\mathcal{S}_c, \bar{F}&\models^+_n Sz \wedge (\forall w)_O(Ewz\Leftrightarrow ST_w(\phi)\ast ST_w(\psi))[z/[\phi\ast\psi]_\simeq]\notag
	\end{align}
	whence $\mathcal{S}_c, \bar{F}\models_{fo} \eqref{E:th8}$ clearly follows by \eqref{E:T18}. The argument for the satisfaction of \eqref{E:th7} is similar
	
	\eqref{E:th9}. Again, let $\bar{F} \in Glob$ and $a, b \in S^{\mathbb{M}^+}$. Then let $\phi, \psi \in \mathcal{CN}$ be such that $a = [\phi]_\simeq$ and $b = [\psi]_\simeq$. Lemma \ref{L:sheaf-truth}.3 implies that both \eqref{E:phi} and \eqref{E:psi} hold, and that we have:
	$$
	\mathcal{S}_c, \bar{F}\models^+_n Sz \wedge (\forall w)_O(Ewz\Leftrightarrow ST_w(\phi\boxto\psi))[z/[\phi\boxto\psi]_\simeq].
	$$
	By Corollary \ref{C:simplified} and \eqref{E:T18}, it follows now that we must have:
	$$
	\mathcal{S}_c, \bar{F}\models^+_n Sz \wedge (\forall w)_O(Ewz\Leftrightarrow \forall w'(Rwxw'\to ST_{w'}(\psi)))[x/a,z/[\phi\boxto\psi]_\simeq].
	$$
	whence $\mathcal{S}_c, \bar{F}\models^+_n \eqref{E:th9}$ clearly follows. 
\end{proof}
We are now finally in a position to prove a converse to Proposition \ref{P:easy}:
\begin{proposition}\label{P:hard}
	For all $\Gamma, \Delta \subseteq \mathcal{CN}$ and for every $x \in Ind$, if $Th, ST_x(\Gamma)\models_{\mathsf{QN4}}ST_x(\Delta)$, then  $\Gamma \models_{\mathsf{N4CK}} \Delta$.
\end{proposition}
\begin{proof}
	We argue by contraposition. If $\Gamma \not\models_{\mathsf{N4CK}} \Delta$, then $(\Gamma, \Delta)$ must be $\mathsf{N4CK}$-satisfiable, and therefore, applying the standard Lindenbaum construction (cf. \cite[Lemma 8]{nelsonian}), we can find a $(\Gamma', \Delta')\in W_c$ such that $(\Gamma', \Delta')\supseteq (\Gamma, \Delta)$. By Definition \ref{D:standard-sequence}, we have $(\Gamma', \Delta')\in Seq$, therefore, Lemma \ref{L:sheaf-truth}.1 and Lemma \ref{L:th} together imply that $\mathcal{S}_c, \Lambda \models^+_n (Th \cup ST_x(\Gamma), ST_x(\Delta)[x/(\Gamma', \Delta')]$, 
	or, equivalently, that $Th, ST_x(\Gamma)\not\models_{\mathsf{QN4}}ST_x(\Delta)$, as desired.
\end{proof}
We can now formulate and prove the main result of this section:
\begin{theorem}\label{T:foil}
	For all $\Gamma, \Delta \subseteq \mathcal{CN}$ and for every $x \in Ind$, we have $\Gamma \models_{\mathsf{N4CK}} \Delta$ iff $Th\cup ST_x(\Gamma)\models_{\mathsf{QN4}}ST_x(\Delta)$.
\end{theorem}
\begin{proof}
	By Propositions \ref{P:easy} and \ref{P:hard}.
\end{proof}
Before we end this section, we would like to mention, that the proof of Theorem \ref{T:foil} not only allows us to completely mirror, on the level of conditional logic, the selection of the embedding results for modal logic reported in Section \ref{S:idea}, but also to improve on \cite[Proposition 7]{odintsovwansing} for the $\mathsf{N4}$-based modal logic $\mathsf{FSK}^d$; namely, if we repeat the constructions of this section adapting them to the modal case (basically, using $\sigma\tau_x$ in place of $ST_x$ and throwing away all the parts dealing with $Th$ and the existence of formula-defined truth-sets) then we arrive at the following
\begin{theorem}\label{T:modal}
	For all $\Gamma, \Delta \subseteq \mathcal{MD}$ and for every $x \in Ind$, we have $\Gamma \models_{\mathsf{FSK}^d} \Delta$ iff $\sigma\tau_x(\Gamma)\models_{\mathsf{QN4}}\sigma\tau_x(\Delta)$.
\end{theorem}
This alternative embedding result shows that the very definition of translation used for the classical logic $\mathsf{K}$, also embeds $\mathsf{FSK}^d$ into $\mathsf{QN4}$; and this time our embedding commutes with the propositional connectives of $\mathsf{N4}$.

That such a result would fall out from our proof of Theorem \ref{T:foil} should not surprise us, as $\mathsf{FSK}^d$ was shown to be the modal companion to $\mathsf{N4CK}$ in \cite[Lemma 16]{nelsonian}. 
 
 \section{Discussion, conclusion, and future work}\label{S:conclusion}
Theorem \ref{T:foil} shows that $\mathsf{N4CK}$ is faithfully embedded into $\mathsf{QN4}$ by a (classically equivalent) variant of the classical formalization of the Chellas semantics for $\mathsf{CK}$ based on the pair $(Th, ST_x)$ for any given $x \in Ind$. The conditional logic $\mathsf{N4CK}$ can be thus seen as a result of a Nelsonian reading of the classical conditional logic $\mathsf{CK}$, and, therefore, as one plausible counterpart to $\mathsf{CK}$ on the basis of Nelson's logic of strong negation. This result allows to view $\mathsf{N4CK}$ as a strong candidate for the role of a natural minimal $\mathsf{N4}$-based logic of conditionals, and thus completes a series of other arguments to that effect presented earlier in \cite{nelsonian}. 

Our nearest research plans include extension of this series to other non-classical constructive logics. An especially fitting candidate for such an extension appears to be the negation-inconsistent constructive logic $\mathsf{C}$ introduced by H. Wansing in \cite{w}, especially given that the subject of conditional logics on the basis of the propositional fragment of $\mathsf{C}$ has already seen its first rather intriguing steps in \cite{wu}, and the methods of the current paper seem to open a way to a considerable refinement of these first results. However, one should also keep in mind that $\mathsf{C}$ is not a sublogic of $\mathsf{CL}$ and is therefore not in the scope of the general criterion put forward in the final paragraphs of Section \ref{S:conditional}. 

Before we end this paper, we would like to observe that $Th$ is but one of the infinitely many theories that can be used to prove results like Theorem \ref{T:foil}. More precisely, let us call a $T\subseteq \mathcal{FO}^\emptyset$ $\mathsf{N4}$-\textit{conditional} iff it can be substituted for $Th$ in Theorem \ref{T:foil} above, in other words, iff for all $\Gamma, \Delta \subseteq \mathcal{CN}$ and for every $x \in Ind$, it is true that $\Gamma \models_{\mathsf{N4CK}} \Delta$ iff $T\cup ST_x(\Gamma)\models_{\mathsf{QN4}}ST_x(\Delta)$. Then Theorem  \ref{T:foil} can be reformulated as stating that $Th$ is $\mathsf{N4}$-conditional. But we can also show that:
\begin{corollary}\label{C:foil}
	Let $Th'\subseteq \mathcal{FO}^\emptyset$ be such that $\mathcal{S}_c\models^+_n Th'\supseteq Th$. Then $Th'$ is $\mathsf{N4}$-conditional.
\end{corollary}
\begin{proof}
	Assume the hypothesis. If $\Gamma, \Delta \subseteq \mathcal{CN}$ and $x \in Ind$ are such that $\Gamma \models_{\mathsf{N4CK}} \Delta$, then, by Proposition \ref{P:easy}, we must have $Th\cup ST_x(\Gamma)\models_{\mathsf{QN4}}ST_x(\Delta)$, whence, by monotonicity of $\models_{\mathsf{QN4}}$ and $Th'\supseteq Th$, also $Th'\cup ST_x(\Gamma)\models_{\mathsf{QN4}}ST_x(\Delta)$. In the other direction, we can just repeat the proof of Proposition \ref{P:hard} using the fact that $\mathcal{S}_c\models^+_n Th'$.
\end{proof}
Corollary \ref{C:foil} allows us to strengthen the connection between Theorem \ref{T:foil} and the earlier embedding result \cite[Theorem 2]{olkhovikov} reported for the intuitionistic conditional logic $\mathsf{IntCK}$. Namely, we can define a theory $Th^i\supseteq Th$ which is both $\mathsf{N4}$-conditional and classically equivalent to the theory mentioned in \cite{olkhovikov}. To obtain $Th^i$, we need to extend $Th$ with the following sentences, for every $p^1 \in Prop$:
\begin{align*}
\forall x(Sx\vee Ox),\,&\forall x\sim(Sx\wedge Ox),\,\forall x(px\to Ox),\,
\forall x\forall y(Exy\to (Ox\wedge Sy))\\
&\forall x\forall y\forall z(Rxyz\to (Ox\wedge Sy\wedge Oz))	
\end{align*}
It is clear that all of these sentences are in the scope of Corollary \ref{C:foil}; but they are by no means the only interesting additions to $Th$ made available by the said  corollary. It is easy to see, for example, that we can also  replace the main $\to$ in \eqref{E:th5} with $\Rightarrow$ without affecting Theorem \ref{T:foil}.

Of course, the same trick is possible w.r.t. Proposition \ref{P:ck-into-cl} which was our blueprint for Theorem \ref{T:foil}. Indeed, we can call a $T\subseteq \mathcal{FO}^\emptyset$ $\mathsf{CL}$-\textit{conditional} iff for all $\Gamma, \Delta \subseteq \mathcal{CN}$ and for every $x \in Ind$, it is true that $\Gamma \models_{\mathsf{CK}} \Delta$ iff $T\cup st_x(\Gamma)\models_{\mathsf{CL}}st_x(\Delta)$. Then Proposition  \ref{P:ck-into-cl} shows us that $Th_{ck}$ is $\mathsf{CL}$-conditional, and we can also strengthen it as follows:
\begin{corollary}\label{C:ck-into-cl}
	Let $T\subseteq \mathcal{FO}^\emptyset$ be such that $\mathcal{M}^{cl}\models_c T\supseteq Th_{ck}$. Then $T$ is $\mathsf{CL}$-conditional.
\end{corollary}
The proof is similar to the above proof of Corollary \ref{C:foil}. However, now that we have established the polymorphism of both $\mathsf{N4}$-conditional and $\mathsf{CL}$-conditional first-order theories, further (and increasingly deeper) questions suggest themselves. For example, what is the relation between the theories mentioned in Corollary \ref{C:foil} to the theories mentioned in Corollary \ref{C:ck-into-cl}? It appears that at least the following is highly plausible:

\textbf{Conjecture}. \textit{Let $Th'\subseteq \mathcal{FO}^\emptyset$ be such that $\mathcal{S}_c\models^+_n Th'\supseteq Th$. Then $\mathcal{M}^{cl}\models_c Th'$; in particular, $Th'$ is also $\mathsf{CL}$-conditional.}

If the above conjecture is true, then it might still be the case, that some subtheories of the complete theory of $\mathcal{M}^{cl}$ that extend $Th_{ck}$ are such that none of their classical equivalents both extends $Th$ and is a subtheory of the complete $\mathsf{N4}$-theory of $\mathcal{S}_c$. Moreover, this discrepancy between the two sets of theories can be either coincidental or necessary. Indeed, notice that the definitions of $\mathcal{M}^{cl}$ and $\mathcal{S}_c$ are not completely forced by our main result: at least the definition of $\mathbb{M}^-$ could have been given differently without affecting Theorem \ref{T:foil}. So is it possible to revise those definitions in such a way that one gets a perfect match between the theories arising in the two structures? We sum these considerations up as the following

\textbf{Open question 1}. \textit{Assume the truth of the above conjecture. Is it true that for every $T\subseteq \mathcal{FO}^\emptyset$ we have $\mathcal{M}^{cl}\models_c T\supseteq Th_{ck}$ iff there exists a $T'\subseteq \mathcal{FO}^\emptyset$ such that $T'$ is equivalent to $T$ over $\mathsf{CL}$ and $\mathcal{S}_c\models^+_n T'\supseteq Th$? If not, then is it possible to tweak the definitions of both $\mathcal{M}^{cl}$ and $\mathcal{S}_c$ in such a way that the answer to this question becomes affirmative and both Proposition \ref{P:ck-into-cl} and Theorem \ref{T:foil} are still true relative to these new definitions?}

However, the above question may seem a bit too narrow in that we only ask about the relation between the theories arising in two particular structures (even when we allow those structures to differ from $\mathcal{M}^{cl}$ and $\mathcal{S}_c$ as defined above). A much more interesting question is whether we can establish any meaningful relation between the $\mathsf{N4}$-conditional theories and $\mathsf{CL}$-conditional theories in general. In other words:

\textbf{Open question 2}. \textit{Let $T\subseteq \mathcal{FO}^\emptyset$ be $\mathsf{N4}$-conditional. Is it then also $\mathsf{CL}$-conditional?} 

\textbf{Open question 3}. \textit{Let $T\subseteq \mathcal{FO}^\emptyset$ be $\mathsf{CL}$-conditional. Is it always the case that there exists some $T'\subseteq \mathcal{FO}^\emptyset$ which is equivalent to $T$ over $\mathsf{CL}$ and also $\mathsf{N4}$-conditional?}

It is much harder to formulate the right inverses to the Open questions 2 and 3. It seems that not all $\mathsf{CL}$-conditional theories are also $\mathsf{N4}$-conditional: one example is likely to be provided by $Th_{ck}$ itself. However, some sort of a canonical reformulation $\rho(T)$ of a classical theory $T$ might be still possible with the effect that $\rho(T)$ and $T$ are classically equivalent and, in case $T$ is $\mathsf{CL}$-conditional, $\rho(T)$ is $\mathsf{N4}$-conditional. In case such a reformulation technique is found and given in an effective way, the above conjectures and open questions can be extended by similar items asking about the format of the possible match-up between $\mathsf{CL}$-conditional theories and their canonical reformulations.

\begin{appendices}
\section{Proofs of some results about $\mathsf{QN4}$}\label{App:N4}
To simplify the notation, we fix an $\Omega \subseteq \Pi$ and set $\mathcal{FO}^+:= \mathcal{FO}^+(\Omega^\pm \cup \{\epsilon^2\})$, $\mathbb{I}:= \mathbb{I}(\Omega^\pm \cup \{\epsilon^2\})$, $EP_i:= EP_i(\Omega^\pm \cup \{\epsilon^2\})$, $\mathsf{QIL}^+:= \mathsf{QIL}^+(\Omega^\pm \cup \{\epsilon^2\})$, and $\mathfrak{QIL}^+:= \mathfrak{QIL}^+(\mathcal{FO}^+)$. Moreover, we set $\mathcal{FO}:= \mathcal{FO}(\Omega)$, $At:= At(\Omega)$, $Lit:= Lit(\Omega)$, $\mathbb{N}4:= \mathbb{N}4(\Omega)$, and $EP_n:= EP_i(\Omega)$ for this Appendix.

\subsection{Proof of Proposition \ref{P:intuitionistic-embedding}}\label{App:intuitionistic-embedding}
The proof of Proposition \ref{P:intuitionistic-embedding} is complicated by the fact that $Tr$ is not injective; we have, for example
$$
Tr(\sim(v_0 \equiv v_1 \to v_1 \equiv v_2)) = (v_0 \equiv v_1 \wedge \epsilon(v_1,v_2)) = Tr(v_0 \equiv v_1 \wedge \sim(v_1 \equiv v_2)). 
$$
Therefore, we need to choose a subset of $\mathcal{FO}$ which can both serve as a representative for the whole set $\mathcal{FO}$ and provide a basis for an injective restriction of $Tr$. The next definition singles out this class:
\begin{definition}\label{D:nnf}
	A $\phi \in \mathcal{FO}$ is in \textit{negation normal form} iff, for every $\sim\psi \in Sub(\phi)$ it is true that $\psi\in At$.
\end{definition}
The following Lemma sums up the properties of negation normal forms in $\mathsf{N4}$ required for our argument:
\begin{lemma}\label{L:nnf}
	The following statements hold:
	\begin{enumerate}
		\item $tr:= Tr\upharpoonright\{\phi \in \mathcal{FO}\mid \phi\text{ is in negation normal form}\}$ is injective.
		
		\item For every $\phi \in \mathcal{FO}$, define $NNF(\phi)\in \mathcal{FO}$ by the following induction on the construction of $\phi$:
		\begin{align*}
			NNF(\phi)&:= \phi &&\text{$\phi$ is a literal}\\
			NNF(\psi\ast\chi)&:= NNF(\psi)\ast NNF(\chi)&&\ast \in \{\wedge, \vee, \to\}\\
			NNF(Qx\psi)&:= QxNNF(\psi)&&Q \in \{\forall, \exists\}\\
			NNF(\sim\sim\psi)&:= NNF(\psi)\\
			NNF(\sim(\psi\star\chi))&:= NNF(\sim\psi)\ast NNF(\sim\chi) &&\{\star, \ast\} = \{\wedge, \vee\}\\
			NNF(\sim(\psi\to\chi))&:= NNF(\psi)\wedge NNF(\sim\chi);\\
			NNF(\sim Qx\psi)&:= Q'xNNF(\sim\psi) &&\{Q,Q'\} = \{\forall, \exists\}
		\end{align*} 
		Then $NNF(\phi)$ is in negation normal form, and $\phi\leftrightarrow NNF(\phi)\in\mathsf{QN4}$.
		
		\item For every $\phi \in \mathcal{FO}$, we have $Tr(\phi) = Tr(NNF(\phi))$.
	\end{enumerate}
\end{lemma}
\begin{proof}[Proof (a sketch)]
	(Part 1) Observe that $Tr$ is injective on literals and commutes with the connectives and quantifiers for the non-negated complex formulas.
	
	(Part 2) That $NNF(\phi)$ is in the negation normal form is immediate from the definition. As for the provability of $\phi\leftrightarrow NNF(\phi)$, we reason by induction on the construction of $\phi \in \mathcal{FO}$, using Lemma \ref{L:intuitionistic-inclusion} plus the axioms \eqref{E:a1}--\eqref{E:a6}.
	
	(Part 3) By a straightforward induction on the construction of $\phi \in \mathcal{FO}$.
\end{proof}
We are now ready to prove our proposition:
\begin{proof}[Proof of Proposition \ref{P:intuitionistic-embedding}]
	(Left-to-right) Assume that $\Gamma\models_{\mathsf{QN4}}\Delta$ and choose any $\mathfrak{QN4}$-deduction $\bar{\chi}_r$ from the premises in $\Gamma$ such that, for some $\psi_1,\ldots,\psi_s\in \Delta$ we have $\chi_r = \psi_1\vee\ldots\vee\psi_s$. Consider $Tr(\bar{\chi}_r)$: the translation of every axiom of $\mathfrak{QIL}^+$ is again an instance of the same axiom and the same is true for the applications of every rule in $\mathfrak{QIL}^+$. Finally, all the instances of \eqref{E:a1}--\eqref{E:a6} get translated into formulas of the form $\psi\leftrightarrow\psi$ for an appropriate $\psi\in \mathcal{FO}^+$; all of these translations are also clearly provable in $\mathfrak{QIL}^+$. Therefore, $Tr(\bar{\chi}_r)$ is straightforwardly extendable to a $\mathfrak{QIL}^+$-deduction of $Tr(\chi_r) = Tr(\psi_1\vee\ldots\vee\psi_s) = Tr(\psi_1)\vee\ldots\vee Tr(\psi_s)$ 
	from premises in $Tr(\Gamma)$. We have thus shown that $Tr(\Gamma)\models_{\mathsf{QIL}^+}Tr(\Delta)$.
	
	(Right-to-left) Assume that $Tr(\Gamma)\models_{\mathsf{QIL}^+}Tr(\Delta)$ and choose any $\mathfrak{QIL}^+$-deduction $\bar{\chi}_r$ from the premises in $Tr(\Gamma)$ such that, for some $\psi_1,\ldots,\psi_s\in \Delta$ we have $\chi_r =  Tr(\psi_1)\vee\ldots\vee Tr(\psi_s)$. Consider $tr^{-1}(\bar{\chi}_r)$. Again, every instance of a $\mathfrak{QIL}^+$-axiom is transferred by $tr^{-1}$ into an instance of the same axiom and that the applications of all the rules in $\mathfrak{QIL}^+$ are likewise preserved by $tr^{-1}$. Therefore, $tr^{-1}(\bar{\chi}_r)$ must be a deduction of $tr^{-1}(\chi_r)$ in $\mathfrak{QIL}^+$ and hence also in $\mathfrak{QN4}$. Note, next, that for every $\theta \in \mathcal{FO}$ we have:
	\begin{align*}
		tr^{-1}(Tr(\theta)) &= tr^{-1}(Tr(NNF(\theta))) &&\text{by Lemma \ref{L:nnf}.3}\\
		&= tr^{-1}(tr(NNF(\theta)))&&\text{by Lemma \ref{L:nnf}.2}\\
		&= NNF(\theta)&&\text{by Lemma \ref{L:nnf}.1}
	\end{align*}
	Therefore $tr^{-1}(\bar{\chi}_r)$ is an $\mathfrak{QN4}$-deduction of
	$$
	tr^{-1}(\chi_r) = tr^{-1}(Tr(\psi_1)\vee\ldots\vee Tr(\psi_s)) = tr^{-1}Tr(\psi_1)\vee\ldots\vee tr^{-1}Tr(\psi_s) = NNF(\psi_1)\vee\ldots\vee NNF(\psi_s)
	$$
	from premises in $NNF(\Gamma)$. In view of Lemma \ref{L:nnf}.2, it is straightforward to extend $tr^{-1}(\bar{\chi}_r)$ to an $\mathfrak{QN4}$-deduction of $\psi_1\vee\ldots\vee\psi_s$ from the premises in $\Gamma$.
\end{proof}

\subsection{Proof of Proposition \ref{P:nelsonian-sheaves}}\label{App:nelsonian-sheaves}
Proposition \ref{P:nelsonian-sheaves} follows from Proposition \ref{P:intuitionistic-embedding} together with the following two lemmas:
\begin{lemma}\label{L:sheaf-embedding1}
	Let $\mathcal{S}\in \mathbb{N}4$, and let $\mathcal{S}^i = (W^i, \leq^i, \mathtt{M}^i, \mathtt{H}^i)$ be such that $W^i := W$, $\leq^i:= \leq$, $\mathtt{H}^i:= \mathtt{H}$, and that, for every $\mathbf{w} \in W^i = W$, we have $U^{\mathtt{M}^i_\mathbf{w}}:= U_\mathbf{w} = U^{\mathtt{M}^+_\mathbf{w}} = U^{\mathtt{M}^-_\mathbf{w}}$. Finally, for all $\mathbf{w} \in W$ and $P^n \in \Pi$, we set that
	\begin{align*}
		P^{\mathtt{M}^i_\mathbf{w}}_+&:=P^{\mathtt{M}^+_\mathbf{w}};&&P^{\mathtt{M}^i_\mathbf{w}}_-:= P^{\mathtt{M}^-_\mathbf{w}};&&\epsilon^{\mathtt{M}^i_\mathbf{w}}:= \epsilon^{\mathtt{M}^-_\mathbf{w}}.
	\end{align*}
	Then all of the following statements hold:
	\begin{enumerate}
		\item $\mathcal{S}^i \in \mathbb{I}$.
		
		\item For every $\mathbf{w} \in W$ and every function $f$, we have $(\mathcal{S}, \mathbf{w}, f) \in EP_{n}$ iff $(\mathcal{S}^i, \mathbf{w}, f) \in EP_{i}$.
		
		\item For every $\mathbf{w} \in W$, every $f$ such that $(\mathcal{S}, \mathbf{w}, f) \in EP_{n}$, and every $\phi \in \mathcal{FO}$, we have $\mathcal{S}, \mathbf{w}\models^+_{n} \phi[f]$ iff $\mathcal{S}^i, \mathbf{w}\models_i Tr(\phi)[f]$.
	\end{enumerate}	
\end{lemma}
\begin{proof}
	Parts 1 and 2 follow trivially by definition. Part 3 is proved by a straightforward induction on the construction of $\phi\in \mathcal{FO}$ for all $\mathbf{w}$ and $f$ such that $(\mathcal{S}, \mathbf{w}, f) \in EP_n$. We consider some typical cases in this induction in more detail.
	
	\textit{Case 1}. $\phi  =  P(\bar{x}_m)$ for some $m \in \omega$, $P^m \in \Pi$ and $\bar{x}_m\in Ind^m$. But then:
	\begin{align*}
		\mathcal{S}, \mathbf{w}\models^+_n \phi[f] \text{ iff } f(\bar{x}_m)\in 	P^{\mathtt{M}^+_\mathbf{w}} \text{ iff } f(\bar{x}_m)\in P_+^{\mathtt{M}^i_\mathbf{w}}
		\text{ iff } \mathcal{S}^i, \mathbf{w}\models_i P_+(\bar{x}_m)[f]\text{ iff } \mathcal{S}^i, \mathbf{w}\models_i Tr(\phi)[f].
	\end{align*}
	On the other hand, we have:
	\begin{align*}
		\mathcal{S}, \mathbf{w}\models^+_n \sim\phi[f] \text{ iff } \mathcal{S}, \mathbf{w}\models^-_n \phi[f] &\text{ iff } f(\bar{x}_m)\in 	P^{\mathtt{M}^-_\mathbf{w}} \text{ iff } f(\bar{x}_m)\in P_-^{\mathtt{M}^i_\mathbf{w}}\\
		&\text{ iff } \mathcal{S}^i, \mathbf{w}\models_i P_-(\bar{x}_,)[f]\text{ iff } \mathcal{S}^i, \mathbf{w}\models_i Tr(\sim\phi)[f].
	\end{align*}
	
	\textit{Case 2}. $\phi = (\psi \to \chi)$. Then we have:
	\begin{align*}
		\mathcal{S}, \mathbf{w}\models^+_n \phi[f] &\text{ iff }(\forall \mathbf{v}\geq \mathbf{w})(\mathcal{S},\mathbf{v}\not\models^+_{n}\psi[f\circ\mathtt{H}_{\mathbf{w}\mathbf{v}}]\text{ or }\mathcal{S}, \mathbf{v}\models^+_{n}\chi[f\circ\mathtt{H}_{\mathbf{w}\mathbf{v}}])\\
		&\text{ iff }(\forall \mathbf{v}\geq^i \mathbf{w})(\mathcal{S}^i,\mathbf{v}\not\models_iTr(\psi)[f\circ\mathtt{H}^i_{\mathbf{w}\mathbf{v}}]\text{ or }\mathcal{S}^i, \mathbf{v}\models_iTr(\chi)[f\circ\mathtt{H}^i_{\mathbf{w}\mathbf{v}}])&&\text{by IH}\\
		&\text{ iff } \mathcal{S}^i, \mathbf{w}\models_i (Tr(\psi)\to Tr(\chi))[f] \text{ iff } \mathcal{S}^i, \mathbf{w}\models_i Tr(\phi)[f].
	\end{align*}
	On the other hand, we have:
	\begin{align*}
		\mathcal{S}, \mathbf{w}\models^+_n \sim\phi[f] &\text{ iff } \mathcal{S},\mathbf{w}\models^+_{n}\psi[f]\text{ and }\mathcal{S}, \mathbf{v}\models^-_{n}\chi[f]\\
		&\text{ iff } \mathcal{S},\mathbf{w}\models^+_{n}\psi[f]\text{ and }\mathcal{S}, \mathbf{w}\models^+_{n}\sim\chi[f]\\
		&\text{ iff } \mathcal{S}^i,\mathbf{w}\models_iTr(\psi)[f]\text{ and }\mathcal{S}^i, \mathbf{w}\models_iTr(\sim\chi)[f]&&\text{by IH}\\
		&\text{ iff } \mathcal{S}^i, \mathbf{w}\models_i Tr(\psi)\wedge Tr(\sim\chi)\text{ iff } \mathcal{S}^i, \mathbf{w}\models_i Tr(\phi)[f].
	\end{align*}
	
	\textit{Case 3}. $\phi = \forall x\psi$. Then we have:
	\begin{align*}
		\mathcal{S}, \mathbf{w}\models^+_n \phi[f] &\text{ iff } (\forall\mathbf{v}\geq \mathbf{w})(\forall a\in U_{\mathbf{v}})(\mathcal{S},\mathbf{v}\models^+_{n}\psi[(f\circ\mathtt{H}_{\mathbf{w}\mathbf{v}})[x/a]])\\
		&\text{ iff } (\forall\mathbf{v}\geq^i \mathbf{w})(\forall a\in U^{\mathtt{M}^i_\mathbf{v}})(\mathcal{S}^i,\mathbf{v}\models_iTr(\psi)[(f\circ\mathtt{H}^i_{\mathbf{w}\mathbf{v}})[x/a]])&&\text{by IH}\\
		&\text{ iff } \mathcal{S}^i, \mathbf{w}\models_i \forall xTr(\psi)[f]\text{ iff } \mathcal{S}^i, \mathbf{w}\models_i Tr(\phi)[f].
	\end{align*}
	On the other hand, we have: 
	\begin{align*}
		\mathcal{S}, \mathbf{w}\models^+_n \sim\phi[f] &\text{ iff } (\exists a\in U_{\mathbf{w}})(\mathcal{S},\mathbf{w}\models^-_{n}\psi[f[x/a]])\\
		&\text{ iff } (\exists a\in U_{\mathbf{w}})(\mathcal{S},\mathbf{w}\models^+_{n}\sim\psi[f[x/a]])\\
		&\text{ iff } (\exists a\in U^{\mathtt{M}^i_\mathbf{w}})(\mathcal{S}^i,\mathbf{w}\models_iTr(\sim\psi)[f[x/a]])&&\text{by IH}\\
		&\text{ iff } \mathcal{S}^i, \mathbf{w}\models_i \exists xTr(\sim\psi)[f]\text{ iff } \mathcal{S}^i, \mathbf{w}\models_i Tr(\phi)[f].
	\end{align*}
	The remaining cases are trivial, or similar to the cases considered above, or both.
\end{proof}
\begin{lemma}\label{L:sheaf-embedding2}
	Let $\mathcal{S}\in \mathbb{I}$ and let $\mathcal{S}^n = (W^n, \leq^n, \mathtt{M}^{n+}, \mathtt{M}^{n-}, \mathtt{H}^n)$ be such that $W^n := W$, $\leq^n:= \leq$, $\mathtt{H}^n:= \mathtt{H}$, and that, for every $\mathbf{w} \in W^n = W$, we have $U^n_\mathbf{w} = U^{\mathtt{M}^{n+}_\mathbf{w}} = U^{\mathtt{M}^{n-}_\mathbf{w}}:= U^{\mathtt{M}_\mathbf{w}}$. Finally, for all $\mathbf{w} \in W$ and $P^m \in \Pi$, we set that
	
	\begin{align*}
		P^{\mathtt{M}^{n+}_\mathbf{w}}&:= P^{\mathtt{M}_\mathbf{w}}_+;&&P^{\mathtt{M}^{n-}_\mathbf{w}}:= P^{\mathtt{M}_\mathbf{w}}_- ;&&\epsilon^{\mathtt{M}^{n-}_\mathbf{w}}:= \epsilon^{\mathtt{M}_\mathbf{w}}.
	\end{align*}	
	Then all of the following statements hold:
	\begin{enumerate}
		\item $\mathcal{S}^n \in \mathbb{N}4$.
		
		\item For every $\mathbf{w} \in W$ and every function $f$, we have $(\mathcal{S}, \mathbf{w}, f) \in EP_{i}$ iff $(\mathcal{S}^n, \mathbf{w}, f) \in EP_{n}$.
		
		\item For every $\mathbf{w} \in W$, every $f$ such that $(\mathcal{S}, \mathbf{w}, f) \in EP_{n}$, and every $\phi \in \mathcal{FO}$, we have $\mathcal{S}^n, \mathbf{w}\models^+_n \phi[f]$ iff $\mathcal{S}, \mathbf{w}\models_i Tr(\phi)[f]$.
	\end{enumerate}	
\end{lemma}
\begin{proof}
	Again, the first two parts are trivial and Part 3 is proved by an induction on the construction of $\phi \in \mathcal{FO}$. We illustrate the idea using the same selection of cases as in the previous proof:
	
	\textit{Case 1}. $\phi  =  P(\bar{x}_m)$ for some $m \in \omega$, $P^m \in \Pi$ and $\bar{x}_m\in Ind^m$. But then:
	\begin{align*}
		\mathcal{S}^n, \mathbf{w}\models^+_n \phi[f] \text{ iff } f(\bar{x}_m)\in 	P^{\mathtt{M}^{n+}_\mathbf{w}} \text{ iff } f(\bar{x}_m)\in P_+^{\mathtt{M}_\mathbf{w}}
		\text{ iff } \mathcal{S}, \mathbf{w}\models_i P_+(\bar{x}_m)[f]\text{ iff } \mathcal{S}, \mathbf{w}\models_i Tr(\phi)[f].
	\end{align*}
	On the other hand, we have:
	\begin{align*}
		\mathcal{S}^n, \mathbf{w}\models^+_n \sim\phi[f] \text{ iff } \mathcal{S}^n, \mathbf{w}\models^-_n \phi[f] &\text{ iff } f(\bar{x}_m)\in 	P^{\mathtt{M}^{n-}_\mathbf{w}} \text{ iff } f(\bar{x}_m)\in P_-^{\mathtt{M}_\mathbf{w}}\\
		&\text{ iff } \mathcal{S}, \mathbf{w}\models_i P_-(\bar{x}_m)[f]\text{ iff } \mathcal{S}, \mathbf{w}\models_i Tr(\sim\phi)[f].
	\end{align*}
	
	\textit{Case 2}. $\phi = (\psi \to \chi)$. Then we have:
	\begin{align*}
		\mathcal{S}^n, \mathbf{w}\models^+_n \phi[f] &\text{ iff }(\forall \mathbf{v} \geq^n \mathbf{w} )(\mathcal{S}^n,\mathbf{v}\not\models^+_{n}\psi[f\circ\mathtt{H}^n_{\mathbf{w}\mathbf{v}}]\text{ or }\mathcal{S}^n, \mathbf{v}\models^+_{n}\chi[f\circ\mathtt{H}^n_{\mathbf{w}\mathbf{v}}])\\
		&\text{ iff }(\forall \mathbf{v} \geq \mathbf{w})(\mathcal{S},\mathbf{v}\not\models_iTr(\psi)[f\circ\mathtt{H}_{\mathbf{w}\mathbf{v}}]\text{ or }\mathcal{S}, \mathbf{v}\models_iTr(\chi)[f\circ\mathtt{H}_{\mathbf{w}\mathbf{v}}])&&\text{by IH}\\
		&\text{ iff } \mathcal{S}, \mathbf{w}\models_i (Tr(\psi)\to Tr(\chi))[f] \text{ iff } \mathcal{S}, \mathbf{w}\models_i Tr(\phi)[f].
	\end{align*}
	On the other hand, we have:
	\begin{align*}
		\mathcal{S}^n, \mathbf{w}\models^+_n \sim\phi[f] &\text{ iff } \mathcal{S}^n,\mathbf{w}\models^+_{n}\psi[f]\text{ and }\mathcal{S}^n, \mathbf{w}\models^-_{n}\chi[f]\\
		&\text{ iff } \mathcal{S}^n,\mathbf{w}\models^+_{n}\psi[f]\text{ and }\mathcal{S}^n, \mathbf{w}\models^+_{n}\sim\chi[f]\\
		&\text{ iff } \mathcal{S},\mathbf{w}\models_iTr(\psi)[f]\text{ and }\mathcal{S}, \mathbf{w}\models_iTr(\sim\chi)[f]&&\text{by IH}\\
		&\text{ iff } \mathcal{S}, \mathbf{w}\models_i Tr(\psi)\wedge Tr(\sim\chi)\text{ iff } \mathcal{S}, \mathbf{w}\models_i Tr(\phi)[f].
	\end{align*}
	
	\textit{Case 3}. $\phi = \forall x\psi$. Then we have:
	\begin{align*}
		\mathcal{S}^n, \mathbf{w}\models^+_n \phi[f] &\text{ iff } (\forall \mathbf{v} \geq^n \mathbf{w})(\forall a\in U^n_{\mathbf{v}})(\mathcal{S}^n,\mathbf{v}\models^+_{n}\psi[(f\circ\mathtt{H}^n_{\mathbf{w}\mathbf{v}})[x/a]])\\
		&\text{ iff } (\forall \mathbf{v} \geq \mathbf{w})(\forall a\in U^{\mathtt{M}_\mathbf{v}})(\mathcal{S},\mathbf{v}\models_iTr(\psi)[(f\circ\mathtt{H}_{\mathbf{w}\mathbf{v}})[x/a]])&&\text{by IH}\\
		&\text{ iff } \mathcal{S}, \mathbf{w}\models_i \forall xTr(\psi)[f]\text{ iff } \mathcal{S}, \mathbf{w}\models_i Tr(\phi)[f].
	\end{align*}
	On the other hand, we have: 
	\begin{align*}
		\mathcal{S}^n, \mathbf{w}\models^+_n \sim\phi[f] &\text{ iff } (\exists a\in U^n_{\mathbf{w}})(\mathcal{S}^n,\mathbf{w}\models^-_{n}\psi[f[x/a]])\\
		&\text{ iff } (\exists a\in U^n_{\mathbf{w}})(\mathcal{S}^n,\mathbf{w}\models^+_{n}\sim\psi[f[x/a]])\\
		&\text{ iff } (\exists a\in U^{\mathtt{M}_\mathbf{w}})(\mathcal{S},\mathbf{w}\models_iTr(\sim\psi)[f[x/a]])&&\text{by IH}\\
		&\text{ iff } \mathcal{S}, \mathbf{w}\models_i \exists xTr(\sim\psi)[f]\text{ iff } \mathcal{S}, \mathbf{w}\models_i Tr(\phi)[f].
	\end{align*}
\end{proof}
We can now prove our proposition:
\begin{proof}[Proof of  Proposition \ref{P:nelsonian-sheaves}]
	For all $\Gamma, \Delta \subseteq \mathcal{FO}$ it is true that:
	\begin{align*}
		\Gamma\models_{\mathsf{QN4}}\Delta &\text{ iff } Tr(\Gamma)\models_{\mathsf{QIL}^+}Tr(\Delta) &&\text{by Proposition \ref{P:intuitionistic-embedding}}\\
		&\text{ iff } Tr(\Gamma)\models_{i}Tr(\Delta) &&\text{by definition of }\mathsf{QIL}^+\\
		&\text{ iff } \Gamma\models^+_{n}\Delta &&\text{by Lemmas \ref{L:sheaf-embedding1} and \ref{L:sheaf-embedding2}}
	\end{align*}
\end{proof}

\subsection{Proof of Lemma \ref{L:n4-standard}}\label{App:n4-standard}

	(Part 1) If $\mathbf{w}\leq\mathbf{v}$ and $\mathcal{S},\mathbf{w}\models^+_{n}\phi[f]$, then, by Lemma \ref{L:sheaf-embedding1}, we must have $\mathcal{S}^i,\mathbf{w}\models^+_{i}Tr(\phi)[f]$, whence $\mathcal{S}^i,\mathbf{v}\models^+_{i}Tr(\phi)[f\circ\mathtt{H}^ i_{\mathbf{w}\mathbf{v}}]$ by Lemma \ref{L:intuitionistic-standard}.1 and $\mathcal{S},\mathbf{v}\models^+_{n}\phi[f\circ\mathtt{H}_{\mathbf{w}\mathbf{v}}]$, again by Lemma \ref{L:sheaf-embedding1}. In case $\mathcal{S},\mathbf{w}\models^-_{n}\phi[f]$, we must have  $\mathcal{S},\mathbf{w}\models^+_{n}\sim\phi[f]$, whence $\mathcal{S},\mathbf{v}\models^+_{n}\sim\phi[f\circ\mathtt{H}_{\mathbf{w}\mathbf{v}}]$ which is the same as $\mathcal{S},\mathbf{v}\models^-_{n}\phi[f\circ\mathtt{H}_{\mathbf{w}\mathbf{v}}]$.
	
	(Part 2) We note that we clearly have $(\mathcal{S}|_\mathbf{w})^i = (\mathcal{S}^i)|_\mathbf{w}$. Thus, by Lemma \ref{L:sheaf-embedding1},  we have $\mathcal{S}|_\mathbf{w}, \mathbf{v}\models^+_{n}\phi[f]$ iff $(\mathcal{S}|_\mathbf{w})^i, \mathbf{v}\models_{i}Tr(\phi)[f]$ iff $(\mathcal{S}^i)|_\mathbf{w}, \mathbf{v}\models_{i}Tr(\phi)[f]$ iff, by Lemma \ref{L:intuitionistic-standard}.2, we have $\mathcal{S}^i, \mathbf{v}\models_{i}Tr(\phi)[f]$, iff, again by  Lemma \ref{L:sheaf-embedding1}, $\mathcal{S}, \mathbf{v}\models^+_{n}\phi[f]$.
	
	On the other hand, we have $\mathcal{S}|_\mathbf{w}, \mathbf{v}\models^-_{n}\phi[f]$ iff $\mathcal{S}|_\mathbf{w}, \mathbf{v}\models^+_{n}\sim\phi[f]$ iff $\mathcal{S}, \mathbf{v}\models^+_{n}\sim\phi[f]$ iff $\mathcal{S}, \mathbf{v}\models^-_{n}\phi[f]$.

\subsection{Proof of Lemma \ref{L:derived-connectives}}\label{App:derived-connectives}
	The claims of Lemma \ref{L:derived-connectives} can be easily verified on the basis of the Nelsonian sheaf semantics. Also, it is pretty straightforward to obtain zhe proofs of \eqref{E:T1}--\eqref{E:T16} in $\mathfrak{QN4}$. As an example, we sketch the following proofs:
	
	\eqref{E:T13}. We consider the case when $\ast = \to$. We start by building the following chains of $\mathsf{QN4}$-valid implications:
	\begin{align*}
		((\phi \Leftrightarrow \psi)  \wedge (\chi \Leftrightarrow \theta))&\to ((\phi \leftrightarrow \psi)  \wedge (\chi \leftrightarrow \theta))&&\text{ by Lemma \ref{L:intuitionistic-inclusion}, def. of }\Leftrightarrow\\
		&\to (\phi \to \chi)\leftrightarrow (\psi \to \theta)&&\text{ by Lemma \ref{L:intuitionistic-inclusion}}
	\end{align*}
	and
	\begin{align*}
		((\phi \Leftrightarrow \psi)  \wedge (\chi \Leftrightarrow \theta))&\to ((\phi \leftrightarrow \psi)  \wedge (\sim\chi \leftrightarrow \sim\theta))&&\text{ by Lemma \ref{L:intuitionistic-inclusion}, def. of }\Leftrightarrow\\
		&\to (\phi \wedge \sim\chi)\leftrightarrow (\psi \wedge \sim\theta)&&\text{ by Lemma \ref{L:intuitionistic-inclusion}}\\
		&\to \sim(\phi \to \chi)\leftrightarrow \sim(\psi \to \theta) &&\text{ by Lemma \ref{L:intuitionistic-inclusion}, }\eqref{E:a4}
	\end{align*}
	We thus have shown that 
	$$
	((\phi \Leftrightarrow \psi)  \wedge (\chi \Leftrightarrow \theta)) \to ((\phi \to \chi)\leftrightarrow (\psi \to \theta)\wedge \sim(\phi \to \chi)\leftrightarrow \sim(\psi \to \theta)),
	$$
	whence \eqref{E:T13} for $\ast = \to$ follows by the definition of $\Leftrightarrow$.
	
	\eqref{E:T14}. We consider the case when $Q = \exists$. Again, we build two valid implication chains (the commentary applies to both chains simultaneously):
	\begin{align*}
		(\phi \Leftrightarrow \psi) &\to (\phi \leftrightarrow \psi)&& (\phi \Leftrightarrow \psi)\to (\sim\phi \leftrightarrow \sim\psi)&&\text{ by Lemma \ref{L:intuitionistic-inclusion}, def. of }\Leftrightarrow\\
		&\to (\exists x\phi\leftrightarrow \exists x\psi)&&\qquad\qquad\to (\forall x\sim\phi\leftrightarrow\forall x\sim\psi)&&\text{ by Lemma \ref{L:intuitionistic-inclusion}}\\
		& &&\qquad\qquad\to(\sim\exists x\phi\leftrightarrow\sim\exists x\psi)&&\text{ by Lemma \ref{L:intuitionistic-inclusion}, }\eqref{E:a5}
	\end{align*}
	Combining the two implication chains yields that 
	$$
	(\phi \Leftrightarrow \psi)\to ((\exists x\phi\leftrightarrow \exists x\psi)\wedge(\sim\exists x\phi\leftrightarrow\sim\exists x\psi)),
	$$
	whence \eqref{E:T14} for $Q = \exists$ follows by the definition of $\Leftrightarrow$.
	
	\subsection{The use of ampersand in $\mathsf{QN4}$}\label{App:ampersand}
In this appendix, we quickly motivate the usefulness of ampersand. Ampersand is especially convenient for handling the restricted existential quantification in the context of $\mathsf{QN4}$. The latter statement can be easily motivated as long as one is prepared to assume that motivation the states of any given Nelsonian sheaf $\mathcal{S}$ can be seen as knowledge states so that an object $a \in U_\mathbf{w}$ is understood as an object known to exist at $\mathbf{w}$; the states that are accessible from $\mathbf{w}$ are then to be understood as future knowledge states that are possible in $\mathbf{w}$. 

Indeed, when we say that some object has some property, say $p_0$, this is true iff at least one currently known object verifies $p_0$; this is false iff every object that is either currently known or will become known in the future, falsifies $p_0$. Let us now try to formalize this idea: if our current state of knowledge, located in the web of other possible knowledge states is faithfully captured by some world $\mathbf{w}$ in some Nelsonian sheaf $\mathcal{S}$, then we will say that our existence claim is true at $(\mathcal{S}, \mathbf{w})$  iff $U_\mathbf{w} \cap p^{\mathtt{M}_\mathbf{w}^+}_0\neq \emptyset$; and false at $(\mathcal{S}, \mathbf{w})$ iff, for every $\mathbf{v} \geq \mathbf{w}$, we have $U_\mathbf{v} \subseteq p^{\mathtt{M}_\mathbf{v}^-}_0$. It is easy to see now that both the truth and the falsity condition of our existence claim are exactly those of the sentence $\exists xp_0(x)$, which is also how one usually formalizes simple existence claims classically.

Suppose now that we would like to restrict our claim about existence of an object with the property $p_0$ to the family of objects that verify some additional property $p_1$. At least one natural way to understand such a restriction is to say that our claim is true iff at least one object that is currently known to have $p_1$, also verifies $p_0$. As for the falsity condition, our claim must be false iff every object that is either currently known or will be known in the future to have $p_1$ also falsifies $p_0$. Turning again to possible formalizations, we are looking for a formula that must be true at $(\mathcal{S}, \mathbf{w})$  iff $(U_\mathbf{w}\cap p^{\mathtt{M}_\mathbf{w}^+}_1) \cap p^{\mathtt{M}_\mathbf{w}^+}_0\neq \emptyset$ and false at $(\mathcal{S}, \mathbf{w})$ iff, for every $\mathbf{v} \geq \mathbf{w}$, we have $(U_\mathbf{v}\cap  p^{\mathtt{M}_\mathbf{v}^+}_1)\subseteq p^{\mathtt{M}_\mathbf{v}^-}_0$. 

Now, the classical formalization of this sort of restricted existential quantification is given by $\exists x(p_1(x)\wedge p_0(x))$, and it is easy to see that in the context of $\mathsf{QN4}$ this sentence has exactly the truth condition that we have given above. However, its falsity condition is different, namely, the sentence is false at $(\mathcal{S}, \mathbf{w})$ iff, for every $\mathbf{v} \geq \mathbf{w}$, we have $U_\mathbf{v} \subseteq (p^{\mathtt{M}_\mathbf{v}^-}_1\cup  p^{\mathtt{M}_\mathbf{v}^-}_0)$. Note that in the semantics of Nelsonian sheaves, $p^{\mathtt{M}_\mathbf{v}^-}_1$ and $p^{\mathtt{M}_\mathbf{v}^+}_1$ are, generally speaking, completely independent from one another and, therefore, the latter condition is also independent from the requirement that $(U_\mathbf{v}\cap  p^{\mathtt{M}_\mathbf{v}^+}_1)\subseteq p^{\mathtt{M}_\mathbf{v}^-}_0$ for every  $\mathbf{v} \geq \mathbf{w}$.

We are doing much better, though, if in the above formalization we replace conjunction with ampersand. The truth condition then remains the same, since Lemma \ref{L:intuitionistic-inclusion} and \eqref{E:T13} together imply that $
\exists x(p_1(x)\,\&\, p_0(x))\leftrightarrow \exists x(p_1(x)\wedge p_0(x))$.

However, the falsity condition looks different, since, by Lemma \ref{L:intuitionistic-inclusion}, \eqref{E:T2}, \eqref{E:T11}, and \eqref{E:T14}, it follows that $
\sim\exists x(p_1(x)\,\&\, p_0(x))\leftrightarrow \forall x(p_1(x)\to \sim p_0(x))$.

In other words, our alternative formalization is false at $(\mathcal{S}, \mathbf{w})$  iff, for every $\mathbf{v} \geq \mathbf{w}$ we have $(U_\mathbf{v}\cap  p^{\mathtt{M}_\mathbf{v}^+}_1)\subseteq p^{\mathtt{M}_\mathbf{v}^-}_0$, which is exactly what we wanted initially.

Note that ampersand becomes necessary for our representations of restricted existential quantifications only insofar as we get interested in what happens when a statement is false. This is not always the case though. For example, if one is formulating a theory in $\mathsf{QN4}$ and would like to ensure the existence of an object with the property $p_0$ among those objects that happen to satisfy $p_1$, then it might make perfect sense to formalize this part of one's theory by requiring that  $\exists x(p_1(x)\wedge p_0(x))$ as long as one is only interested in the models where this claim is true and is indifferent to the models where the claim happens to be false.

\subsection{A proof of Lemma \ref{L:substitution}}\label{App:substitution}
	We proceed by induction on the construction of $\chi \in \mathcal{FO}$.
	
	\textit{Basis}. If $\chi \in At$, then the following cases are possible:
	
	\textit{Case 1}. $\chi = (v_0 \equiv v_0)$. Then \eqref{E:T17} assumes the form $
	(\phi\Leftrightarrow\psi)\to(\phi\Leftrightarrow\psi)$,
	and is a theorem of $\mathsf{QN4}$ by Lemma \ref{L:intuitionistic-inclusion}.
	
	\textit{Case 2}. $\chi \neq (v_0 \equiv v_0)$. Then \eqref{E:T17} is just
	$(\phi\Leftrightarrow\psi)\to(\chi\Leftrightarrow\chi)$,
	and is a theorem of $\mathsf{QN4}$ by \eqref{E:T5} and Lemma \ref{L:intuitionistic-inclusion}.
	
	\textit{Induction step}. The following cases are possible:
	
	\textit{Case 1}. $\chi = (\chi_0\ast\chi_1)$ for some $\ast\in \{\wedge, \vee, \to\}$. Then we reason as follows:
	\begin{align}
		\vdash_{\mathsf{QN4}}(\phi\Leftrightarrow\psi)&\to((\chi_0/\phi)\Leftrightarrow(\chi_0/\psi))\label{E:sb1}&&\text{by IH}\\
		\vdash_{\mathsf{QN4}}(\phi\Leftrightarrow\psi)&\to((\chi_1/\phi)\Leftrightarrow(\chi_1/\psi))\label{E:sb2}&&\text{by IH}\\
		\vdash_{\mathsf{QN4}}(\phi\Leftrightarrow\psi)&\to(((\chi_0/\phi)\Leftrightarrow(\chi_0/\psi))\wedge((\chi_1/\phi)\Leftrightarrow(\chi_1/\psi)))\label{E:sb3}&&\text{by \eqref{E:sb1}, \eqref{E:sb2}, Lm \ref{L:intuitionistic-inclusion}}\\
		\vdash_{\mathsf{QN4}}(((\chi_0/\phi)&\Leftrightarrow(\chi_0/\psi))\wedge((\chi_1/\phi)\Leftrightarrow(\chi_1/\psi)))\to\notag\\
		&\qquad\qquad\to(((\chi_0\ast\chi_1)/\phi)\Leftrightarrow((\chi_0\ast\chi_1)/\psi))\label{E:sb4}&&\text{by \eqref{E:T13}}\\
		\vdash_{\mathsf{QN4}}(\phi\Leftrightarrow\psi)&\to((\chi/\phi)\Leftrightarrow(\chi/\psi))\label{E:sb5}&&\text{by \eqref{E:sb3},\eqref{E:sb4},Lm \ref{L:intuitionistic-inclusion}}		
	\end{align}
	
	\textit{Case 2}.  $\chi = \sim\chi_0$ or $\chi = Qx\chi_0$ for some $Q\in \{\forall, \exists\}$ and some $x \in Ind$. We reason as in Case 1 applying \eqref{E:T4} (resp. \eqref{E:T14}) in place of \eqref{E:T13}.
\section{Proof of Lemma \ref{L:n4-comprehension}}\label{App:n4-comprehension}
	We proceed by induction on the construction of $\phi \in \mathcal{CN}$. The basis of induction follows from \eqref{E:th6}. As for the induction step, we need to consider the following cases:
	
	\textit{Case 1}. $\phi = \psi\ast\chi$, where $\ast \in \{\wedge, \to\}$. We consider the following derivation D1 from premises in $\mathsf{QN4}$:
	\begin{align}
		&Sx\wedge(\forall w)_O(Ewx\Leftrightarrow ST_w(\psi))\label{E:fo1}&&\text{premise}\\
		&Sy\wedge(\forall w)_O(Ewy\Leftrightarrow ST_w(\chi))\label{E:fo2}&&\text{premise}\\
		&Sz\wedge(\forall w)_O(Ewz\Leftrightarrow (Ewx\ast Ewy))\label{E:fo3}&&\text{premise}\\
		&(\forall w)_O(Ewz\Leftrightarrow(ST_w(\psi)\ast Ewy))\label{E:fo4}&&\text{by \eqref{E:fo1}, \eqref{E:fo3}, \eqref{E:T18}}\\
		&(\forall w)_O(Ewz\Leftrightarrow(ST_w(\psi)\ast ST_w(\chi)))\label{E:fo5}&&\text{by \eqref{E:fo2}, \eqref{E:fo4}, \eqref{E:T18}}\\
		&\exists z(Sz\wedge (\forall w)_O(Ewz\Leftrightarrow ST_w(\psi\ast\chi)))\label{E:fo7}&&\text{by \eqref{E:fo5}, def. of $ST$}
	\end{align}
	We now reason as follows:
	\begin{align*}
		&Th, \eqref{E:fo1}, \eqref{E:fo2} \models_{n} \exists z\eqref{E:fo3}\to \exists z(Sz\wedge (\forall w)_O(Ewz\Leftrightarrow ST_w(\psi\ast\chi)))&&\text{(D1, \eqref{R:DT}, \eqref{R:E})}\\
		&Th, \eqref{E:fo1}, \eqref{E:fo2} \models_{n} Sx\wedge Sy &&\text{(trivially)}\\
		&Th,  \eqref{E:fo1}, \eqref{E:fo2} \models_{n} (Sx\wedge Sy )\to \exists z\eqref{E:fo3}&&\text{\eqref{E:th8}}\\
		&Th, \eqref{E:fo1}, \eqref{E:fo2} \models_{n}\exists z(Sz\wedge (\forall w)_O(Ewz\Leftrightarrow ST_w(\psi\ast\chi)))&&\text{\eqref{E:mp}}\\
		&Th\models_{n} \exists x\eqref{E:fo1} \to (\exists z\eqref{E:fo2}\to \exists z(Sz\wedge (\forall w)_O(Ewz\Leftrightarrow ST_w(\psi\ast\chi))))&&\text{\eqref{R:DT}, \eqref{R:E}}\\
		&Th\models_{n} \exists x\eqref{E:fo1} \wedge \exists y\eqref{E:fo2}&&\text{(IH)}\\
		&Th\models_{n}\exists z(Sz\wedge (\forall w)_O(Ewz\Leftrightarrow ST_w(\psi\ast\chi)))&&\text{\eqref{E:mp}} 
	\end{align*}
	
	\textit{Case 2}. $\phi = \sim\psi$. Similar to Case 1, but using \eqref{E:th7} in place of \eqref{E:th8}.
	
	\textit{Case 3}. $\phi = \psi\vee\chi$. By Cases 1 and 2 and Corollary \ref{C:disjunstion}.
	
	\textit{Case 4}. $\phi = (\psi\boxto\chi)$. We consider the following deductions from premises in $\mathsf{QN4}$, letting $T:= Th \cup \{Sy'\wedge(\forall w)_O(Ewy'\Leftrightarrow ST_w(\psi))\} = Th \cup \{\theta\}$.
	
	Deduction D2+:
	\begin{align}
		&\forall w(Rxy'w\to ST_w(\chi))\label{E:ffo2}&&\text{premise}\\
		&\exists y((Sy\wedge(\forall w)_O(Ewy\Leftrightarrow ST_w(\psi)))\,\&\,\notag\\
		&\qquad\qquad\qquad\,\&\,\forall w(Rxyw\to ST_w(\chi)))\label{E:ffo4}&&\text{by $\theta$,  \eqref{E:ffo2}, \eqref{Ax:10}, \eqref{E:T15}}\\
		&ST_x(\psi\boxto\chi)\label{E:ffo5}&&\text{by \eqref{E:ffo4}, def. of $ST$}
	\end{align}
	
	Deduction D2--:
	\begin{align}
		&\sim ST_x(\psi\boxto\chi)\label{E:ffo1-}&&\text{premise}\\
		&\forall y((Sy\wedge(\forall w)_O(Ewy\Leftrightarrow ST_w(\psi)))\to\notag\\
		&\qquad\qquad\qquad\to\sim\forall w(Rxyw\to ST_w(\chi)))\label{E:ffo2-}&&\text{by \eqref{E:ffo1-},  \eqref{E:a5}, \eqref{E:T16}}\\
		&\theta\to\forall w(Rxy'w\to ST_w(\chi))\label{E:ffo3-}&&\text{by \eqref{E:ffo2-}, \eqref{Ax:9}}\\	
		&\sim\forall w(Rxy'w\to ST_w(\chi))\label{E:ffo3-}&&\text{by $\theta$,  \eqref{E:ffo2-}}
	\end{align}
	
	Deduction D3+:
	\begin{align}
		&Sy\wedge(\forall w)_O(Ewy\Leftrightarrow ST_w(\psi))\label{E:ffo7}&&\text{premise}\\	
		&\forall w(Rxyw\to ST_w(\chi))\label{E:ffo8}&&\text{premise}\\
		&(\forall w)_O(Ewy\Leftrightarrow Ewy')\label{E:ffo8a}&&\text{by $\theta$, \eqref{E:ffo7}, \eqref{E:T18}}\\
		&y \equiv y'\label{E:ffo9}&&\text{by \eqref{E:ffo8a}, \eqref{E:th5}}\\
		&\forall w(Rxy'w\to ST_w(\chi))&&\text{by \eqref{E:ffo8}, \eqref{E:ffo9}, \eqref{Ax:12}}
	\end{align} 
	
	Deduction D3--:
	\begin{align}
		&\eqref{E:ffo3-}, \eqref{E:ffo7}&&\text{premises}\notag\\	
		&y \equiv y'\label{E:ffo9-}&&\text{as in D3+}\\
		&\sim\forall w(Rxyw\to ST_w(\chi))&&\text{by \eqref{E:ffo3-}, \eqref{E:ffo9-}, \eqref{Ax:12}}
	\end{align} 
	We sum up the intermediate results of these deductions:
	\begin{align}
		&T\models_{n}\forall w(Rxy'w\to ST_w(\chi))\to ST_x(\psi\boxto\chi)\label{fo1}\,\text{D2+, \eqref{R:DT}}\\
		&T\models_{n}(\eqref{E:ffo7}\wedge\eqref{E:ffo8})\to\forall w(Rxy'w\to ST_w(\chi))\label{fo1a}\,\text{D3+, \eqref{R:DT}}\\
		&T\models_{n}(\eqref{E:ffo7}\,\&\,\eqref{E:ffo8})\to\forall w(Rxy'w\to ST_w(\chi))\label{fo1b}\,\text{\eqref{fo1a}, \eqref{E:T15}}\\
		&T\models_{n}\exists y(\eqref{E:ffo7}\,\&\,\eqref{E:ffo8})\to\forall w(Rxy'w\to ST_w(\chi))\label{fo1c}\,\text{\eqref{fo1b}, \eqref{R:E}}\\
		&T\models_{n}ST_x(\psi\boxto\chi)\to\forall w(Rxy'w\to ST_w(\chi))\label{fo2}\,\text{\eqref{fo1c}, def. $ST$}\\
		&T\models_{n}\forall x(ST_x(\psi\boxto\chi)\leftrightarrow\forall w(Rxy'w\to ST_w(\chi)))\label{fo3}\,\text{\eqref{fo1},\eqref{fo2}, \eqref{R:Gen}}\\
		&T\models_{n}\sim ST_x(\psi\boxto\chi)\to\sim\forall w(Rxy'w\to ST_w(\chi))\label{fo1-}\,\text{D2--, \eqref{R:DT}}\\
		&T\models_{n}\eqref{E:ffo3-}\to(\eqref{E:ffo7}\to \sim\forall w(Rxyw\to ST_w(\chi)))\label{fo1a-}\,\text{D3--, \eqref{R:DT}}\\
		&T\models_{n}\eqref{E:ffo3-}\to\forall y(\eqref{E:ffo7}\to \sim\forall w(Rxyw\to ST_w(\chi)))\label{fo1b-}\,\text{\eqref{fo1a-}, \eqref{R:A}}\\
		&T\models_n\sim\forall w(Rxy'w\to ST_w(\chi))\to\sim ST_x(\psi\boxto\chi)\label{fo1c-}\,\text{\eqref{fo1b-}, \eqref{E:a5}, \eqref{E:T16}}\\
		&T\models_n\forall x(\sim ST_x(\psi\boxto\chi)\leftrightarrow\sim\forall w(Rxy'w\to ST_w(\chi)))\label{fo2-}\,\text{\eqref{fo1-},\eqref{fo1c-}, \eqref{R:Gen}}\\
		&T\models_n\forall x( ST_x(\psi\boxto\chi)\Leftrightarrow\forall w(Rxy'w\to ST_w(\chi)))\label{fo3-}\,\text{\eqref{fo3},\eqref{fo2-}}		
	\end{align}	
	We now feed these results into the next deduction D4:
	\begin{align}
		&Sz\wedge(\forall w)_O(Ewz\Leftrightarrow ST_w(\chi))\label{E:ffo13}&&\text{premise}\\	
		&Sy'\wedge Sz\label{E:ffo14}&&\text{by $\theta$, \eqref{E:ffo13}}\\
		&\exists z'(Sz'\wedge(\forall x)_O(Exz'\Leftrightarrow \forall w(Rxy'w\to Ewz)))\label{E:ffo15}&&\text{by \eqref{E:ffo14}, \eqref{E:th9}}\\
		&\exists z'(Sz'\wedge(\forall x)_O(Exz'\Leftrightarrow \forall w(Rxy'w\to ST_w(\chi))))\label{E:ffo16}&&\text{by \eqref{E:ffo13}, \eqref{E:ffo15}, \eqref{E:T18}}\\
		&\exists z'(Sz'\wedge(\forall x)_O(Exz'\Leftrightarrow ST_x(\psi\boxto\chi)))\label{E:ffo17}&&\text{by \eqref{E:ffo16}, \eqref{fo3-}, \eqref{E:T18}}
	\end{align}
	We now finish our reasoning as follows:
	\begin{align*}
		Th&\models_{n}\exists y'\theta \to (\exists z\eqref{E:ffo13} \to \exists z'(Sz' \wedge(\forall x)_O(Exz'\Leftrightarrow ST_x(\psi\boxto\chi))))&&\text{(D4, \eqref{R:DT}, \eqref{R:E})}\\
		Th&\models_{n} \exists y'\theta \wedge \exists z\eqref{E:ffo13}&&\text{(IH)}\\
		Th&\models_{n}\exists z'(Sz'\wedge (\forall x)_O(Exz'\Leftrightarrow ST_x(\psi\boxto\chi)))&&\eqref{E:mp}
	\end{align*}

\section{Proof of Lemma \ref{L:global-functions}}\label{App:seq}
In order to prove the Lemma, we will need to prove a couple of technical results first.
\begin{lemma}\label{L:standard-sequence}
	The following statements hold:
	\begin{enumerate}
		\item If $\alpha, \beta \in Seq$ are such that $\alpha\mathrel{\prec}\beta$, and $\phi \in \mathcal{CN}$, $(\Gamma, \Delta) \in W_c$ are such that $\alpha' = \alpha^\frown(\phi, (\Gamma, \Delta)) \in Seq$, then there exists a $\beta' = \beta^\frown(\phi, (\Xi, \Theta)) \in Seq$ and $\alpha'\mathrel{\prec}\beta'$. 
		
		\item For all $(\Gamma,\Delta), (\Xi,\Theta) \in W_c$ such that $(\Gamma,\Delta)\leq_c (\Xi,\Theta)$, and every $\alpha \in Seq(\Gamma,\Delta)$, there exists a $\beta \in Seq(\Xi, \Theta)$ such that $\alpha\mathrel{\prec}\beta$.
		\item For all $(\Gamma,\Delta), (\Xi,\Theta) \in W_c$ and $\alpha \in Seq(\Gamma,\Delta)$ such that $end(\alpha)\leq_c(\Xi,\Theta)$, there exists a $\beta \in Seq$ such that $\alpha\mathrel{\prec}\beta$ and $end(\beta) = (\Xi,\Theta)$. 
	\end{enumerate} 	
\end{lemma}
\begin{proof}
	(Part 1) Assume the hypothesis. Since $\alpha\mathrel{\prec}\beta$, we must have $end(\alpha)\leq_c end(\beta)$, and, since $\alpha' = \alpha^\frown(\phi, (\Gamma, \Delta)) \in Seq$, we must have $R_c(end(\alpha),\|\phi\|_{\mathcal{M}_c}, (\Gamma, \Delta))$. Since $\mathcal{M}_c$ satisfies condition \eqref{Cond:1} of Definition \ref{D:model}, there must exist some $(\Xi, \Theta)\in W_c$ such that both $(\Gamma, \Delta)\mathrel{\leq_c}(\Xi, \Theta)$ and $R_c(end(\beta),\|\phi\|_{\mathcal{M}_c}, (\Xi, \Theta))$. The latter means that we have both $\beta' = \beta^\frown(\phi, (\Xi, \Theta)) \in Seq$ and $\alpha'\mathrel{\prec}\beta'$.
	
	(Part 2) By induction on the length of $\alpha$. If $\alpha$ has length $1$, then we can set $\beta:= (\Xi,\Theta)$. If $\alpha$ has length $k + 1$ for some $1 \leq k < \omega$, then we apply IH and Part 1.
	
	(Part 3) Again, we proceed by induction on the length of $\alpha$. If $\alpha$ has length $1$, then we can set $\beta:= (\Xi,\Theta)$. If $\alpha$ has length $k + 1$ for some $1 \leq k < \omega$, then for some $\alpha' \in Seq$ of length $k$ and some $\phi \in \mathcal{CN}$, we must have $\alpha = ((\Gamma, \Delta), \phi)^\frown(\alpha')$, with $end(\alpha) = end(\alpha')$. Applying now IH to $\alpha'$, we find some $\beta' \in Seq$ such that $end(\beta') = (\Xi,\Theta)$ and $\alpha'\mathrel{\prec}\beta'$. But then, in particular, we must have $init(\alpha') \leq_c init(\beta')$. Moreover, since $\alpha = ((\Gamma, \Delta), \phi)^\frown(\alpha') \in Seq$, we must also have $R_c((\Gamma, \Delta),\|\phi\|_{\mathcal{M}_c}, init(\alpha'))$. But then, since $\mathcal{M}_c$ satisfies condition \eqref{Cond:2} of Definition \ref{D:model}, there must exist some $(\Gamma', \Delta')\in W_c$ such that both $(\Gamma, \Delta)\mathrel{\leq_c}(\Gamma', \Delta')$ and $R_c((\Gamma', \Delta'),\|\phi\|_{\mathcal{M}_c},init(\beta'))$. For $\beta := ((\Gamma', \Delta'), \phi)^\frown(\beta')$ we have $\beta \in Seq$, $end(\beta) = end(\beta') = (\Xi,\Theta)$, and  $\alpha\mathrel{\prec}\beta$.  
\end{proof}
 The following Lemma sums up some facts about local choice functions:
\begin{lemma}\label{L:choice-functions}
	Let  $(\Gamma, \Delta) \in W_c$. Then the following statements hold:
	\begin{enumerate}
		\item For every $(\Xi, \Theta)\in W_c$ such that $(\Gamma,\Delta)\leq_c (\Xi,\Theta)$, we have $\mathfrak{F}((\Gamma,\Delta),(\Xi,\Theta))\neq \emptyset$.
		
		\item For every $\alpha \in Seq(\Gamma,\Delta)$ and every $(\Xi, \Theta)\in W_c$ such that $end(\alpha) \leq_c (\Xi, \Theta)$, there exist a $\beta \in Seq$ such that $end(\beta) = (\Xi, \Theta)$, and an $f\in\mathfrak{F}((\Gamma,\Delta),init(\beta))$ such that $f(\alpha) = \beta$.
		
		\item For every $(\Xi, \Theta)\in W_c$, $id[Seq(\Xi,\Theta)]\in \mathfrak{F}((\Xi,\Theta),(\Xi,\Theta))$.
		
		\item Given an $n \in \omega$, any $(\Gamma_0,\Delta_0),\ldots,(\Gamma_n, \Delta_n)\in W_c$ such that $(\Gamma_0,\Delta_0)\leq_c\ldots\leq_c(\Gamma_n, \Delta_n)$, and any $f_1,\ldots,f_n$ such that for every $i < n$ we have $f_{i + 1}\in \mathfrak{F}((\Gamma_i, \Delta_i),(\Gamma_{i + 1}, \Delta_{i + 1}))$, we also have that $f_1\circ\ldots\circ f_n \in \mathfrak{F}((\Gamma_0, \Delta_0),(\Gamma_{n}, \Delta_{n}))$.
	\end{enumerate} 
\end{lemma}
\begin{proof}
	(Part 1) By Lemma  \ref{L:standard-sequence}.2 and the Axiom of Choice.
	
	(Part 2) Assume the hypothesis. By Lemma  \ref{L:standard-sequence}.3, we can find a $\beta\in Seq$ such that both $\alpha\mathrel{\prec}\beta$ and $end(\beta) = (\Xi, \Theta)$ are satisfied. Trivially, we must also have $\beta \in Seq(init(\beta))$. Assume, wlog, that $\alpha = ((\Gamma_0,\Delta_0),\phi_1,\ldots,\phi_n, (\Gamma_n, \Delta_n))$ and that $\beta = ((\Xi_0,\Theta_0),\phi_1,\ldots,\phi_n, (\Xi_n, \Theta_n))$. We now set $f((\Gamma_0,\Delta_0),\phi_1,\ldots,\phi_i, (\Gamma_i, \Delta_i)):= ((\Xi_0,\Theta_0),\phi_1,\ldots,\phi_i, (\Xi_i, \Theta_i))$ for every $i \leq n$ and, proceeding by induction on the length of a standard sequence, extend this partial function to other elements of $Seq(\Gamma,\Delta)$ in virtue of Lemma \ref{L:standard-sequence}.1. The resulting function $f$ clearly has the desired properties.
	
	Part 3 is trivial, and an easy induction on $n \in \omega$ also yields us Part 4.
\end{proof}
\begin{proof}[Proof of Lemma \ref{L:global-functions}]
	(Part 1) Note that we have $Seq(\Gamma_0,\Delta_0) \cap Seq(\Gamma_1,\Delta_1) = \emptyset$ whenever $(\Gamma_0,\Delta_0), (\Gamma_1,\Delta_1) \in W_c$ are such that $(\Gamma_0,\Delta_0)\neq (\Gamma_1,\Delta_1)$. Therefore, we can define the global choice function in question by $F:= (\bigcup\{Id[Seq(\Gamma_0,\Delta_0)]\mid (\Gamma_0,\Delta_0)\neq (\Gamma,\Delta)\}) \cup f$.
	Lemma \ref{L:choice-functions}.3 then implies that $F\in\mathfrak{G}$.
	
	(Part 2) Assume the hypothesis. By Lemma  \ref{L:choice-functions}.2, we can choose a $\beta \in Seq$ such that $end(\beta) = (\Xi, \Theta)$, and an $f\in\mathfrak{F}((\Gamma,\Delta),init(\beta))$ such that $f(\alpha) = \beta$. By Part 1, we can find an $F\in\mathfrak{G}$ such that $F\upharpoonright Seq(\Gamma,\Delta) = f$ and thus also $F(\alpha) = f(\alpha) = \beta$.  Parts 3 and 4 are, again, straightforward.
	
	As for Part 5, note that we must have  both $\alpha \in Seq(init(\alpha))$ and, for an appropriate $(\Gamma, \Delta) \in W_c$, that $F\upharpoonright Seq(init(\alpha)) \in \mathfrak{F}(init(\alpha), (\Gamma, \Delta))$. But then we must have 
	\noindent$\alpha\mathrel{\prec}(F\upharpoonright Seq(init(\alpha)))(\alpha) = F(\alpha)$ by definition of a local choice function.
\end{proof}

\end{appendices}

}
\end{document}